\documentclass[11pt]{amsart}
\usepackage{amsmath,amsthm,amssymb,latexsym,graphicx,verbatim,enumerate,%
ifpdf,mdwlist, xspace}
\usepackage{mathabx}
\ifpdf
\usepackage{hyperref}
\else
\usepackage[hypertex]{hyperref}
\fi

\usepackage{amssymb, amsmath, amsthm}
\usepackage{graphicx}
\usepackage{cancel}
\usepackage{float}

\newcommand{\basta}{finitary-diagonalization\xspace}

\providecommand{\I}{\mathcal{I}}

\providecommand{\nrel}[1]{\cancel{\rel{#1}}}

\DeclareMathOperator{\im}{im}
\DeclareMathOperator{\ran}{ran}
\DeclareMathOperator{\Id}{Id}

\newenvironment{q}{\begin{question} \rm}{ \end{question}}
\newtheorem{question}{Question}
\newtheorem{thm}{Theorem}[section]
\newtheorem{definition}[thm]{Definition}
\newtheorem{lem}[thm]{Lemma}
\newtheorem{remark}[thm]{Remark}

\newtheorem{lemma}[thm]{Lemma}

\newtheorem{obs}[thm]{Observation}

\newtheorem{cory}[thm]{Corollary}

\newtheorem{fact}[thm]{Fact}
\newenvironment{defn}{\begin{definition} \rm}{ \end{definition}}

\renewcommand{\phi}{\varphi}

\DeclareMathOperator{\Dark}{\mathbf{Dark}}
\DeclareMathOperator{\Light}{\mathbf{Light}}
\DeclareMathOperator{\range}{range}
\DeclareMathOperator{\Ceers}{\mathbf{Ceers}}

\newcommand{\rel}[1]{\mathrel{#1}}

\title{Joins and Meets in the Structure of Ceers}%

\author[Andrews]{Uri Andrews}
\email{\href{mailto:andrews@math.wisc.edu}{andrews@math.wisc.edu}}
\urladdr{\url{http://www.math.wisc.edu/~andrews/}}

\address{Department of Mathematics\\
University of Wisconsin\\
Madison, WI 53706-1388\\
USA}
\thanks{Andrews was partially supported
by NSF Grant 1600228 and Grant 3952/GF4 of the
Science Committee of the Republic of Kazakhstan}

\author[Sorbi]{Andrea Sorbi}
\email{\href{mailto:andrea.sorbi@unisi.it}{andrea.sorbi@unisi.it}}
\urladdr{\url{http://www3.diism.unisi.it/~sorbi/}}

\address{Department of Information Engineering and Mathematics\\
University of Siena\\
53100 Siena\\
Italy}
\thanks{Sorbi is a member of INDAM-GNSAGA; he was partially supported
by Grant 3952/GF4 of the
Science Committee of the Republic of Kazakhstan, and by PRIN 2012 ``Logica
Modelli e Insiemi ''.
\\
Most of the material of this paper was presented at the \emph{Workshop on
Computability Theory} held in Ghent, July 4-6 2016}

\begin{document}
	
\begin{abstract}
We study computably enumerable equivalence relations (abbreviated as
\emph{ceers}) under computable reducibility, and we investigate the
resulting degree structure $\Ceers$, which is a poset with a smallest and a
greatest element. We point out a partition of the ceers into three classes:
the finite ceers, the light ceers, and the dark ceers. These classes yield
a partition of the degree structure as well, and in the language of posets
the corresponding classes of degrees are first order definable within
$\Ceers$. There is no least, no maximal, no greatest dark degree, but there
are infinitely many minimal dark degrees. We study joins and meets in
$\Ceers$, addressing the cases when two incomparable degrees of ceers $X,Y$
have or do not have a join or a meet according to where $X,Y$ are located
in the classes of the aforementioned partition: in particular no pair of
dark ceers has a join, and no pair in which at least one ceer is dark has a
meet. We also exhibit examples of ceers $X,Y$ having as a join their
uniform join $X \oplus Y$, but also examples with a join which is strictly
less than $X \oplus Y$. We study join-irreducibility and
meet-irreducibility: every dark ceer is both join-, and meet-irreducible.
In particular we characterize the property of being meet-irreducible for a
ceer $E$, by showing that it coincides with the property of $E$ being
self-full, meaning that every reducibility from $E$ to itself is in fact
surjective on its equivalence classes (this property properly extends
darkness). We then study the quotient structure obtained by dividing the
poset $\Ceers$ by the degrees of the finite ceers, and study joins and
meets in this quotient structure: interestingly, contrary to what happens
in the structure of ceers, here there are pairs of incomparable equivalence
classes of dark ceers having a join, and every element different from the
greatest one is meet-reducible. In fact in this quotient structure, every
degree different from the greatest one has infinitely many strong minimal
covers, whereas in $\Ceers$ every degree different from the greatest one
has either infinitely many strong minimal covers, or the cone strictly
above it has a least element: this latter property characterizes the
self-full degrees. We look at automorphisms  of $\Ceers$, and show that
there are continuum many automorphisms fixing the dark ceers, and continuum
many automorphisms fixing the light ceers. Finally, we compute the
complexity of the index sets of the classes of ceers studied in the paper.
\end{abstract}

\maketitle
\section{Introduction}\label{sct:introduction}
Given equivalence relations $E,R$ on the set $\omega$ of natural numbers we
say that $E$ is \emph{computably reducible} (or, simply, \emph{reducible}) to
$R$ (notation: $E \leq R$) if there exists a computable function $f$ such
that $x \rel{E} y$ if and only if $f(x) \rel{R} f(y)$, for all $x,y \in
\omega$. This reducibility (which can be viewed as a natural computable
version of Borel reducibility on equivalence relations, widely studied in
descriptive set theory, see for instance~\cite{Becker-Kechris:descriptive,
Gao:Book}) has recently been investigated by several authors, both as a
suitable tool for measuring the relative complexity of familiar equivalence
relations in computable mathematics (see for instance
\cite{Fokina-et-al-several, Fokina-Friedman-Nies, Ianovski-et-al}), and as an
interesting object in itself which is worthy of being studied from the
computability theoretic point of view (see for instance \cite{Gao-Gerdes,
Coskey-Hamkins-Miller, ceers, Andrews-Sorbi:index-sets, jumpsofceers}).

Since $\leq$ is a pre-ordering relation, it originates an equivalence
relation $\equiv$ by letting $E \equiv R$ if $E \leq R$ and $R\leq E$. The
equivalence class of an equivalence relation $E$ is called the \emph{degree}
of $E$, denoted by $\deg(E)$; on degrees one defines the partial ordering
relation $\deg(E) \leq \deg(R)$ if $E\leq R$.

\subsection{The computably enumerable equivalence relations}\label{sct:ceers}
As is often the case in computability theory when studying degree structures,
also for degrees of equivalence relations it seems natural to restrict
attention to local structures of degrees, for instance confining oneself  to
classes of equivalence relations in the arithmetical hierarchy or in
other hierarchies,
and in particular to computably enumerable equivalence relations (i.e.
equivalence relations $E$ on $\omega$ such that the set $\{(x,y) \mid x
\rel{E} y\}$ is computably enumerable, or, simply, c.e.), hereinafter called
\emph{ceers}, which play an important role in mathematical logic: they appear
for instance as relations of provable equivalence of well formed formulas in
formal systems; or as word problems of finitely presented structures; or as
equality in computably enumerable structures (also called positive structures
in the Russian literature, where ceers are more often called \emph{positive
equivalence relations}), a topic which dates back at least to
Mal'tsev~\cite{Mal'cev:towards} (translated in \cite{Mal'tsev:book}): recent
papers relating ceers and computable reducibility to various algebraic and
relational structures are \cite{Gavryushkin-Jain-Khoussainov-Stephan,
Franketal:structures, FoKhSeTu}.

Although there is already a nontrivial literature on applications of
computable reducibility to ceers (pioneering papers in this regard are
\cite{Ershov:positive, Gao-Gerdes}), emphasis so far has been mostly on the
so-called universal ceers, i.e. those ceers to which every other ceer is
reducible (for a recent survey on universal ceers see \cite{ABS}): for
instance, the relation of provable equivalence of strong enough formal
systems, such as Peano Arithmetic, is universal (\cite{Visser:Numerations,
Bernardi-Sorbi:Classifying, Montagna:ufp, Lachlan:note}); the relation of
isomorphism of finite presentations of groups is universal
(\cite{MillerIII:Book}); there are finitely presented groups whose word
problem is universal (\cite{MillerIII:Book}; see also \cite{Nies-Sorbi});
sufficient conditions for ceers guaranteeing universality have been pointed
out, including precompleteness (\cite{Bernardi-Sorbi:Classifying}), uniform
finite precompleteness (\cite{Montagna:ufp}), and uniform effective
inseparability of distinct pairs of equivalence classes  (\cite{ceers}).
Finally, the universal ceers can be nicely characterized as the ceers which
coincide up to equivalence with their jumps (\cite{ceers}).

Restriction of $\leq$ to ceers gives rise to
a degree structure $(\Ceers, \leq)$ which is a poset with a least
element $\mathbf{0}$ (the degree of the ceers with only one equivalence
class) and a greatest element $\mathbf{1}$, consisting of the universal ceers.
It is shown in \cite{ceers}
that the first order theory (in the language of posets) of the poset
$(\Ceers, \leq)$ is undecidable.
Apart from this, not much is known about the structure of $(\Ceers, \leq)$,
and this paper aims to fill in this gap.

Following Ershov~\cite{Ershov:NumberingsI}, ceers together with the
reducibility $\leq$ can be structured as a category: if $E,R$ are ceers then
a \emph{morphism} $\mu: E \longrightarrow R$ is a function $\mu: \omega_{/E}
\longrightarrow \omega_{/R}$, between the respective quotient sets for which
there is a computable function $f$ such that $\mu([x]_E)=[f(x)]_R$: we say in
this case that $f$ \emph{induces} $\mu$. Since monomorphisms in this category
are easily seen to coincide with the injective morphisms, then clearly $E\le
R$ if and only if there exists a monomorphism $\mu: E \longrightarrow R$. We
say that two ceers $E, R$ are \emph{isomorphic} (in symbols: $E \simeq R$) if
they are isomorphic in the sense of category theory, which, in this case,
amounts to saying that there exists a computable function (not necessarily a
bijection) which induces a $1$-$1$ and onto morphism. The symbol $\Id$
denotes the equality relation, whereas $\Id_{n}$ for $n \ge 1$ denotes
equivalence $\textrm{mod}_{n}$. Every computable ceer with infinitely many
classes is isomorphic to $\Id$; and every ceer with $n$ classes is isomorphic
to $\Id_{n}$.

Clearly $\simeq$ implies $\equiv$, but it is easy to see that the converse is
not always true: in fact it is known (see \cite{Badaev-Sorbi}) that the
universal degree $\mathbf{1}$ contains infinitely many different computable
isomorphism types. Moreover, $\simeq$ does not imply in general computable
isomorphism, i.e. the existence of a computable permutation inducing the
isomorphism, but it is easy to see (see e.g. \cite{ceers}) that it does so if
all classes in both equivalence relations are infinite.

\subsection{Notations and some background material}\label{ssct:background}
This paper is essentially self-contained. Computability theoretic notations
and terminology can be found in any standard textbook such as
\cite{Rogers:Book} or \cite{Soare:Book}: in particular $\{\phi_e\mid e \in
\omega\}$ is a standard listing of all partial computable functions, and
$\{W_e\mid e \in \omega\}$ is a standard listing of all c.e. sets. In the
rest of this section we review some basic facts concerning ceers: for more on
the topic see also the papers \cite{Ershov:positive,Gao-Gerdes, ceers}.

We will refer to some acceptable indexing $\{R_z \mid z \in \omega\}$ of the
ceers (such as the one defined in \cite{ceers}), and to approximations
$\{R_{z,s} \mid z,s \in \omega\}$ such that $R_{z,0}=\Id$, $R_{z,s} \subseteq
R_{z,s+1}$, $R_z=\bigcup_s R_{z,s}$, and $R_{z,s} \smallsetminus \Id$ is a
finite set uniformly given in $z,s$ by its canonical index.

Given an equivalence relation $E$, the $E$-equivalence class of a number $x$
will be denoted by $[x]_{E}$; if $U \subseteq \omega$ then $[U]_{E}$ denotes
the $E$-closure of $U$, i.e. the set of all numbers that are $E$-equivalent
to some $x \in U$; a set $U$ is said to be \emph{$E$-closed} if $U=[U]_{E}$.

The next lemma shows that every onto monomorphism is an isomorphism in the
category of ceers.

\begin{lemma}\label{ontoReductionsFlip}
If $f$ is a reduction of $E$ to $R$ with the property that the range of $f$
intersects every class in $R$, i.e., for every $y$, there is some $x$ so that
$f(x)\rel{R} y$, then $R \leq E$.
\end{lemma}

\begin{proof}
Given a reduction $f$ of $E$ to $R$, let $g(y)$ be the first $x$ so that we
see that $f(x) \rel{R} y$. This gives a reduction of $R$ to $E$.
\end{proof}

If $U \subseteq \omega$ then $R_{U}$ denotes the equivalence relation $x
\rel{R_{U}} y$ if and only if $x, y \in U$, or $x=y$. Clearly, if $U$ is c.e.
then $R_{U}$ is a ceer.

\begin{fact}\cite{ceers}\label{fact:fund-ce-sets}
Let $U,V$ be c.e.\ sets. The following hold:
\begin{enumerate}
  \item If $V$ is infinite then $U \le_1 V$ if an only if $R_U \leq R_V$.
  \item If $U$ is c.e. and $R\leq R_U$ then there exists a c.e. set $V$
      such that $R\equiv R_V$.
\end{enumerate}
\end{fact}

\begin{proof} We sketch the proofs:

\begin{enumerate}
\item If $U \le_{1} V$ via $f$ then clearly $f$ is also a reduction from
    $R_{U}$ to $R_{V}$ as it maps distinct equivalence classes to distinct
    equivalence classes. If $f$ is a reduction witnessing that $R_{U} \le
    R_{V}$ and $V$ is infinite, then consider the function $g$ which maps
    $x$ to $f(x)$ if $f(x) \notin \{g(i)\mid i< x\}$; otherwise map $x$ to
    the first $y\in V$ such that $y \notin \{g(i)\mid i< x\}$. The function
    $g$ provides the desired reduction $U\leq_{1}V$.

\item Suppose that $R\leq R_U$ via a reduction $f$. Pick any computable
    surjection $h: \omega \rightarrow \im(f)$ and let $V=h^{-1}[U]$: it is
    easy to check that $R\equiv R_{V}$.
\end{enumerate}
\end{proof}

We recall the definition of the jump operation on ceers, due to Gao and
Gerdes~\cite{Gao-Gerdes}. If $E$ is a ceer, then we define the \emph{jump} of
$E$ to be the ceer $E'$, where $x \rel{E'} y$ if and only if $x=y$ or both
$\phi_x(x), \phi_y(y)$ converge and $\phi_x(x) \rel{E} \phi_y(y)$. For more
information about this jump see ~\cite{Gao-Gerdes,ceers,jumpsofceers}.

Given equivalence relations $E,R$ we denote by $E \oplus R$ the equivalence
relation (called the \emph{uniform join}, or \emph{uniform upper bound}, of
$E,R$) so that $x$ and $y$ are equivalent if and only if $x$ and $y$ are both
even, say $x=2u$ and $y=2v$ and $u \rel{E} v$; or $x$ and $y$ are both odd,
say $x=2u+1$ and $y=2v+1$ and $u \rel{R} v$. Clearly, $E \oplus R$ is an
upper bound of $E$ and $R$ with respect to $\leq$. A ceer $E$ is called
\emph{uniform join-irreducible} (\cite{jumpsofceers}) if $E\leq R$ or $E\leq  S$
whenever $E\leq R\oplus S$.

\begin{fact}\cite{jumpsofceers}\label{JumpsUniformJoinIrreducible}
For every $E$, the jump $E'$ is uniform join-irreducible.
\end{fact}

\begin{proof}
This is Theorem~2.4 in \cite{jumpsofceers}: for the convenience of the
reader, we sketch the proof therein given. Suppose $f$ witnesses that $E'\leq
R\oplus S$.  By effective inseparability of the sets $K_i=\{x \mid
\phi_x(x)\downarrow=i\}$ and the definition of $E'$, it is not difficult to
see that for $i,j \in K$, $f(i)$ and $f(j)$ have the same parity. So suppose
$f(i)$ is even for every $i\in K$ (a similar argument will apply if $f(i)$ is
odd for every $i\in K$): this gives that the set $Y=f^{-1}[2 \omega+1]$ is a
decidable set in $\omega\smallsetminus K$. By productivity of $\omega
\smallsetminus K$ it is easy to see that there exists an infinite decidable
set $X$ contained in $\omega\smallsetminus (K\cup Y)$. We are now able to
show that $E'\leq R$. Fix a computable injection $h$ from $X\cup Y$ to $X$
(notice that $f(h(z))$ is even for every $z \in X\cup Y$), and define
\[
g(z)=
\begin{cases}
\frac{f(z)}{2}, &\text{if $z\in \omega\smallsetminus (X\cup Y)$},\medskip\\
\frac{f(h(z))}{2}, &\text{if $z\in X\cup Y$}.
\end{cases}
\]
A straightforward case-by-case inspection shows that $g$ reduces $E'\leq R$.
\end{proof}

\begin{cory}[\cite{ceers}]\label{cor:universal-join-irreducible}
If $E$ is universal then $E$ is uniform join-irreducible.
\end{cory}
\begin{proof}
If $E$ is universal then $E\equiv E'$, see e.g. \cite{Gao-Gerdes}.
\end{proof}

\subsection{A warning about terminology}\label{ssct:terminology}
To simplify notations and terminology, throughout the paper we will often
identify the degree $\deg(X)$ of a given ceer $X$ with the ceer $X$ itself.
In this vein (although nothing prevents one from talking about joins and
meets in a pre-ordered structure, though they need not be unique, whereas
they are unique in posets) when talking about least upper bounds or
greatest lower bounds of ceers we will mean least upper bounds or greatest lower bounds of their degrees.
Moreover we will often identify computable functions with the morphisms they
induce: for instance, when we say that an equivalence class is in the range
of a computable function, we will in fact mean by this that the equivalence
class is in the range of the morphism induced by the function.

\subsection{The main results of the paper}
We show that the ceers can be partitioned into three classes: the finite
ceers ($\I)$, the light ceers ($\Light$, those for which there exists some
computable listing $(y_i)_{i\in \omega}$ of infinitely many pairwise
non-equivalent numbers), and the dark ceers ($\Dark$, the remaining ones).
These notions are closed under equivalence of ceers, so they partition the
degrees of ceers as well. In the language of posets, the corresponding
classes of degrees (also denoted by $\I$, $\Light$ and $\Dark$) are first
order definable in $\Ceers$. There is no least, no maximal, no greatest dark
degree, but there are infinitely many minimal dark degrees. We carry out a
thorough investigation on when incomparable degrees of ceers $X,Y$ have a
join or a meet, as they vary in the classes $\Dark$ and $\Light$: basically
all possibilities may happen, i.e. one can find pairs of ceers $X,Y$ with or
without a join or a meet, except for the case when $X,Y$ are both dark (when
they have no join), and the case of a pair in which at least one ceer is dark
(when they have no meet). Dark degrees are both join-irreducible and
meet-irreducible. In fact, meet-irreducibility characterizes a useful class
of ceers, called self-full, namely ceers $E$ such that the range of every
reduction from $E$ to $E$ intersects all equivalence classes. Self-fullness
properly extends darkness. It is also interesting to observe that there are
ceers $X,Y$ with a join $X\lor Y$ such that $X\lor Y \equiv X \oplus Y$, and
examples when $X\lor Y < X \oplus Y$. One of the main tools to prove results
about meets is the Exact Pair Theorem~\ref{thm:EPT}. We then analyze what
happens in the quotient structure $\Ceers_{\I}$ obtained by dividing $\Ceers$
modulo the degrees of finite ceers. Basically the main differences are that
in the quotient structure there are incomparable dark degrees with a join,
incomparable dark-light pairs with a meet, and all non-universal degrees are
meet-reducible. Particular attention in the above investigation has been
given to the ceers of the form $R_X$, where $X$ is a c.e. set. The
aforementioned results show in many cases simple elementary differences
between the structures $\Ceers$, $\Dark$, $\Light$, $\Ceers_{\I}$,
$\Dark_{\I}$ and $\Light_{\I}$. We show that in $\Ceers_{\I}$ every degree
different from the greatest one has infinitely many strong minimal covers,
whereas in $\Ceers$ every degree different from the greatest one has either
infinitely many strong minimal covers, or the cone strictly above it has a
least element: this latter property characterizes the self-full degrees.
Tables~\ref{table:1}~through~\ref{table:6} summarize the behavior of joins
and meets in the structures $\Ceers$ and $\Ceers_\I$. We look at
automorphisms of $\Ceers$, and show that there are continuum many
automorphisms fixing the dark ceers, and continuum many automorphisms fixing
the light ceers. Any automorphism fixing the light ceers send every ceer to
a ceer in the same $\I$-degree. Since there are automorphisms
of $\Ceers$ which do not preserve the jump
operation it follows that the jump is not first order definable in the structure
of $\Ceers$. Finally, we compute the complexity of the
index sets of the classes of ceers studied in the paper.

\section{Preliminary observations: uniform joins, quotients of
ceers, and restrictions}\label{sct:preliminary}

We begin with a few notions and easy preliminary observations
which will be repeatedly used  in the rest of the paper.

\subsection{Uniform joins}\label{ssct:uniform-joins}
The operation $\oplus$ on equivalence relations is associative modulo
$\equiv$ (even $\simeq$): whatever way one decides to associate, it is easy
to see that $E_0 \oplus E_1 \oplus \cdots \oplus E_{n-1}$, with $n\ge 2$, is
equivalent to the equivalence relation $x \rel{E} y$ if and only if there is
$i<n$ such that $x, y= i \textrm{ mod}_{n}$ and $\frac{x-i}{n} \rel{E_i}
\frac{x-i}{n}$: this will therefore be taken as the definition of $E_0 \oplus
E_1 \oplus \cdots \oplus E_{n-1}$ throughout the paper. Given a countable
collection $(E_i)_{i \in \omega}$ of ceers we define $\bigoplus_i E_i$ to be
the ceer so that $ \langle j,x\rangle \rel{\bigoplus_i E_i} \langle k,y
\rangle$ if and only if $j=k$ and $x \rel{E_j}y$: if the family $(E_i)_{i \in
\omega}$ is uniformly c.e. then $\bigoplus_i E_i$ is a ceer.

The mapping $X \mapsto X \oplus \Id_{1}$ is an order embedding:

\begin{lemma}\label{SuccessorIsInjective}
$E \oplus \Id_1\leq R\oplus \Id_1$ if and only if $E\leq R$.
\end{lemma}

\begin{proof}
The right-to-left direction is immediate. For the other direction, let $f$ be
a reduction of $E\oplus \Id_1$ to $R\oplus \Id_1$. If the images of the
$E$-classes all land in $R$-classes, then the claim is obvious. Otherwise,
let $f(2a)$ be odd. Then the $E$-class of $a$ is computable, so we can define
$g(x)=\frac{f(1)}{2}$ if $x\rel{E} a$ and $g(x)=\frac{f(2x)}{2}$ otherwise.
This gives a reduction of $E$ to $R$.
\end{proof}

\begin{defn}\label{def:effectively-discrete}
We call an \emph{effective transversal} of a ceer $E$  any c.e. set $U$ whose
elements are pairwise non-$E$-equivalent.
A \emph{strong effective transversal} of a ceer $E$ is a decidable effective
transversal $U$ such that $[U]_{E}=U$.
\end{defn}

\begin{lemma}\label{obs:coproduct2}
Suppose $E,R$ are ceers such that there are a reduction $f$ showing $E \le
R$, and an infinite effective transversal $U$ of $R$ such that the predicate
$f(x) \in [U]_R$ is decidable. Then $E\oplus \Id \leq R$. In particular (by
considering the identity reduction $E \leq E$) if $E$ has an infinite
effective transversal $U$ such that $[U]_E$  is decidable then $E \oplus \Id
\leq E$. (Notice that the assumption that the predicate $f(x) \in [U]_R$ is
decidable is trivially fulfilled if $U$ is a strong effective transversal).
\end{lemma}

\begin{proof}
Let $f$ be a reduction from $E$ to $R$, and let $U$ be an infinite effective
transversal of $R$, effectively listed without repetitions by $(y_i)_{i\in
\omega}$. Let $V,W$ be infinite disjoint c.e. sets so that $U = V \cup W$,
with $h,k$ computable functions such that $h$ is a bijection between $\{y_i
\mid i\in \omega\}$ and $V$, and $k$ is a bijection between $\omega$ and $W$.
Define a reduction from $E \oplus \Id$ to $R$ as follows: map an even number
$2x$ to $f(x)$ if $f(x) \notin [U]_{R}$, and otherwise map $2x$ to $h(y_i)$
if $f(x)\rel{R} y_i$; map $2x+1$ to $k(x)$.
\end{proof}

\begin{defn}\label{def:effectively-for-reduction}
If $E,R,U,f$ are as in Lemma~\ref{obs:coproduct2} we call $U$ an
\emph{effective transversal
(\emph{or a} strong effective transversal \emph{according to the case}) for
the reduction $f$ from $E$ to $R$}.
\end{defn}

\begin{lemma}\label{obs:plus-span}
Suppose that $f$ is a reduction from $E$ to $R$, and there is  an effective
transversal $U$ of $R$ such that $[U]_{R}$ does not intersect $\im(f)$ and
$n=|U|$ (where $|U|$ denotes the cardinality of $U$), with $1 \leq n\leq
\omega$. Then $E \oplus \Id_n \leq R$. Moreover, if $[U]_R$ contains all the
equivalence classes not in the range of $f$ then $E\oplus \Id_n \equiv R$.
\end{lemma}

\begin{proof}
Let  $(y_i)_{i < n}$ effectively list $U$  without repetitions. Map $2x
\mapsto f(x)$ and $2x+1 \mapsto y_i$ where $x=i \textrm{ mod}_n$ ($x=i
\textrm{ mod}_\omega$ means $x=i$). The latter claim about the equivalence
$E\oplus \Id_n \equiv R$ follows from Lemma~\ref{ontoReductionsFlip}.
\end{proof}

\subsection{Quotients}
If $E$ is an equivalence relation and $W\subseteq \omega^2$ then by $E_{/W}$
we denote the equivalence relation generated by the set of pairs $E\cup W$.
If $W$ is a singleton, say $W=\{(x,y)\}$, then we simply write $E_{/(x,y)}$
instead of $E_{/\{(x,y)\}}$. Clearly, if $E$ is a ceer and $W$ is c.e. then
$E_{/W}$ is a ceer, called a \emph{quotient} of $E$: notation and terminology
are motivated by the obvious facts that $E \subseteq E_{/W}$ and, given a
ceer $R$, we have that $R$ is a quotient of $E$ in the sense of category
theory (i.e. there is an onto morphism from $E$ to $R$) if and only if there
is a c.e. set $W\subseteq \omega^2$ such that $R \simeq E_{/W}$.

\begin{lem}\label{ComputableCollapses}
Let $E$ be any ceer with a computable class $[x]_E$, and let $y$ be
non-$E$-equivalent to $x$. Then $E_{/(x,y)}\oplus \Id_1\equiv E$. More
generally if $(x_1,y_1), \ldots, (x_n, y_n)$, $n \ge 1$, are pairs so that
each $[x_i]_E$ is computable and at least one pair consists of
$E$-inequivalent numbers, then there is $1\le k\le n$ such that
$E_{/\{(x_i,y_i)\mid 1\le i\le n\}} \oplus \Id_k \equiv E$.
\end{lem}

\begin{proof}
We prove the claim for $n=1$. By Lemma~\ref{obs:plus-span}, it is enough to
see that one can reduce $E_{/(x,y)}$ to $E$ by a reduction that misses
exactly one class. For this reduction simply send every $z\in [x]_E$ to $y$,
every other element to itself. The range of this reduction misses exactly
$[x]_E$.

The general case is similar. Just notice that $k$ might be $<n$ even if all
$x_i$ are non-$E$-equivalent and all $y_i$ are non-$E$-equivalent: this is
the case for instance if $x_{1} \rel{E} y_{2}$ and $y_{1} \rel{E} x_{2}$,
in which case we are creating less than $n$ $E$-collapses.
\end{proof}

\begin{lemma}\label{obs:coproduct}
Suppose we have reductions of ceers $E_1 \le R$, $E_2 \leq R$ witnessed
respectively by computable functions $f_1, f_2$, and let $W=\{(x,y)\mid
f_1(x) \rel{R} f_2(y\}$. Then $(E_1 \oplus E_2)_{/W} \leq R$, where for
simplicity we denote  $(E_1 \oplus E_2)_{/W}=(E_1 \oplus E_2)_{/\{(2x, 2y+1)
\mid (x,y) \in W\}}$.
\end{lemma}

\begin{proof}
A reduction $h$ from $(E_1 \oplus E_2)_{/W}$ to $R$ is simply
$h(2x)=f_{1}(x)$ and $h(2x+1)=f_{2}(x)$.
\end{proof}

\subsection{Restrictions}
We conclude this preliminary section by defining the notion of restriction of
a ceer with respect to a given c.e. set.

If $E$ is a ceer and $W$ is a c.e. set then pick a computable surjection $h:
\omega \rightarrow  [W]_E$ and define $E\restriction{W}$ (called the
\emph{restriction of $E$ to $W$}) to be the ceer $x \rel{E\restriction{W}} y$
if and only if $h(x) \rel{E} h(y)$. It is easy to see that up to $\equiv$ the
definition does not depend on the chosen $h$. If $[W]_{E}$ is infinite we may
assume that $h$ is a computable bijection $h: \omega \rightarrow [W]_E$.
Clearly $E\restriction{W} \leq E$.

The next lemma summarizes some properties of restrictions which will be
repeatedly used throughout the paper.

\begin{lemma}\label{coproduct3}
The following hold:
\begin{enumerate}
\item Suppose $f$ gives a reduction $E \le R$, $W$ is a c.e. set, and $U$
    is an effective transversal of $R$ such that $[U]_R\cap [W]_R=
    \emptyset$, and $\im(f) \subseteq [W]_R\cup [U]_R$. Then there exists a
    ceer $E_0$ such that $E \leq E_{0} \oplus \Id_n$ where $n \leq \omega$
    is the number of classes in $[U]_{R}$ (we agree that $E \leq E_{0}
    \oplus \Id_0$ must be understood as $E\leq E_0$), and $E_0 \leq R$ via
    a reduction whose image coincides exactly with the $R$-equivalence
    classes of elements in $\ran(f)\cap [W]_R$ (thus if $R=X \oplus Y$ and
    $W$ is the set of even numbers then $E_0 \leq X$). Moreover all classes
    in the $E_0$-part of $E_{0} \oplus \Id_n$ are in the range of the
    reduction $E \leq E_{0} \oplus \Id_n$; and if the range of $f$
    intersects all classes in $[U]_{R}$ then $E \equiv E_{0} \oplus \Id_n$.

\item Suppose $f$ gives a reduction $E \le R \oplus \Id_{n}$, for some $1
    \leq n\in \omega$. Then there exists a ceer $E_{0}$ such that $E_{0}
    \le R$ and $E \equiv E_{0} \oplus \Id_{k}$, for some $k \le n$.

\item  Suppose $E \le X \oplus \Id_{n}$ and $E \le Y \oplus \Id_{n}$, for
    some $1 \leq n\in \omega$. Then there exists a ceer $E_{0}$ such that
    $E_{0} \le X,Y$ and $E \equiv E_{0} \oplus \Id_{k}$, for some $k \le
    2n$.

\item Suppose $f$ gives a reduction $E \le R \oplus \Id$, where $\im(f)$
    intersects infinitely many classes in the $\Id$-part. Then there exists
    a ceer $E_{0}$ such that $E_{0} \le R$ and $E \equiv E_{0} \oplus \Id$.
\end{enumerate}
\end{lemma}

\begin{proof}
We prove the various items one by one:
\begin{enumerate}

\item Let $E,R,f,W, U$ be as in the statement of the lemma. Let $E_{0}=
    E\restriction{V}$ where $V= f^{-1}[[W]_R]$, and let $h$ be the
    computable surjection used to define the restriction. Then $f \circ h$
    gives a reduction of $E_0$ to $R$ whose image coincides exactly with
    the $R$-equivalence classes of elements in $\ran(f)\cap [W]_R$.
    For a reduction $E
    \leq E_{0} \oplus \Id_n$, choose a computable listing $(y_i)_{i < n}$
    without repetitions of $U$, and consider the computable function $g$
    defined as follows: on input $x$ search for the first $y \in W \cup U$
    so that $f(x) \rel{R} y$ (exactly one of the two cases among $y \in W$
    or $y \in U$ holds): if $y \in U$, and $y \rel{R} y_i$, then map
    $x\mapsto 2i+1$; otherwise map $x \mapsto 2h^{-1}(x)$, where $h^{-1}$
    is the partial computable function so that $h^{-1}(z)$ is the first
    seen $u$ such that $h(u)=z$. It is clear that all classes in the
    $E_0$-part are in the range of the reduction. The last claim (about $E
    \equiv E_{0} \oplus \Id_n$ if the range of $f$ intersects all classes
    in $[U]_R$) follows from Lemma~\ref{ontoReductionsFlip}.

\item In item (1) take $W$ to be the even numbers and $U$ its complement.
    Then by item (1) there exists a ceer $E_0$ such that $E_{0} \leq R$ and
    $E \le E_{0} \oplus \Id_{n}$. If $k\le n$ is such that $f$ hits exactly
    $k$ classes in the $\Id_n$-part then, by a slight modification, $f$ can
    be viewed as a reduction $E \le E_{0} \oplus \Id_{k}$, which can be
    inverted by Lemma~\ref{ontoReductionsFlip}, as all classes in the
    $E_0$-part and in the $\Id_k$-part are all in the range of the
    reduction.

\item Let $f,g$ provide reductions $E \le X \oplus \Id_{n}$ and $E \le Y
    \oplus \Id_{n}$, respectively. Take $W$ to be the even numbers, and let
    $E_{0}=E\restriction{V}$, where $V=f^{-1}[W]\cap g^{-1}[W]$. Then
    arguing as in item (1) it is not difficult to show that $E_{0} \le X,
    Y$, and $E \le E_{0} \oplus \Id_{2n}$: in fact, as in the previous
    item, $E \equiv E_{0} \oplus \Id_{k}$ for some $k \le 2n$, as all
    classes in the $E_0$-part are in the range of the reduction.

\item Let $Y=\{y \mid 2y+1 \in \range(f)\}$. Since $Y$ is an infinite c.e.
    set let $g$ be a computable bijection from $Y$ to $\omega$. Then the
    function $h(x)=f(x)$ if $f(x)$ is even, and $h(x)=g(f(x))$ if $f(x)$ is
    odd is a reduction, having all odd numbers in its range. By item (1)
    there is a ceer $E_0$ such that $E \le E_{0} \oplus \Id$ and all
    classes of $E_0 \oplus \Id$ are in the range of the reduction, so that
    $E \equiv E_{0} \oplus \Id$ by Lemma~\ref{ontoReductionsFlip}.
\end{enumerate}
\end{proof}

\section{The collection of dark ceers}
It is known that there are ceers $R$ with infinitely many classes such that
$\Id \nleq R$: for instance, take $R_U$ where $U$ is any simple set. This
observation originates the next definition, which singles out the class of
dark ceers.

\begin{defn}
A ceer $R$ is \emph{dark} if it has infinitely many classes and $\Id\not\leq
R$, i.e. there is no infinite c.e. set $W$ so that $x \nrel{R} y$ for each
pair of distinct $x,y\in W$.
If $\Id\leq E$, then we say that $E$ is \emph{light}.
\end{defn}

Equivalently $E$ is light if and only if there is an infinite effective
transversal of $E$.

Notice that every universal ceer is light. The next observation shows that
the dark ceers are downward closed among those with infinitely many classes,
and they are closed under $\oplus$. In particular, a degree contains a dark
ceer if and only if it is comprised of only dark ceers; similarly a degree
contains a light ceer if and only if it is comprised of only light ceers.

\begin{obs}\label{obs:dark-closure}
If $E\leq R$, $E$ has infinitely many classes, and $R$ is dark, then $E$ is
dark. If $E_1, E_2$ are dark ceers then $E_1\oplus E_2$ is a dark ceer.
\end{obs}

\begin{proof}
If $E\leq R$, $E$ has infinitely many classes, and $\Id\leq E$, then $\Id
\leq R$.

Suppose $E_1\oplus E_2$ is light, i.e., let $X$ be an infinite effective
transversal of $E_1\oplus E_2$ (see
Definition~\ref{def:effectively-discrete}). Then either $X$ contains an
infinite set of even elements or an infinite set of odd elements. Thus either
$E_1$ or $E_2$ is light.
\end{proof}

The following theorem shows that there is no least dark ceer, unlike the
light ceers (where $\Id$ is the least light ceer), and there are infinitely
many minimal dark ceers. It will be useful later to have this result combined
with lower cone avoidance, so we do that here.

\begin{thm}\label{MinimalDark}
Let $R$ be a given non-universal ceer. Then there are infinitely many
pairwise incomparable dark ceers $(E_l)_{l \in \omega}$ such that, for every
$l$ and ceer $X$, $E_l \not\leq R$ and
\[
X<E_l \Rightarrow (\exists n)[X\leq \Id_n].
\]
\end{thm}

\begin{proof}
Let $R$ be a given non-universal ceer, with computable approximations $\{R_s
\mid s \in \omega\}$ as in Section~\ref{ssct:background}. We construct a
family $(E_l)_{l\in \omega}$ of ceers with the following requirements, for
all $i,j,l,l',n,o \in \omega$:
\begin{itemize}
\item[$P^l_{i,j}$:] If $W_i$ intersects infinitely many $E_l$-classes, then
    it intersects $[j]_{E_l}$.
\item[$Q^{l,l'}_n$:] if $l\ne l'$ then $\phi_n$ is not a reduction of $E_l$
    to $E_{l'}$.
\item[$T^l_o$:] $\phi_o$ is not a reduction of $E_l$ to $R$.
\end{itemize}

The $Q$-requirements ensure that the ceers $E_l$ are pairwise incomparable,
and also they have infinitely many classes as each finite ceer $\Id_{n}$ is
comparable with any ceer. Satisfaction of the $P$-requirements will ensure
that the $E_l$ are minimal and dark (see Lemma~\ref{lemPdark} and
Lemma~\ref{lemSdark} below). The $T$-requirements ensure that the ceers $E_l$
avoid the lower cone below $R$.

\begin{lem}\label{lemPdark}
Let $E$ be any ceer satisfying each $P$-requirement (where $E_l$ is taken to
be $E$), and suppose $X< E$. Then $X\leq \Id_n$ for some $n$.
\end{lem}

\begin{proof}
Suppose $f$ gives a reduction of $X$ to $E$. Then either the image $W=\im
(f)$ of $f$ intersects only finitely many classes, so $X\leq \Id_n$, or by
satisfaction of all $P_{i,j}^l$ (where $E_l$ is taken to be $E$ and $W_i$ is
the range of $f$) it intersects every class; but if $f$ is a reduction of $X$
to $E$ whose image intersects every class, then $E \leq X$ by
Lemma~\ref{ontoReductionsFlip}.
\end{proof}

\begin{lem}\label{lemSdark}
Let $E$ be any ceer with infinitely many classes satisfying every
$P$-requirement (where $E_l$ is taken to be $E$). Then $E$ is dark.
\end{lem}

\begin{proof}
Suppose, for a contradiction, that $W$ is an infinite effective transversal
of $E$. Let $x\in W$ and let $V=W\smallsetminus \{x\}$. Then $V$ intersects
infinitely many classes, thus by the $P_{i,j}^l$-requirements (where $W_i$ is
taken to be $V$ and $E_l$ is taken to be $E$), it must also intersect the
class of $x$. Thus $W$ contains two elements which are $E$-equivalent,
contradicting the hypothesis.
\end{proof}

We fix a computable order of the requirements of order type $\omega$. A
requirement $R$ has \emph{higher priority} than a requirement $R'$ if $R$
strictly precedes $R'$ in this order (we also say in this case that $R'$ has
\emph{lower priority} than $R$). Each requirement will be allowed to
\emph{restrain} finitely many pairs of equivalence classes: if a requirement
imposes a restraint so that the equivalence classes $([a]_{E_l},[b]_{E_l})$
can not be $E_l$-collapsed by lower priority requirements, then we say that
it imposes the \emph{restraint} $(a,b,l)$. In turn, each strategy must be
able to satisfy its requirement given that it inherits such a finite
restraint from higher priority requirements. Then the standard finite injury
machinery (in particular, whenever any strategy acts in any way, it
re-initializes all lower priority requirements (or at least those which may
be injured by the actions of higher priority strategies and thus need to be
re-initialized when this happens) which must therefore start anew to pursue
their strategies) completes the construction. We describe the strategy for
each requirement, analyzing its outcomes.

\medskip

\emph{Strategy for $P^l_{i,j}$:} Suppose $W_i$ and $[j]_{E_l}$ are still
disjoint, and $P^l_{i,j}$ inherits finitely many pairs of classes which are
restrained. $P^l_{i,j}$ waits for $W_i$ to enumerate an element $x$ so that
there is no higher priority restraint $(j,x,l)$. Since higher priority
requirements only restrain finitely many pairs of classes, if $W_i$
intersects infinitely many $E_l$-classes, then eventually it will enumerate
some element $x$ so that there is no restraint $(j,x,l)$: at this point
$P^l_{i,j}$ becomes \emph{ready to act}, and (\emph{$P^l_{i,j}$-action}) it
$E_l$-collapses $x$ to $j$. The outcomes are clear: either we wait forever
for such a number $x$ (this is the case when $W_i$ intersects only finitely
many $E_l$ equivalence classes); or we $E_l$-collapse some $x\in W_i$ with
$j$ (this outcome includes in fact also the case when $W_i \cap [j]_{E_l}$
becomes nonempty even without our direct action, but only because $W_i$
enumerates or has already enumerated a number already in $[j]_{E_l}$). Both
outcomes fulfill the requirement.

\medskip

\emph{Strategy for $Q^{l,l'}_n$:} Since we control both $E_l$ and $E_{l'}$,
we can diagonalize directly. That is: We fix two distinct new elements $x$
and $y$ for $E_l$ ($x,y$ are the \emph{parameters of $Q^{l,l'}_n$}; being new
they are still non-$E_{l}$-equivalent) and restrain $(x,y,l)$. When
$\phi_n(x),\phi_n(y)$ both converge, then  we diagonalize, that is: either
already $\phi_n(x) \rel{E_{l'}} \phi_n(y)$, in which case we just keep our
restraint $(x,y,l)$, or still $\phi_n(x)\cancel{\rel E_{l'}}\phi_n(y)$, in
which case $Q^{l,l'}_n$ becomes \emph{ready to act}, and
(\emph{$Q^{l,l'}_n$-action}) it $E_l$-collapses $x$ to $y$ and puts the
restraint $(\phi_n(x), \phi_n(y), l')$. The outcomes, both fulfilling the
requirement, are: either we wait forever for $\phi_n(x)$ and $\phi_n(y)$ to
converge; or these computations converge, and when they do so, either we do
nothing but keeping the restraint $(x,y,l)$ if already $\phi_n(x)
\rel{E_{l'}} \phi_n(y)$; or we act if still $\phi_n(x) \nrel{E_{l'}}
\phi_n(y)$. Either way,  $Q^{l,l'}_n$ creates a diagonalization which is then
preserved by the restraints imposed by the strategy.

\medskip

\emph{Strategy for $T^l_o$:} This strategy, which will be employed several
times in this paper, at first sight looks like it places infinite restraint,
but in truth it only places finite restraint. We call it the \emph{\basta}
strategy, as a reminder that, despite its appearance, it is finitary. We fix
a universal ceer $T$, with computable approximations $\{T_s \mid s \in
\omega\}$ as in Section~\ref{ssct:background}. We begin by choosing two new
elements $a_0,a_1$. At this stage we set the restraint $(a_0,a_1, l)$ and we
go into a \emph{waiting state}; at each later stage when we define a new
$a_k$, where $k$ is least so that $a_k$ is currently not defined, we set
restraints $(a_i,a_k,l)$ for every $i<k$ and we go into a \emph{waiting
state}. We emerge from a waiting state at stage $s$, having defined $a_{0},
\ldots, a_{k}$, if for every $i\le k$ we have (at $s$) $\phi_{o}(a_i)$
converges, and for every $i,j\le k $ we have $a_i \rel{E_{l}} a_j$ if and
only if $\phi_o(a_i) \rel{R} \phi_o(a_j)$. In this case
(\emph{$T^l_o$-action}), we $E_l$-collapse every pair $a_i,a_j$ such that, at
$s$, $i \rel{T} j$ and we go on defining a new $a_{k+1}$. (Notice that this
is the only way some of the $a_{i}$ can be $E_{l}$-collapsed, as
lower-priority requirements are not allowed to do so due to the restraints
imposed by $T^l_o$ every time we appoint a new $a_{k}$.) The available
numbers $a_i$ at stage $s$ are \emph{the parameters of $T^l_o$ at $s$}. The
strategy can be looked at as having possible sub-outcomes $1,2, \ldots$: when
we appoint $a_0, a_1$ we have current sub-outcome $1$; from sub-outcome $k$
we take sub-outcome $k+1$ when the strategy acts,  i.e. we emerge (due to new
$R$-collapses) from the waiting state relative to $a_0, \ldots, a_k$, and we
appoint $a_{k+1}$. If we never abandon sub-outcome $k$ (and this is the case
if $\phi_o(a_0)$ or $\phi_o(a_k)$ does not converge, or otherwise $\phi_o$ is
defined on each $a_0, \ldots, a_{k}$ but fails anyway to be a reduction $E_l
\leq R$, as witnessed by some pair some $a_i, a_j$, with $i,j \le k$) and
thus we never move to sub-outcome $k+1$, then we have that sub-outcome $k$ is
the (final) outcome of the strategy and the requirement is fulfilled as
$\phi_o$ is not a reduction.

The next lemma shows that in isolation the $T$-strategy eventually hits its
final winning outcome.

\begin{lem}\label{lemTdark}
The
\basta strategy is finitary. That is: it places finite restraint, acts only
finitely often, and the corresponding requirement is satisfied.
\end{lem}

\begin{proof}
Suppose that the requirement acts infinitely often. Then for every $i$ a
parameter $a_i$ is eventually appointed, and for each $i,j$, we must have $i
\rel{T} j$ if and only if $\phi_o(a_i) \rel{R} \phi_o(a_j)$. But this gives a
reduction of $T$ to $R$ contradicting the assumption that $R$ is not
universal.
\end{proof}

\medskip
\emph{The construction:} The construction is by stages. At stage $s$ we
define the current values  $E_{l,s}$ of $E_{l}$ and of the parameters
relative to the various requirements: only finitely many parameters are
defined at each stage. A requirement $R=Q^{l,l'}_n$ has parameters $x^R(s),
y^R(s)$; and $R=T^l_o$ has parameters $a^R_{0,s}, \ldots, a^R_{k,s}$, for a
certain $k\ge 1$, which can also be regarded as a parameter of the
requirement. Moreover each requirement $R$ deals with a finite set $\rho^R_s$
consisting of all triples $(a,b,l)$ imposed by a higher priority requirement:
of course $R$ cares only for those restraints $(a,b,l)$ where $l$ is such
that $R$-action could entail $E_l$-collapsing. In accordance with the
informal description of the strategies, a requirement $R$ is \emph{ready to
act} when its strategy, for the relevant index $l$, may $E_l$-collapse two
numbers $u, v$ (as in the description of the strategies) so that there is no
$(a,b,l)\in \rho^R$, with currently $u\rel{E_l} a$ and $v \rel{E_l} b$. In
fact only $P$-requirements $R$ wait for an action which avoids $\rho^{R}$:
the other requirements simply deal with $\rho^{R}$ by choosing parameters
which are new and thus none of them is equivalent with any $a$ which is a
first or second coordinate of a triple in $\rho^{R}$.

In the rest of the proof, in reference to the various parameters, we will
omit to specify the superscript $R$ which will be clearly understood from the
context. At the beginning of a stage $s$ a number is \emph{new} if it is
bigger than any number $E_l$-equivalent, for some $l$, to any number
mentioned so far by the construction. A requirement $R$ which is not a
$P$-requirement is \emph{initialized at stage $s$} if the parameters of $R$
are set to be undefined at $s$: we also stipulate that when $R$ is
initialized the restraint imposed by $R$ is cancelled. We say that a
requirement $R$ \emph{requires attention at stage $s$}, if either
$R=P^l_{i,j}$ (for some $l,i,j$), and $W_i$ and $[j]_l$ are still disjoint,
but now $R$ is ready to act as described; or  $R\in \{Q^{l,l'}_n, T^l_o\mid
k, l,l', o \in \omega\}$ and $R$ is initialized, or otherwise $R$ has not as
yet acted after its last initialization but now is ready to take action as
described. (The asymmetry in the previous definitions between requirements of
the form $P^l_{i,j}$ and the other ones, is that once it has acted,
$P^l_{i,j}$ is satisfied once for all, it will never be injured, and thus it
does not need to be re-initialized. On the other hand, the other requirements
must choose new parameters to avoid the restraint imposed by higher priority
requirements when these act.)

\medskip
\emph{Stage $0$.} Initialize all requirements $R$ which are not
$P$-requirements. Define $E_{l,0}=\Id$, for each $l$.

\medskip
\emph{Stage $s+1$.} (All parameters, computations and approximations are
understood to be evaluated at stage $s$.) Consider the least $R$ that
requires attention at $s+1$. (Since infinitely many strategies are
initialized, such a least strategy $R$ exists.)

If $R=P^l_{i,j}$, for some $l,i,j$, then  take $P^l_{i,j}$-action as
described. Otherwise
\begin{itemize}
  \item if $R$ is initialized then choose new parameters for $R$, as
      described in the strategy for $R$. In detail:
      if $R=Q^{l,l'}_n$ then $R$
      chooses new $x,y$; if $R=T^l_o$, then $R$ chooses new $a_0, a_1$.

  \item if $R$ is not initialized and $R$ is a $Q$- or $T$-requirement then
      take $R$-action. (If $R$ is a $T$-requirement, this means also to
      appoint a new parameter $a_{k+1}$ on top of the already existing
      $a_0, \ldots, a_k$.) Moreover, $R$ may put new restraints (as
      described earlier), thus updating the restraint sets $\rho^{R'}$ for
      lower-priority requirements $R'$.
\end{itemize}
Whatever case holds, we initialize all requirements $R'$ of lower priority
than $R$ which are not $P$-requirements, and go to next stage. Initialization
is a mechanism that will guarantee the desired restraints against lower
priority requirements: since after initialization a requirement $R$ chooses
new parameters (thus not $E_l$-equivalent to numbers appearing in $\rho^R$
for any $l$) its strategy is automatically respectful of the restraints
imposed by higher priority requirements, and for every $l$ the new parameters
are also pairwise $E_l$-inequivalent, so that the requirement may in turn
restrain them if needed.

Finally define $E_{l,s+1}$ to be the equivalence relation generated by
$E_{l,s}$ plus the pairs which have been $E_l$-collapsed at stage $s+1$.

\medskip
\emph{Verification:} Each strategy is finitary and the corresponding
requirement is eventually satisfied. Indeed, assume inductively on the
priority ordering of requirements that $R$ is such that every higher priority
requirement stops acting at some stage and is satisfied. Then there is a
least stage after which $R$ is not re-initialized any more and its current
parameters, if any and once defined, are never cancelled; since the restraint
set $\rho^R$ built up by the higher-priority requirements stops changing at
this stage, $R$ can pursue its strategy without any more interferences due to
higher-priority requirements. If $R$ is a $T$-requirement then
Lemma~\ref{lemTdark} shows that eventually $R$ stops acting, and is
satisfied. As to the other strategies, after last initialization they never
act, or act at most once in the cases of $P$- or $Q$-requirements, and are
satisfied, as is clear by Lemmata~\ref{lemPdark}, \ref{lemSdark}, and the
above analysis of the outcomes.
\end{proof}

\begin{cory}
There is no greatest dark ceer.
\end{cory}

\begin{proof}
Given a dark (hence non-universal) ceer $R$, the above theorem produces dark
ceers not below $R$.
\end{proof}

\begin{cory}
There is no maximal dark ceer.
\end{cory}

\begin{proof}
We argue that any maximal dark ceer would also be a greatest dark ceer. Let
$E$ be a maximal dark ceer and let $R$ be any dark ceer. Then $E\oplus R$,
being dark and  $\geq E$, must also be $\leq E$. Thus $R\leq E$.
\end{proof}

Figure~\ref{fig:poset} illustrates the partition of ceers into the three
classes $\Dark=\textrm{dark ceers}$, $\Light=\textrm{light ceers}$,
$\I=\textrm{finite ceers}$: the partition of course induces a partition of
the degrees of ceers as well. Every dark ceer $E$ is bounded by $E \oplus
\Id$ which is light and non-universal by
Corollary~\ref{cor:universal-join-irreducible}. For a complete understanding
of the picture notice also the following corollary.

\begin{cory}\label{cor:belowId}
If $X \leq \Id$ then $X \in \I$ or $X \equiv \Id$.
\end{cory}

\begin{proof}
Suppose that $X\leq \Id$ and $X$ has infinitely many classes. Then
$X \leq \Id_{1} \oplus \Id$ and the hypotheses of Lemma~\ref{coproduct3}(4),
are satisfied: thus $X \equiv \Id_{1} \oplus \Id$, which is $\equiv \Id$.
\end{proof}

\begin{figure}[h!]
  \centering
  \includegraphics[scale=.5]{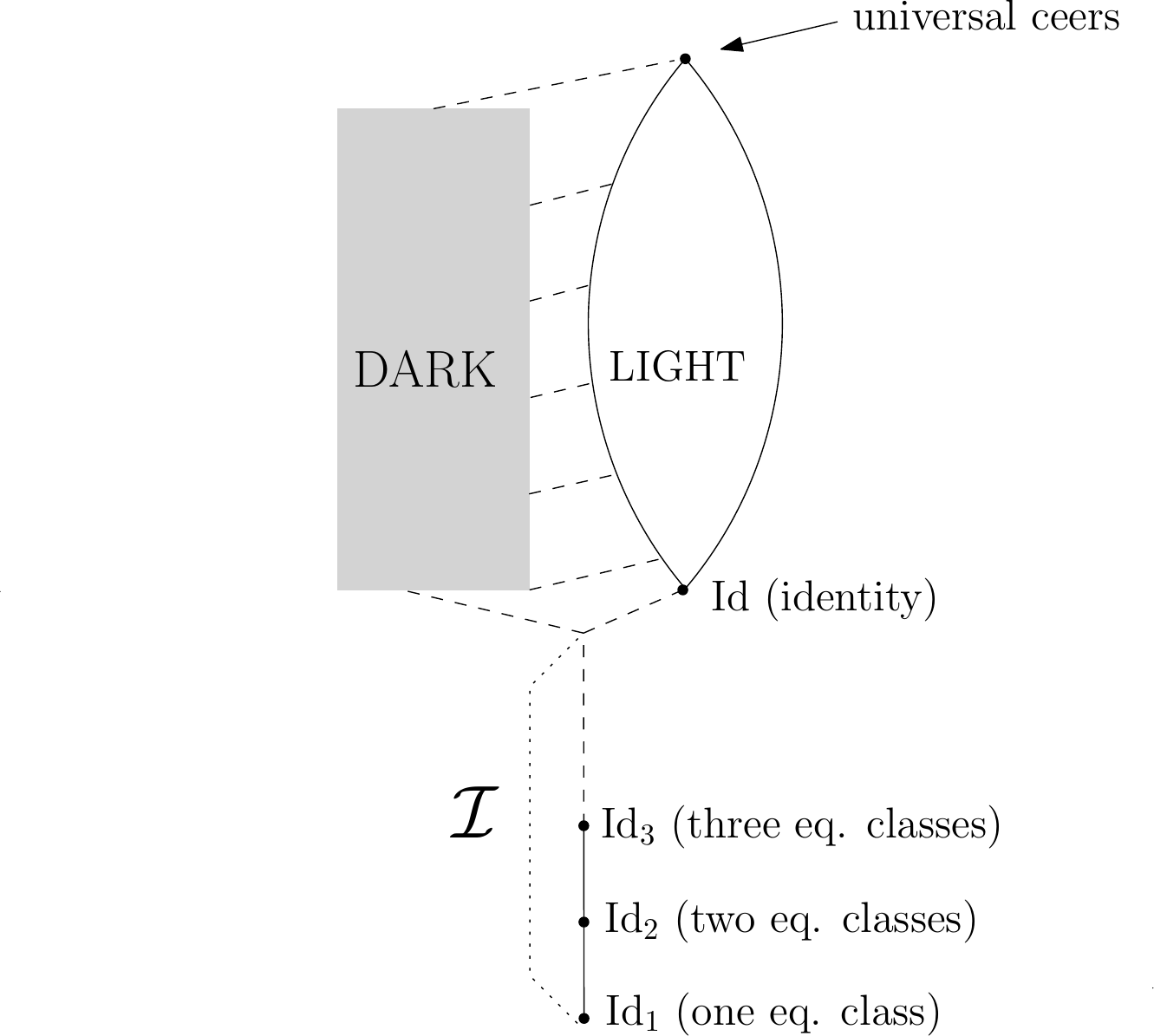}
  \caption{The poset of degree of ceers: some of the degrees are identified
  through representatives. The three classes $\mathcal{I}$, $\Dark$, and
  $\Light$ are pairwise disjoint.}\label{fig:poset}
\end{figure}

\section{Self-Fullness}
The next definition singles out the ceers $E$ for which, in a category
theoretic terminology, every monomorphism from $E$ to $E$ is an isomorphism.

\begin{defn}
We say that $E$ is \emph{self-full} if for every reduction $g$ of $E$ to $E$
and every $j\in \omega$, there is a $k$ so that $g(k)\in [j]_E$.
\end{defn}

\begin{obs}\label{obs:nonselffull}
$E$ is non-self-full if and only if $E\oplus \Id_1\leq E$.
\end{obs}

\begin{proof}
Suppose $E$ is non-self-full: then there is a reduction $f$ of $E$ to itself
missing the class of $j$ for some element $j$. Then we can provide a
reduction of $E\oplus \Id_1$ to $E$ by $g(2x)=f(x)$ and $g(2x+1)=j$.

In the other direction, suppose that $E\oplus \Id_1\leq E$ and let $g$ be a
reduction witnessing this. Then $n\mapsto g(2n)$ gives a reduction of $E$ to
itself missing the class of $g(1)$.
\end{proof}

\begin{cory}\label{cor:from-sf-to-sf}
If $E$ is self-full then for every $h \in \omega$, $E\oplus \Id_h$ is
self-full.
\end{cory}

\begin{proof}
By the previous observation it is easy to see that if $E$ is self-full then
so is $E \oplus \Id_1$.
\end{proof}

The next corollary shows that every self-full degree is comprised of only
self-full ceers.

\begin{cory}
If $E\equiv R$, then $E$ is self-full if and only if $R$ is self-full.
\end{cory}

\begin{proof}
Assume $E\equiv R$ and suppose $E$ is non-self-full. Then by
Observation~\ref{obs:nonselffull} $R\oplus \Id_1\leq E\oplus \Id_1\leq E\leq
R$. Thus, again by Observation~\ref{obs:nonselffull}, it follows that $R$ is
non-self-full. Symmetrically, if $R$ is non-self-full, then so is $E$.
\end{proof}

In a pre-order $\langle P, \le\rangle$, we say that $y$ is a \emph{strong minimal cover
of}  $x$, if $x<y$ and for all $z$, $z<y$ implies $z \le x$.

\begin{lem}\label{StrongMinimalCoversOfSF}
For any self-full degree $E$, $E\oplus \Id_1$ is a strong minimal cover of
$E$. Further if $X > E$, then $X\geq E\oplus \Id_1$, i.e., in the dual
pre-order, $E$ is a strong minimal cover of $E\oplus \Id_1$.
\end{lem}

\begin{proof}
Let $X\leq E\oplus \Id_1$ be given by a reduction $f$, and let $E$ be
self-full. If the $\Id_1$-class is not in the range of $f$, then $X\leq E$.
If some class $[j]_{E\oplus \Id_1}$ in the $E$-part is missed by $f$, then
$X\oplus \Id_1 \leq E\oplus \Id_1$ by sending $2x$ to $f(x)$, and the odd
numbers to $j$. Thus by Lemma~\ref{SuccessorIsInjective}, $X\leq E$.
Otherwise, $f$ intersects all classes of $E\oplus \Id_1$, so we get $E\oplus
\Id_1\leq X$ by Lemma~\ref{ontoReductionsFlip}.

Suppose $E<X$ given by the reduction $f$. Since $X\not\leq E$, $f$ does not
intersect all the classes of $X$ by Lemma~\ref{ontoReductionsFlip}. Let
$x$ be so that $[x]_X$ is not in the range of $f$. Then $g(2y)=f(y)$ and
$g(2y+1)=x$ is a reduction of $E\oplus \Id_1$ to $X$.
\end{proof}

In next lemma, and in other parts of the paper, we employ to notation
$f^{(k)}$ to denote the $k$-th iterate of a given (possibly partial) function
$f$ (i.e., $f^{(0)}(x)=x$ and $f^{(k+1)}(x)=f(f^{(k)}(x))$.

\begin{lem}\label{darksAreSelfFull}
Every dark ceer is self-full.
\end{lem}

\begin{proof}
Suppose $E$ is not self-full witnessed by a reduction $g$ of $E\oplus \Id_1$
to $E$ (see Observation~\ref{obs:nonselffull}). Let $f$ be the map $f(x)=
g(2x)$. Let $a$ be $g(1)$. Then since the $E$-class of $a$ is not in the
range of $f$, for any $k \ge 1$ we have that $a\cancel{\rel E} f^{(k)}(a)$.
But $f$ is a reduction of $E$ to $E$, so $f^{(n)}(a) \cancel{\rel E}
f^{(n+k)}(a)$ for every $n$ and for every $k\ge 1$. Thus the infinite c.e.
set $\{f^{(m)}(a)\mid m\in \omega\}$ witnesses that $E$ is light.
\end{proof}

An interesting consequence of the previous lemma is the following result.

\begin{lem}\label{darkQuotients}
Let $E$ be a dark ceer, and let $R$ be any proper quotient of $E$. Then
$E\not\leq R$.
\end{lem}

\begin{proof}
Suppose $f$ is a reduction from $E$ to $R$, with $R$ a proper quotient of
$E$, and $E$ dark. Firstly, we note that $f$ preserves non-$E$-equivalence:
if $ x\cancel{\rel E} y$ then $ f(x)\cancel{\rel R}f(y)$, but since $R$ is a
quotient of $E$, we see that this implies $ f(x)\cancel{\rel E}f(y)$.

We now claim that for every $c$ there exists $x$ such that $f(x) \rel{E} c$.
Suppose not, and let $c$ be such that $f(x) \cancel{\rel{E}} c$ for every
$x$. But then $f^{(n)}(c) \cancel{\rel{E}} c$ for every $n \ge 1$. Since $f$
preserves non-$E$-equivalence we have that $f^{(n+k)}(c) \cancel{\rel{E}}
f^{(k)}(c)$ for every $n \ge 1$, and $k$. This gives an infinite c.e. set
$\{f^{(m)}(c): m \in \omega\}$ of non-$E$-equivalent elements, contradicting
that $E$ is dark.

Let $a,b$ be such that $a \cancel{\rel{E}} b$ but $a \rel{R} b$. Using the
above claim, define by induction the following sequence $(x_i)_{i \in
\omega}$ of numbers: let $x_0$ be such that $f(x_0) \rel{E} a$, and $x_{i+1}$
be such that $f(x_{i+1}) \rel{E} x_i$. Note that, as $f$ preserves
non-$E$-equivalence, the sequence of the $E$-equivalence classes of these
numbers is uniquely determined, although the choice of $x_i$ itself is not.
We now distinguish two cases for the sequence $(x_i)_{i \in \omega}$, and see
that either case leads to a contradiction, so no computable function $f$
reducing $E$ to $R$ exists.

\medskip

Case 1:  $(\forall i)[x_i \cancel{\rel{E}} a]$ or $(\forall i)[x_i
\cancel{\rel{E}} b]$.  We show that this case implies that $E$ is light, a
contradiction. Assume first that $x_i \cancel{\rel{E}} a$ for all $i$. By
induction on $n$ we show the following claim: $f^{(n)}(a) \cancel{\rel{E}}
x_{i}$ for all $i$, and if $n>0$ then $f^{(n)}(a) \cancel{\rel{E}} a$. For
$n=0$ this is trivial, as we already assume that $f^{(n)}(a)=a
\cancel{\rel{E}} x_{i}$ for all $i$. If $n=1$ then $f(a) \cancel{\rel{E}} a$
as otherwise $f(a) \rel{E} a \rel{E} f(x_0)$ giving that $a \rel{E}x_0$ (by
uniqueness of the $E$-class of $x_0$), contrary to the assumptions; if $f(a)
\rel{E} x_i$ then, by uniqueness of the $E$-class of $x_{i+1}$ we have
$x_{i+1} \rel{E} a$,contrary to the assumptions. So assume that the claim is
true of $n\geq 1$: if $f^{(n+1)}(a)= f(f^{(n)}(a)) \rel{E} a$, then again
$f^{(n)}(a) \rel{E} x_{0}$, contrary to the inductive hypothesis; if
$f^{(n+1)}(a)= f(f^{(n)}(a)) \rel{E} x_i$ then again $f^{(n)}(a) \rel{E}
x_{i+1}$, contrary to the inductive hypothesis. So, in particular, for every
$n>0$, $f^{(n)}(a) \cancel{\rel{E}} a$, and as $f$ preserves
non-$E$-equivalence, for every $k$ and $n \ge 1$, $f^{(n+k)}(a)
\cancel{\rel{E}} f^{(k)}(a)$, showing that the c.e. set $\{f^{(m)}(a) \mid
m\in \omega\}$ is comprised of $E$-inequivalent numbers, so that $E$ would be
light, a contradiction.

Now, we can run the same argument for $b$: so assume that $x_i
\cancel{\rel{E}} b$ for all $i$. Note that if we used $b$ to define the
sequence $(x_i)_{i \in \omega}$ (i.e. we start with $x_0$ be such that
$f(x_0) \rel{E} b$) then we get precisely the same sequence of $E$-classes
$[x_{i}]_{E}$. This is because $f$ is a reduction of $E$ to $R$ and $a
\rel{R} b$ so the only $E$-class that can be sent to $[b]_{E}$ (thus sent to
$[b]_{R}$) is the same $[x_0]_{E}$ as before. Therefore precisely the same
argument as above again shows that $E$ is light assuming no $x_i$ is
$E$-equivalent to $b$.
\medskip

Case 2: Otherwise. Since Case 1) does not hold, let $i,j$ be least so that
$x_i \rel{E} a$ and $x_j \rel{E} b$, respectively. Obviously $i\ne j$ because
$a \cancel{\rel{E}} b$. So suppose $i<j$. From $x_i \rel{E} a$ we have
$f(x_{i+1}) \rel{E} x_i \rel{E} a$, thus $x_{i+1}\rel{E} x_0$, and for all
$k$, $x_{i+k+1} \rel{E} {x_k}$ (again by uniqueness of the $E$-classes of the
$x$-sequence). Now, let $k$ be such that $j=i+k+1$: it then follows that
$x_j=x_{i+k+1} \rel{E} x_k$, thus $b\rel{E} x_k$ with $k<j$ contradicting
minimality of $j$.

A similar argument applies if we assume that $j<i$.
\end{proof}

In fact, Lemma~\ref{darkQuotients} leads to another characterization of the
dark ceers (and consequently of the light ceers):

\begin{thm}
A ceer $E$ is light if and only if there is a non-trivial quotient $R$ of $E$
so that $R$ is universal.
\end{thm}

\begin{proof}
Suppose $E$ is light. Let $Y=\{y_i\mid i\in\omega\}$ be the range of the
reduction $\Id\leq E$. Let $A$ be a universal ceer. Then let $X$ be the set
of pairs $\{(y_i,y_j)\mid i \rel{A} j\}$. Then $A\leq E_{/X}$ via the reduction
$i\mapsto y_i$.

Lemma \ref{darkQuotients} gives the reverse direction, as the case of an $E$
with finitely many classes is trivial.
\end{proof}

One approximation to Lemma \ref{darkQuotients} for self-full ceers is the
following:

\begin{lem}\label{lem:quotionts-self-fulls}
If $[x_i]_E$ for $1\leq i\leq n$ are computable $E$-classes and $E$ is
self-full, then $E\not\leq E_{/\{(x_i,y_i)\mid 1\leq i\leq n\}}$ for any
tuple of $y_i$, unless $y_i \rel{E} x_i$ for all $1 \leq i\leq n$.
\end{lem}

\begin{proof}
If $y_i \cancel{\rel{E}} x_i$ for some $1 \leq i\leq n$ then by Lemma
\ref{ComputableCollapses} $E_{/\{(x_i,y_i)\mid 1 \leq i\leq n\}} \oplus
\Id_k\equiv E$ for some $1\le k\leq n$. Thus, if $E\leq E_{/\{(x_i,y_i)\mid 1
\leq i\leq n\}}$, we would see that $E\oplus \Id_1\leq E_{/\{(x_i,y_i)\mid 1
\leq i\leq n\}}\oplus \Id_1\leq E$, contradicting self-fullness of $E$.
\end{proof}

In view of Lemma~\ref{darksAreSelfFull} one might hope that darkness and
self-fullness coincide. The following theorem shows that this is very much
not the case. In fact, it shows that every non-universal ceer is bounded by a
self-full ceer, and even there are self-full strong minimal covers of $\Id$,
as easily follows from the theorem by taking $A=\Id$.

\begin{thm}\label{SelfFullMinimalCoversInI}
Let $A$ be any non-universal ceer. Then there are infinitely many
incomparable self-full ceers $(E_l)_{l \in \omega}$ so that for every $n,l\in
\omega$ and ceer $X$, $A\oplus \Id_n \leq E_l$  and
\[
X< E_l \Rightarrow (\exists k)[X\leq A\oplus \Id_k].
\]
Moreover, each $E_l$ yields a partition in computably inseparable equivalence
classes, and thus each equivalence class is non-computable.
\end{thm}

\begin{proof}
We build an infinite sequence $(E_l)_{l \in \omega}$ of ceers so that in each
$E_l$ the elements of $W=\omega^{[0]}(=\{\langle 0, x \rangle\mid x \in
\omega\}$) are set aside for coding $A$. That is, we collapse $\langle
0,x\rangle$ and $\langle 0,y\rangle$ in $E_l$ if and only if $x \rel{A} y$.
This guarantees that $A\leq E_l$. We have the further requirements:
\begin{itemize}
  \item[$P^k_{i,j}$:] If $W_i$ intersects infinitely many $E_k$-classes
      disjoint from $W$, then $W_i$ intersects the class $[j]_{E_k}$.
  \item[$Q^{k,l}_i$:] if $k \ne l$ then $\phi_i$ is not a reduction from
      $E_k$ to $E_l$.
 \item[$S^k_{i,j,l,l'}$:] If $[i]_{E_k}\neq [j]_{E_k}$ and $W_l$ and
     $W_{l'}$ are complements, then the classes of $i$ and $j$ are not
     separated by $W_l$ and $W_{l'}$.
\end{itemize}

We first note that satisfying these requirements will build the ceers $E_l$
as needed. These $E_l$ are clearly incomparable by the $Q$-requirements.
Pairs of distinct $E_l$-equivalence classes are computably inseparable by the
$S$-requirements. It is clear that $A\leq E_l$ for each $l$ and since the
classes are computably inseparable we see that there are infinitely many
$E_l$-equivalence classes disjoint from $W$: indeed let $V$ be the complement
of $[W]_{E_{l}}$. Now, if there were only finitely many $E_l$-equivalence
classes disjoint from $W$ then $V$ would be c.e. as well; in this case if $V
\ne \emptyset$ then we would have a computable separation of some pair of
equivalence classes, whereas if $V =\emptyset$ then we would have by
Lemma~\ref{ontoReductionsFlip} that $E_l\le A$, giving $E_l \le E_{l'}$ for
every $l'$ contradicting incomparability of the ceers $E_k$. Hence  (by
Lemma~\ref{obs:plus-span}, as for every $n\in \omega$ the reduction $A\leq
E_l$ misses at least $n$ $E_{l}$-equivalence classes, and it is enough to
take as $U$ a set consisting of $n$ representatives from $n$ distinct
equivalence classes not in the range of $f$) every $A\oplus \Id_n$ reduces to
each $E_l$. If $X\leq E_l$ via $f$, then either the range of $f$ intersects
every class, thus $E_l\leq X$ by Lemma~\ref{ontoReductionsFlip}, or it
intersects only finitely many ($n$, say) classes disjoint from $W$ (this
follows by satisfaction of all requirements $P^l_{i,j}$, $j \in \omega$,
where $W_i=\im(f)$). In this latter case we have $X\leq A\oplus \Id_n$ for
some $n$: this follows from Lemma~\ref{coproduct3}(1) (by taking as $W$ our
$W=\omega^{[0]}$ and as $U$ a set of $n$ representatives for the $n$ classes
not intersecting $W$: in fact we have $X\leq E\oplus \Id_n$ for some
restriction $E$ which is reducible to $E_l$ via a reduction whose image takes
only $E_l$-equivalence classes of $W$, in which we code $A$, but our coding
is onto the $E_l$-equivalence classes of $W$, thus $E \leq A$ by an argument
similar to that of Lemma~\ref{ontoReductionsFlip}: given a number $x$ use the
reduction $E \leq E_l$ to get an image $y$, then search for the first $z$
such that $\langle 0,z\rangle \rel{E_l} y$; it is clear that the assignment
$x \mapsto z$ gives a reduction $E \leq A$). It remains to see that each
$E_l$ is self-full. Suppose $f$ is a reduction of $E_l$ to itself. We first
note that $f$ must hit infinitely many classes not containing an element of
$W$. Otherwise, by an argument similar to the previous one, we would have
$E_l\leq A\oplus \Id_n$ for some $n$, but this contradicts incomparability of
the $E_l$'s, since $A\oplus \Id_n\leq E_{l'}$ for any $l'\neq l$. Thus the
range of $f$ intersects infinitely many $E_l$-classes not equivalent to any
element of $W$. Thus satisfying the requirements $P^l_{i,j}$, $j \in \omega$,
guarantees that the range of $f$ intersects every class, showing that $E_l$
is self-full.

\medskip

We allow strategies to place restraints of the form $(a,b,l)$ declaring that
the pair $a$ and $b$ cannot be $E_l$-collapsed by a lower priority
requirement, or restraints of the form $(a,l)$ which declare that $a$ cannot
be $E_l$-collapsed to any element of $W$ by a lower priority requirement.
Each strategy will place only finitely many restraints and must be able to
succeed despite inheriting finitely much such restraint from higher priority
requirements.

\medskip

\emph{Strategy for $P^k_{i,j}$:} If $W_i$ is still disjoint from $[j]_{E_k}$,
then given finite restraint, $P^k_{i,j}$ waits for $W_i$ to enumerate an
element not currently equivalent to any element of $W$ and is not restrained
from $E_k$-collapsing with $j$ by any higher priority requirement. Then
\emph{$P^k_{i,j}$ is ready to act} and (\emph{$P^k_{i,j}$-action}) it causes
this element to $E_k$-collapse to $j$. Note that if $W_i$ intersects
infinitely many classes disjoint from $W$, then it will eventually list some
$x$ which is not restrained from being $E_k$-collapsed to $j$. We thus have
two possible outcomes: a waiting one, and another one provided by the action
of the strategy.

\medskip

\emph{Strategies for $Q^{k,l}_i$:} This strategy is essentially a combination
of the direct diagonalization strategy with the \basta strategy. We would
like to employ the direct diagonalization strategy on parameters $x,y$ as in
Theorem~\ref{MinimalDark}, but $\phi_i(x),\phi_i(y)$ may both already be
equivalent to numbers in $W$. In this case, we cannot know whether or not
they will collapse in the future, as this is controlled by $A$ and not by us,
thus we cannot count on diagonalization by collapsing $x$ with $y$. Thus we
first employ the \basta strategy coding a universal ceer $T$ into $E_k$ to
get either diagonalization (i.e. at some stage we are stuck with $\phi_i$
failing to be a reduction from $E_k$ to $E_l$ on finitely many $a_0, \ldots,
a_m$), or some $\phi_i(x),\phi_i(y)$ not both in $[W]_{E_l}$, in which case
we are free to employ the usual direct diagonalization strategy. In
particular, as in the proof of Theorem~\ref{MinimalDark}, we use the \basta
strategy choosing \emph{parameters} $a_0,a_1,\ldots $ (when we choose $a_m$
we must choose it fresh, and not in the current $[W]_{E_l}$ and in addition
to the restraints $(a_i,a_m,k)$ we also place restraint $(a_m,k)$) until
either we have our explicit diagonalization produced by the \basta strategy,
i.e. we finally take a final winning sub-outcome $m$ saying that $\phi_i$
fails to converge or to be a reduction on all $a_0, a_1, \ldots, a_m$ (notice
that $Q^{k,l}_i$ may need to act several time as demanded by the \basta
strategy: these actions, for which the requirement has previously become
\emph{ready to act} will be called \emph{$Q^{k,l}_i$-actions of type $1$});
or there are $j_1,j_2$ so that not both $\phi_i(a_{j_1}),\phi_i(a_{j_2})$ are
in $[W]_{E_l}$. This must happen if the \basta strategy does not give a
diagonalization outcome, otherwise $i\mapsto a_i$ gives a reduction of the
universal ceer $T$ to $A$, which contradicts non-universality of $A$. In this
case, if already $\phi_i(a_{j_1}) \rel{E_{l}} \phi_i(a_{j_2})$ then we keep
the restraint $(a_{j_1}, a_{j_2}, k)$ (notice that $a_{j_1}, a_{j_2}$ are
still $E_k$-inequivalent as the \basta strategy may $E_k$-collapse them only
after $\phi_i$ converges on both of them); otherwise, if still
$\phi_i(a_{j_1})\cancel{E_{l}} \phi_i(a_{j_2})$, then it is \emph{ready to
act} and (\emph{$Q^{k,l}_i$-action of type $2$}) we $E_k$-collapse $a_{j_1}$
and $a_{j_2}$, and we place restraint $(\phi_i(a_{j_1}), \phi_i(a_{j_2}),l)$,
and also $(\phi_i(a_{j}),l)$ for each $j\in \{j_1, j_2\}$ such that
$\phi_i(a_{j}) \notin [W]_{E_l}$ (thus restraining them from unintentionally
collapsing with each other for the sake of coding $A$). The possible outcomes
of the strategy are the possible outcomes of the \basta strategy (including
infinite waits when we hit divergent computations) with a final winning
sub-outcome $m$ where $a_{m}$ is the last defined $a_{i}$; or diagonalization
as produced by the $Q^{k,l}_i$-action of type $2$.

\medskip

\emph{Strategy for $S^k_{i,j,l,l'}$:} Restrain a new element (the
\emph{parameter of $S^k_{i,j,l,l'}$}) $x$ from $E_k$-collapsing with $i$ or
$j$ (i.e. place restraints of the form $(i,x,k)$ and $(j,x,k)$) or with any
element of $W$ (i.e. we place a restraint of the form $(x,k)$). If $x$ is
enumerated into $W_l$, then $S^k_{i,j,l,l'}$ becomes \emph{ready to act} and
(\emph{$S^k_{i,j,l,l'}$-action}) it $E_k$-collapses $x$ with $j$. If it is
enumerated into $W_{l'}$ it becomes \emph{ready to act} and
(\emph{$S^k_{i,j,l,l'}$-action}) it $E_k$-collapses it with $i$.
In either
case, $W_l$ and $W_l'$ do not separate the $E_k$ classes of $i$ and $j$.
There are two possible outcomes: a waiting one (we wait for the parameter to
be enumerated), and another one produced by the $E_k$-collapse given by the
$S^k_{i,j,l,l'}$-action.

\medskip

By the previous discussion, all outcomes  clearly fulfill the respective
requirements. It is straightforward to see, via the finite injury machinery
that we can construct a sequence of ceers $(E_l)_{l \in \omega}$ satisfying
all requirements. We hint at the formal construction.

\medskip

\emph{The construction:} The construction is by stages. At each stage $s$ we
define the current values for the parameters relative the various
requirements, and we define $E_{l,s}$ for every $l$. If $R=Q^{k,l}_{i}$ then
$R$ has parameters $a^{R}_{0,s}, \ldots, a^{R}_{m-1,s}$ suitable to the
\basta strategy as in the proof of Theorem~\ref{MinimalDark}; if
$R=S^{k}_{i,j,l,l'}$ then $R$ has a parameter $x^{R}_{s}$. Moreover, each
requirement deals with the current restraint set $\rho^{R}$ which consists of
all triples $(a,b,l)$ imposed by higher priority requirements, plus the pairs
$(a,l)$ imposed by higher priority requirements: $R$ is not allowed to
$E_{l}$-collapse pairs $a$ and $b$ for which there is in $\rho^R$ a restraint
of the form $(a,b,l)$, nor is allowed to $E_{l}$-collapse to an element of
$W$ any number $a$ for which there is in $\rho^{R}$ a restraint $(a,l)$.

At odd stages we code $A$ into $W$ for every $E_l$: we assume to have fixed a
computable approximation $\{A_s: s \in \omega\}$ to $A$ as in
Section~\ref{ssct:background}. At even stage we perform the actions demanded
by the strategies relative to the various requirements. We fix as usual a
computable order of the requirements of order type $\omega$.
\emph{Initialization} of requirements different from $P$-requirements is as
in the proof of Theorem~\ref{MinimalDark}, of which we follow terminology and
notations unless otherwise specified. We say that a requirement $R$
\emph{requires attention at stage $s$}, if either $R=P^k_{i,j}$ for some
$k,i,j$, and $R$ has never acted so far, but now $R$ is ready to take
$R$-action as described; or $R \in \{Q_i^{k,l}, S^k_{i,j,l,l'}\mid i,j,k,l,l'
\in \omega\}$ and $R$ is initialized, or otherwise $R$ has not taken action
as yet after last initialization, but it is now ready to take $R$-action.

\medskip
\emph{Stage $0$.} Initialize all requirements $R$, except the
$P$-requirements. Define $E_{l,0}=\Id$.

\emph{Stage $2s+1$.} Code $A_s$ in every $E_l$, with $l\le s$, as described,
i.e. $E_l$-collapse all pairs $\langle 0,x\rangle$, $\langle 0,y\rangle$ if
and only if $x \rel{A_s} y$.

\medskip
\emph{Stage $2s$ with $s>0$.} Consider the least $R$ that requires attention
at $s+1$. (Since infinitely many strategies are initialized, such a least
strategy exists.)

If $R$ is a $P$-requirement, then simply take $R$-action. Otherwise,
\begin{itemize}
  \item if $R$ is initialized then choose new parameters for it as
      described in the strategy for $R$;
  \item if $R$ is not initialized then take $R$-action as described.
\end{itemize}
After $R$ has acted, initialize all requirements $R'$ of lower priority than
$R$, except the $P$-requirements.

At the end of the stage (whether odd or even), define $E_{l,s+1}$ to be the
equivalence relation generated by $E_{l,s}$ plus the pairs which have been
$E_{l}$-collapsed at stage $s+1$. Go to next stage.

\medskip \emph{Verification:}
Restraints against lower priority requirements are automatically preserved by
initializations, and after initialization when a strategy chooses new
parameters it chooses so that for every $l$ the parameters are neither
pairwise $E_l$-equivalent nor are they $E_l$-equivalent to any number so far
mentioned in the construction, so that they may be restrained as needed.

Since after re-initialization each requirement acts only finitely many times,
a straightforward inductive argument on the priority rank of the requirements
shows that each requirement is eventually not re-initialized, and so after
last re-initialization the relative strategy can act, if needed, to satisfy
the requirement as in the above description of strategies.
\end{proof}

The following is a very similar theorem to Theorem
\ref{SelfFullMinimalCoversInI}, where we request that the ceers $E_l$ all
have finite classes. Having finite classes in $E_l$ is incompatible with both
the computable inseparability requirement, and the requirement that $X< E_l$
implies that $X\leq A\oplus \Id_n$ for some $n$. To see that having finite
classes is incompatible with the latter requirement, notice that  if the
classes are computable then by Lemma~\ref{ComputableCollapses},
$X={E_l}_{/(x,y)}$ where $x$ and $y$ are any two non-$E_l$-equivalent
numbers, has the property that $X\oplus \Id_1\equiv E_l$. On the other hand,
if $E_l$ is self-full then by Lemma~\ref{lem:quotionts-self-fulls} $E_l \nleq
X$. Thus $X< E_l$ but we cannot have $X\leq A\oplus \Id_n$ unless also $E_l
\leq A\oplus \Id_{n+1}$. But we also have $A\oplus \Id_{n+2}\leq E_l$ (as
$A\oplus \Id_{k}\leq E_{l'}$ for all $k,l'$), hence $E_l \leq A\oplus
\Id_{n+1} \leq A \oplus \Id_{n+2} \leq E_l$, which gives a reduction of $E_l$
to itself whose target misses at least one equivalence class, contradicting
self-fullness.

\begin{thm}\label{FiniteClassesSelfFulMinimalCoversInI}
Let $A$ be any non-universal ceer with only finite classes. Then there are
infinitely many incomparable self-full ceers $(E_l)_{l \in \omega}$ so that
$A\oplus \Id_n\leq E_l$ for each $n,l$, and for every ceer $X$,
\[
X<E_l \Rightarrow (\exists n)[X\leq A\oplus \Id_n \;\lor\;
E_l\equiv X\oplus \Id_n].
\]
Further, each $E_l$ has finite classes.
\end{thm}

\begin{proof}
Once again, we construct the sequence of ceers $(E_l)_{l \in \omega}$ so that
each $E_l$ codes $A$ on the set $W=\omega^{[0]}$. We have further
requirements:

\begin{itemize}
 \item[$P^l_{i,j}$:] If $i\leq j$, $j$ is the least number in its
     $E_l$-class, and $W_i$ intersects infinitely many $E_l$-classes
     disjoint from $W$, then it intersects $[j]_{E_l}$.
 \item[$SF^l_{i,j}$:] If $\phi_i$ is a reduction of $E_l$ to itself, then
     $[j]_{E_l}\cap \im(\phi_i)\neq \emptyset$.
  \item[$Q^{k,l}_i$:] if $k \ne l$ then $\phi_i$ is not a reduction from
      $E_k$ to $E_l$.
 \item[$F^l_i$:] The $E_l$-class of $i$ is finite.
\end{itemize}

We computably order the requirements with order type $\omega$, so that each
$P^l_{i,j}$ with $i\leq j$ has higher priority than $F^l_j$.

\begin{lem}
If $(E_l)_{l\in \omega}$ is a sequence of ceers satisfying all of these
requirements, then each $E_l$ is self-full and $X<E_l$ implies that for some
$n$, either $X\leq A\oplus \Id_n$ or $E_l \equiv X\oplus \Id_n$.
\end{lem}

\begin{proof}
By the requirements $SF^l_{i,j}$ for every $j$, we see that if $\phi_i$ is a
reduction of $E_l$ to itself, then it is onto the classes of $E_l$, so each
$E_l$ is self-full. If $X\leq E_l$ via $f$, then let $V=\im(f)$. By the
requirements $P^l_{i,j}$, with $V=W_i$, either $V$ is contained in
$[W]_{E_l}$ along with at most finitely many other $E_l$-classes or it
intersects co-finitely many $E_l$-classes.
In the former
case, exactly as we have argued in the proof of
Theorem~\ref{SelfFullMinimalCoversInI}, it follows by
Lemma~\ref{coproduct3}(1) that $X\leq A\oplus \Id_n$ for some $n$ while in
the latter case, it follows from Lemma~\ref{obs:plus-span} that $X\oplus
\Id_n\equiv E_l$ where $n$ is the number of missed classes.
\end{proof}

In this construction, in addition to the usual restrains $(i,j,l)$ saying
that no lower priority requirement can $E_l$-collapse $i$ and $j$, and
$(i,l)$ saying that no lower priority requirement can $E_l$-collapse $i$ to
an element of $W$, we also allow restraints $(i,l)^F$, placed by an
$F$-requirement $F^l_i$, of the form: No lower priority requirement can
$E_l$-collapse anything with $i$.

\medskip

\emph{Strategy for $P^l_{i,j}$:} If $W_{i}$ is still disjoint from
$[j]_{E_{l}}$, we wait for some number $x$ to be enumerated into $W_i$ which
is not currently in $[W]_{E_l}$ and is not restrained from $E_l$-collapsing
with $j$ by any higher priority requirement. At this point $P^l_{i,j}$ is
\emph{ready to act} and we then (\emph{$P^l_{i,j}$-action}) $E_l$-collapse
$x$ with $j$. Notice that as $i \le j$ the strategy is not allowed to grow
equivalence classes of numbers $j' <i$ and thus does not conflict in this
case with the finiteness requirements $F^l_{j'}$ of higher priority. The
waiting outcome, and the outcome given by the action fulfill the
requirement.

\medskip

\emph{Strategy for $SF^l_{i,j}$:} We work under the assumption that still
$[j]_{E_l}\cap \im(\phi_i)= \emptyset$,  and $j$ is least in its $E_{l}$
equivalent class (otherwise the requirement is already satisfied: we say in
this case that it is \emph{inactive}). The strategy is as follows: We wait
for an $n$ to appear so that $[\phi_i^{(n)}(j)]_{E_l}$ is not restrained from
growing by a higher priority $F$-requirement. Then we wait for any $x$ so
that $\phi_i^{(n+1)}(x)$ converges, and $\phi_i^{(n+1)}(x)\notin [W]_{E_l}$
and is not restrained from $E_l$-collapsing with $\phi_i^{(n)}(j)$ by a
higher priority requirement. At this point $SF^l_{i,j}$ is \emph{ready to
act} and (\emph{$SF^l_{i,j}$-action}) we $E_l$-collapse $\phi_i^{(n)}(j)$
with $\phi_i^{(n+1)}(x)$. The outcomes can be described as follows: either
the requirement becomes eventually inactive and thus satisfied (this is the
case if we act and $\phi_i$ is a reduction, as the action guarantees that $j
\rel{E_l} \phi_i(x)$), or we get stuck in a waiting outcome, but then in this
case we will argue in Lemma~\ref{lem:SF-satisfied} that $\phi_i$ is not a
reduction.

\medskip

\emph{Strategy for $Q^{k,l}_i$:} This is exactly as in Theorem
\ref{SelfFullMinimalCoversInI}.

\medskip

\emph{Strategy for $F^l_i$:} This strategy simply restrains the $E_l$-class
of $i$ from being collapsed to any other class by a lower priority
requirement. Higher priority requirements will ignore this, but it will
succeed in guaranteeing that the $E_l$ class of $i$ is finite.

\medskip

\emph{The construction:} The $Q$-requirements have parameters as in the proof
of Theorem~\ref{SelfFullMinimalCoversInI}; the other requirements do not have
parameters, except for the  restraint sets $\rho^R$, which now contain also
pairs in the form $(i,l)^F$ imposed by the higher priority $F$-requirements.
Again, we employ the finite injury machinery to build a sequence of ceers
according to the strategies which we have described. A formal construction is
similar to that in the proof of Theorem~\ref{SelfFullMinimalCoversInI}, using
the initialization mechanism (which automatically guarantees restraints
against lower priority requirements: note that only $Q$-requirements need to
be re-initialized as they need to choose new parameters after action of
higher priority requirements, whereas the other requirements are immediately
satisfied once they act and never injured after then), and by specifying as
in that proof what it means for a requirement to require attention: namely,
not to have ever acted for $P$- and $SF$-requirements but being now ready to
act, or to be initialized or being ready to act (with the obvious meaning as
in Theorem~\ref{MinimalDark} and Theorem~\ref{SelfFullMinimalCoversInI})
after last initialization for $Q$-requirements. The $F$-requirements set once
and for all restraints, never act and never require attention. Initialization
is as in the proof of Theorem~\ref{SelfFullMinimalCoversInI}. At stage $0$ we
define $E_{l,0}=\Id$, we initialize all $Q$-requirements; if $R=F_i^l$ then
$R$ sets up the once for all restraint $(i,l)^F$  which will never be
cancelled so that at all stages, $(i,l)^F \in \rho^{R'}$ for all lower
priority $R'$.

The rest of the construction follows the pattern of the construction in the
proof of Theorem~\ref{SelfFullMinimalCoversInI}.

\medskip

\emph{The verification:} The verification is based on the following lemma:

\begin{lem}\label{lem:SF-satisfied}
Every requirement acts only finitely often. Thus it is re-initialized only
finitely often and is eventually satisfied.
\end{lem}

\begin{proof}
Suppose not, and let $R$ be the highest priority requirement where this
fails. So there is a last stage at which higher priority strategies stop
re-initializing $R$. Note that if $R$ is a $P$- or an $SF$-requirement then
it acts at most once, and $F$-requirements never act. Also, a $Q$-requirement
cannot act infinitely often (recall, the \basta strategy acts only finitely
often: see proof of Theorem~\ref{SelfFullMinimalCoversInI}). Thus $R$  acts
only finitely often as well, after last re-initialization.

As to satisfaction,  we note that every $F^l_i$-requirement is satisfied
(regardless of the satisfaction of other requirements): this is because every
higher priority requirement acts only finitely often, thus there can only be
finitely many times when $[i]_{E_l}$ grows. Similarly, every $Q$-requirement
and $P$-requirement is satisfied (again regardless of the satisfaction of
even higher priority $SF$-requirements): indeed, $Q$-requirements are
initialized only finitely often by action of higher priority requirements,
and after then they are free to pursue their strategies towards satisfaction;
as to a requirement $P_{i,j}^{l}$, if $W_{i}$ intersects infinitely many
$E_{l}$-classes disjoint from $W$, then it eventually enumerates a number
free of higher priority restraint (since this restraint is finite) so that it
can act to satisfy itself. Thus, we must consider the case that $R$ is a
self-fullness requirement $R=SF^l_{i,j}$. We suppose that in fact $j$ is the
least member of its $E_l$-equivalence class, $[j]_{E_l}\cap \im(\phi_i)=
\emptyset$, and $\phi_i$ is a reduction of $E_l$ to itself (so for every
$n\ge 1$, $\phi_i^{(n)}$ is a reduction as well). We will show that this
leads to a contradiction. By the argument from the proof of Lemma
\ref{darksAreSelfFull}, we know that $\{\phi^{(n)}(j)\mid n \in \omega\}$ is
comprised of non-$E_l$-equivalent elements. Thus, one of these
$\phi^{(n)}(j)$ is not equivalent to any $i$ whose class is restrained from
growing by a higher priority $F$-requirement. We now have two cases: If
$\im(\phi_i^{(n+1)})$ is contained in $[W]_{E_l}$ along with finitely many
other classes, then by Lemma~\ref{coproduct3}(1), as argued in the proof of
Theorem~\ref{SelfFullMinimalCoversInI}, we have that $\phi_i^{(n+1)}$ gives a
reduction of $E_l$ to $A\oplus \Id_m$, for some $m$. But we know that
$A\oplus \Id_m \leq E_{l'}$ for every $l'\neq l$ and that $E_l$ is
incomparable with $E_{l'}$. This yields a contradiction. On the other hand,
if $\im(\phi_i^{(n+1)})$ is not contained in $[W]_{E_l}$ along with finitely
many other classes, then there will eventually appear some $y\in
\im(\phi_i^{(n+1)})$ which is not restrained from collapsing with
$\phi_i^{(n)}(j)$, contradicting the assumption that $[j]_{E_l}\cap
\im(\phi_i)= \emptyset$.

Thus every requirement succeeds.
\end{proof}
This completes the proof of the theorem.
\end{proof}

\begin{cory}\label{LightSelfFulls}
There are infinitely many incomparable light self-full ceers $E_l$ with
finite classes so that $X<E_l$ implies $X\leq \Id$ or $X\oplus \Id_n\equiv
E_l$ for some $n$.
\end{cory}

\begin{proof}
In Theorem~\ref{FiniteClassesSelfFulMinimalCoversInI} let $A=\Id$. Note that
$\Id\oplus \Id_n\equiv \Id$ for every $n\ge 1$.
\end{proof}

\begin{cory}\label{cory:selfdarkminimalcovers}
If in Theorem~\ref{SelfFullMinimalCoversInI} or in
Theorem~\ref{FiniteClassesSelfFulMinimalCoversInI} we start with $A$ dark,
then the ceers $E_{l}$ are dark as well.
\end{cory}

\begin{proof}
Consider the case of Theorem~\ref{SelfFullMinimalCoversInI}, and
suppose $E_{l}$ is light. Then $X=\Id \leq E_{l}$ but then $\Id \leq A\oplus
\Id_{n}$ for some $n$ and some reducing function $f$. It follows that
$\{\frac{f(x)}{2} \mid f(x) \textrm{ even}\}$ is an infinite effective
transversal of $A$, contradicting $A$ dark.

In the case of Theorem~\ref{FiniteClassesSelfFulMinimalCoversInI}, we have an
additional case to consider, which is impossible by
Corollary~\ref{cor:belowId}.
\end{proof}

Note that Theorem \ref{SelfFullMinimalCoversInI} shows that the collection of
self-full ceers is upward-dense. Not to be out-done in their upward-density,
the non-self-full ceers are also upward dense:

\begin{obs}
Let $E$ be any non-universal ceer. Then there is a non-universal ceer $X\geq
E$ so that $X$ is non-self-full.
\end{obs}

\begin{proof}
Let $X=E\oplus \Id$. It is immediate that $X\oplus \Id_1\leq X$, so $X$ is
not self-full. Since the universal degree is uniform join-irreducible (see
Corollary~\ref{cor:universal-join-irreducible}), we have that $X$ is not
universal.
\end{proof}

It follows that, unlike darkness, neither self-fullness nor non-self-fullness
is closed downwards or upwards. On the other hand, neither self-fullness nor
non-self-fullness (for non finite ceers) ceers are downward dense, as is
shown by the following theorem.

\begin{thm}
For $X$ any dark ceer, every $E\leq X$ is self-full. Also, there are light
ceers $Y$ so that $E\leq Y$ implies that either $E$ has only finitely many
classes, or $E$ is non-self-full.
\end{thm}

\begin{proof}
If $X$ is dark then every $E\leq X$ is dark, thus self-full, so the former
claim follows immediately. As to the latter claim, consider $Y$ to be the
ceer $\Id$, or $Y=R_B$, where $B$ is a non-decidable non-simple c.e. set
having minimal $1$-degree (i.e. such that, for any non-decidable c.e. set
$A$, if $A\le_1 B$ then $B\le_1 A)$. For the existence of such a c.e. set
$B$, due to Lachlan, see~\cite{Lachlan2} or Odifreddi's survey
paper~\cite[Theorem~5.6]{Odifreddi:survey}. Indeed, if $Y=R_B$ is as above
then let $U\subseteq \overline{B}$ be an infinite computable set; as $U$ is
a strong effective transversal of $Y$,
then by
Lemma~\ref{obs:coproduct2} we have that $Y\oplus \Id \le Y$ so that $Y\oplus
\Id_1 \le Y$, which implies that $Y$ is non-self-full by
Observation~\ref{obs:nonselffull}. Finally, if $E \le Y$, then by
Fact~\ref{fact:fund-ce-sets} $E\equiv R_A$ for some c.e. set $A$: then by
minimality of the $1$-degree of $B$, either $B \equiv_1 A$ so that $E\equiv
Y$ and thus $E$ is non-self-full, or $A$ is decidable, giving $E\equiv \Id_n$
for some $n \le \omega$. Lastly, $\Id$ is non-self-full.
\end{proof}

We know that the self-full degrees are not closed downwards, and we ask:

\begin{q}
If $E_1$ and $E_2$ are both self-full, must $E_1\oplus E_2$ be self-full?
\end{q}

\section{Least Upper Bounds}\label{sct:lubs}
We now examine the general question of when pairs of ceers have a least upper
bound. This examination will be continued in Section~\ref{overI}: in fact
some of the results proved now are consequences of more general results in
next section: we keep them here for the sake of clearer exposition and easier
reference.

Important instances of existing least upper bounds are provided by the
following observation, which will also be employed in a later section to
prove definability results.

\begin{obs}\label{joinIdDark}
$\Id$ has a least upper bound with any dark degree, in fact for $X$ dark, the
degree of $\Id \oplus X$ is the least light degree bounding $X$.
\end{obs}

\begin{proof}
Let $X$ be dark and $Y$ such that $\Id\leq Y$ and $X\leq Y$. Fix $f$ and $g$
to be reductions witnessing this. Then the intersection of the $Y$-closures
of the images of $f$ and $g$ is contained in finitely many $Y$ classes, as
otherwise $X$ would be light. Let $Z$ be the finite set of elements $z$ so
that $f(z)\in \im(g)$. Then let $h$ be a reduction of $\Id$ to $\Id$ with no
element of $Z$ in the range. It follows by Lemma~\ref{obs:coproduct}, by
taking $f_1=f\circ h$, $f_2=g$ and $W=\emptyset$, that $\Id \oplus X \le Y$.
\end{proof}

\subsection{The ceers of the form $R_X$}\label{ssct:RX}
We begin by considering the ceers of the form $R_X$ for a c.e. set $X$, then
consider the more general case.

\begin{thm}\label{NoJoinOf1Degrees}
Let $X$ and $Y$ be infinite c.e. sets so that $X\not\leq_1 Y\oplus \emptyset$
and $Y\not\leq_1 X\oplus \emptyset$. Then $R_X$ and $R_Y$ do not have any
least upper bound in the ceers.
\end{thm}

\begin{proof}
Note that $R_{X\oplus Y}$ is an upper bound for both $R_X$ and $R_Y$. Since
the degrees of ceers of the form $R_Z$ form an interval in $(\Ceers, \leq)$
(Fact~\ref{fact:fund-ce-sets}), we know that any least upper bound must be
equivalent to a ceer of this form. Suppose $R_Z$ (with $Z$ c.e. and
undecidable) was such a least upper bound. But then $R_X\oplus R_Y$ is also
an upper bound. So, $R_Z\leq R_X\oplus R_Y$. But then, by undecidability,
either the infinite class in $R_Z$ is sent to the infinite even class, in
which case we have that $R_Z\leq R_X\oplus \Id$ by the same reducing
function, or the infinite odd class, in which case (with a similar argument)
$R_Z \leq \Id\oplus R_Y \equiv R_Y \oplus \Id$. Without loss of generality,
we suppose we have $R_Y\leq R_Z\leq R_X\oplus \Id$. But $R_X\oplus \Id=
R_{X\oplus \emptyset}$, so we get that $Y\leq_1 X\oplus \emptyset$, which is
a contradiction.
\end{proof}

\begin{cory}\label{cor:lightNoJoin}
If $R_X$ and $R_Y$ are light and incomparable, then they do not have any
least upper bound.
\end{cory}

\begin{proof}
Suppose that $R_X$ and $R_Y$ are light and incomparable. Since $X$ is not
simple there exists an infinite computable set in the complement, which
consists of pairwise $R_X$-inequivalent elements, and thus forms an infinite
strong effective transversal for $R_X$.
Thus by Lemma~\ref{obs:coproduct2}, $R_X \oplus \Id \equiv R_X$. But
$R_X\oplus \Id= R_{X\oplus \emptyset}$, thus $R_{X\oplus \emptyset}\equiv
R_X$; similarly for $R_Y$. Thus, if $R_X$ and $R_Y$ are incomparable, it
follows that $X\not\leq_1 Y\oplus \emptyset$ and $Y\not\leq_1 X\oplus
\emptyset$. The claim then follows from the previous theorem.
\end{proof}

\begin{cory}\label{darkRNoJoin}
If $R_X$ and $R_Y$ are dark and incomparable, then they do not have any least
upper bound.
\end{cory}

\begin{proof}
Suppose that $R_X$ and $R_Y$ have a least upper bound: by
Fact~\ref{fact:fund-ce-sets} it has to be of the form $R_Z$ for some c.e. set
$Z$. Then let $f$ be a reduction witnessing $R_Z\leq R_X\oplus R_Y$. Without
loss of generality, the infinite class of $Z$ gets sent to the infinite class
corresponding to $X$. Then the odd images of $f$ give a set of pairwise
non-$R_Y$-equivalent elements and thus from the darkness of $R_Y$ the image
of $f$ contains only finitely many odd numbers. It follows that $R_Z\leq
R_X\oplus \Id_n$ for some $n$, by Lemma~\ref{coproduct3}(1) (taking $W$ to be
the even numbers and $U$ the finitely many odd numbers in the image of $f$).
Thus $R_Y\leq R_X\oplus \Id_n$. It follows from repeated applications of
Lemma \ref{StrongMinimalCoversOfSF} that either $R_Y\equiv R_X\oplus \Id_k$
for some $k$, or $R_Y\leq R_X$, either way contradicting the assumption of
incomparability of $R_X$ and $R_Y$.
\end{proof}

\begin{lemma}\label{lem:Turing-simple}
There are infinitely many simple sets $(A_{i})_{i \in \omega}$ so that for
distinct $i,j$, $A_{i} \nleq_{T} A_{j}$ and $A_{j} \nleq_{T} A_{i}$, where
$\leq_{T}$ denotes Turing reducibility.
\end{lemma}

\begin{proof}
Folklore.
\end{proof}

\begin{cory}\label{cor:there-are-pairs}
There are pairs of light degrees, pairs of dark degrees, and pairs consisting
of one light and one dark degree (all containing ceers of the form $R_{X}$)
which have no least upper bound.
\end{cory}

\begin{proof}
The first two claims follow easily from Corollary~\ref{cor:lightNoJoin} and
Corollary~\ref{darkRNoJoin} respectively. All three claims however follow
also from Theorem~\ref{NoJoinOf1Degrees}: indeed by
Lemma~\ref{lem:Turing-simple} one can find pairs of simple c.e. sets $U,V$
such that $U |_T V$. Then: $R_U, R_V$ are dark; $R_{2U}, R_{2V}$ are light,
and $R_U$ is dark and $R_{2V}$ is light. In each case the assumptions of
Theorem~\ref{NoJoinOf1Degrees} are fulfilled by Turing incomparability of $U$
and $V$: thus for instance $U \not\leq_1 2V \oplus \emptyset$ and $2V
\not\leq_1 U \oplus \emptyset$, etc.
\end{proof}

If we take $X$ simple then $R_X$ (which is dark) and $\Id$ are incomparable,
and by Observation~\ref{joinIdDark} $R_X \oplus \Id$ is a least upper bound
of $R_X$ and $\Id$, but $\Id=R_\emptyset$ and $R_X \oplus \Id=R_{X\oplus
\emptyset}$. This gives examples of cases when a least upper bound for pairs
of ceers of the form $R_{Z}$ exists, and is of this form. More generally we
can prove the following observation.

\begin{obs}\label{obs:sometimes-R-yes}
If $R_X$ is dark and $R_Y$ is light and $Y\leq_1 X\oplus \emptyset$, then
$R_{X\oplus \emptyset}$ is a least upper bound of $R_X$ and $R_Y$.
\end{obs}

\begin{proof}
It is immediate that $R_{X\oplus \emptyset}$ is an upper bound for $R_X$ and
$R_Y$. But $R_{X\oplus \emptyset}= R_X\oplus \Id$, which is a least upper
bound for $R_X$ and $\Id$ by Observation \ref{joinIdDark}. Since $R_Y\geq
\Id$, it follows that $R_X\oplus \Id$ is also a least upper bound of $R_X$
and $R_Y$.
\end{proof}

\subsection{No least upper bounds for dark pairs in $\Ceers$}
We now extend the result in Corollary \ref{darkRNoJoin} to the general case
of two dark ceers.

\begin{thm}\label{NoJoinOfDark}
If $E_1$ and $E_2$ are incomparable dark ceers then they have no least upper
bound.
\end{thm}

\begin{proof}
Suppose $R$ is a least upper bound of $E_1$ and $E_2$, with $E_1, E_2$ dark.
It follows that $R$ is dark as well, since $R\leq E_1\oplus E_2$ and $E_1
\oplus E_2$ is dark by Observation~\ref{obs:dark-closure}. Let $f_i$ be a
reduction of $E_i$ to $R$.

Let $F_1\subseteq \omega_{/E_1}$ be the collection of classes $[x]_{E_1}$ so
that $f_1(x)Rf_2(y)$ for some $y$. Similarly, let $F_2$ be the collection of
classes $[y]_{E_2}$ so that $f_2(y) R f_1(x)$ for some $x$.

Firstly, suppose $F_1$ contains every class of $E_1$. Then consider the map
$
x\mapsto \text{ the first seen }y \text{ such that }f_1(x)R f_2(y).
$
This gives a
reduction of $E_1$ to $E_2$. Similarly, we see that $F_2$ cannot contain
every class of $E_2$.

Now, let $[x]_{E_1}$ be a class not in $F_1$ and $[y]_{E_2}$ be a class not
in $F_2$. Then consider the proper quotient $R_{/(f_{1}(x),f_{2}(y))}$. Then
$E_1$ and $E_2$ are both reducible to this ceer, but by Lemma
\ref{darkQuotients} $R$ is not.
\end{proof}

It may seem from the proof of this theorem that two incomparable dark ceers
may yet be fairly close to having a least upper bound. In fact, the
obstruction to having a least upper bound which is used in this proof is the
only obstruction. This is shown below in Corollary \ref{cor:dark-with-supI}.

We see below in Corollary \ref{DarkLightJoin} that more pairs (in addition to
those in which one component is $\Id$) with exactly one dark ceer have least
upper bounds, and we see below in Corollary \ref{SometimesJoinOfLight} that
some pairs of light ceers have a least upper bound. Working towards these
results, we find it convenient to first introduce  in Section \ref{overI} the
idea of working over $\I$.

\subsection{Join-irreducibility in $\Ceers$}
Dark ceers of the form $R_{X}$ are join-irreducible:

\begin{cory}\label{JoinIrreducible1Degrees}
If $R$ is a dark ceer of the form $R_X$, then $R$ is join-irreducible.
\end{cory}

\begin{proof}
Any ceer $\leq R$ must also have up to equivalence the form $R_Y$ and must be
dark, or with finitely many classes. So if $R_X$ is a proper join it must be
a join of two incomparable dark ceers of the form $R_Y$ and $R_Z$ but by
Corollary~\ref{darkRNoJoin} no two incomparable dark ceers of this form have
a join, which implies that $R$ cannot be a join of two ceers $<R$
\end{proof}

As the infinite ceers below a dark one are dark, the above argument can be
generalized to show:

\begin{cory}\label{JoinIrreducible1Degreesbis}
Every dark ceer  is join-irreducible.
\end{cory}

\begin{proof}
Immediate.
\end{proof}

Other examples of join-irreducible ceers (pointed out in \cite{jumpsofceers})
are ceers which are equivalent to a jump (so including also the universal
ceers: see also Theorem~\ref{thm:join-irreducible-in-I} below).

\section{Working over $\I$}\label{overI}
We define the relation $A\leq_\I B$ if there is some $n$ so that $A\leq
B\oplus \Id_n$. It is easily seen that the relation is a pre-ordering
relation, i.e. reflexive and transitive. The equivalence relation induced by
$\leq_\I$ will be denoted by $\equiv_\I$. In the usual way we get a degree
structure, the $\I$-degrees: we denote by $(\Ceers_{\I}, \leq_{I})$ the
poset of $\I$-degrees.

\begin{obs}\label{jumpUpInI}
For any non-universal ceer $X$, $X<_\I X'$.
\end{obs}
\begin{proof}
Suppose that $X'\leq_{\I} X$. Then $X'\leq X\oplus\Id_n$ for some $n$. But since every jump is uniform join-irreducible by Fact \ref{JumpsUniformJoinIrreducible}, we then get that $X'\leq X$ or $X'\leq \Id_n$. The former is only possible if $X$ is universal \cite{ceers}, while the latter is impossible since every jump has infinitely many classes.
\end{proof}

\subsection{Darkness, and working over $\mathcal{I}$}
On dark ceers the equivalence relation $\equiv_\I$ can be characterized as
follows.

\begin{thm}\label{thm:characterization-of-equiI}
Let $A$ and $B$ be dark ceers. Then the following conditions are equivalent:

\begin{enumerate}
\item $A\equiv_\I B$
\item $A\oplus \Id\equiv B\oplus \Id$
\item There is an $n$ so that either $A\equiv B\oplus \Id_n$ or $B\equiv
    A\oplus \Id_n$.
\end{enumerate}
Furthermore, (1) and (3) are equivalent for any ceers.
\end{thm}

\begin{proof}
(1) implies (2): If $A\equiv_\I B$, then we have $A\leq B\oplus \Id_n$ for
some $n$. Thus
\[
A\oplus \Id \leq (B\oplus \Id_n )\oplus \Id \equiv B\oplus (\Id_n  \oplus \Id)
\equiv B\oplus \Id.
\]
Symmetrically, $B\oplus \Id\leq A\oplus \Id$.

(1) implies (3): If $A\equiv_\I B$, then $A\leq B\oplus \Id_n$ for some $n$
and $B\leq A\oplus \Id_m$ for some $m$. Let $n$ and $m$ be least so that this
is true. If $n\geq 1$, then Lemma 4.5 shows that $A\equiv B\oplus \Id_n$.
Similarly, if $m\geq 1$, then $B\equiv A\oplus \Id_m$. If $n=m=0$, then
$A\equiv B$.

(2) implies (1): Now, suppose $A\oplus \Id\equiv B\oplus \Id$. Then there is
a reduction $f$ witnessing $A\leq B\oplus \Id$. Consider the set of odd
elements of the range of this reduction. This is a c.e. set $Z$, and for each
$z\in Z$, we let $a_z$ be any element so $f(a_z)=z$. Then $\{a_z\mid z\in
Z\}$ is an effective transversal of $A$. Since $A$ is dark, it follows that
$Z$ is finite, so by Lemma~\ref{coproduct3}(1) (taking $W$ to be the even
numbers, and $U=Z$) $A\leq B\oplus \Id_m$ where $m$ is the number of elements
of $Z$ so that $m$ is least for this reduction: let $g$ witness this
reduction. Similarly one shows that $B\leq A \oplus \Id_k$ for some $k$.

(3) implies (1): Clear from the definition.

Notice that darkness of $A, B$ is used in the above proof only for the
implication (2)$\Rightarrow$(1).
\end{proof}

Notice also that if $(x_1,y_1), \ldots, (x_n, y_n)$, $n \ge 1$, are pairs so
that each $[x_i]_E$ is computable and at least one pair consists of
$E$-inequivalent numbers, then by Lemma~\ref{ComputableCollapses} and
Theorem~\ref{thm:characterization-of-equiI} we have that $E_{/\{(x_i,y_i)\mid
1\le i\le n\}} \equiv_\I E$.

\begin{obs}
The dark ceers are downwards closed and closed under $\oplus$ in the
$\mathcal{I}$-degrees.
\end{obs}

\begin{proof}
Clear from definitions.
\end{proof}

\begin{obs}\label{obs:Icontiguous}
If $E$ is non-self-full, then the $\equiv_\I$-class of $E$ is comprised of
exactly the $\equiv$-class of $E$. In particular the universal
$\equiv_\I$-degree consists of the universal ceers.
\end{obs}

\begin{proof}
If $R\equiv_\I E$, then $R\leq E\oplus \Id_n$ for some $n$, but $E\oplus
\Id_n \equiv E$ by non-self-fullness. So $R\leq E$. Since $E\leq R\oplus
\Id_m$ for some $m$, and $E\equiv E\oplus \Id_m$ by non-self-fullness, we
have $E\oplus \Id_m \le R\oplus \Id_m$. But by
Lemma~\ref{SuccessorIsInjective}, this implies that $E\leq R$. Therefore
$E\equiv R$ as desired.
\end{proof}

\subsection{Least upper bounds and ceers of the form $R_X$ in $\Ceers_\I$}
We now consider the question of the existence of least upper bounds in the
quotient structure $\Ceers_\I$. Some of the results in Section~\ref{sct:lubs}
about non-existence of least upper bounds in $\Ceers$ follow from
corresponding non-existence results in $\Ceers_{\I}$ by virtue of the
following lemma.

\begin{lem}\label{JoinsPreservedByI}
Let $E,X,Y$ be ceers so that $E$ is a least upper bound of $X$ and $Y$.
Then $E$ is also a least upper bound of $X$ and $Y$ in the
$\equiv_\I$-degrees.
\end{lem}

\begin{proof}
Let $Z$ be so $X,Y\leq_\I Z$, and assume $E$ is a least upper bound of $X$
and $Y$ in $\Ceers$. Then for some $n$, $X,Y\leq Z\oplus \Id_n$. Thus $E\leq
Z\oplus \Id_n$. Thus $E\leq_\I Z$.
\end{proof}

The next theorem is a generalization of Theorem~\ref{NoJoinOf1Degrees} (and in
fact it implies it using Lemma~\ref{JoinsPreservedByI}).

\begin{thm}\label{thm:nojoinI}
Let $X$ and $Y$ be infinite c.e. sets so that $X\not\leq_1 Y\oplus \emptyset$
and $Y\not\leq_1 X\oplus \emptyset$. Then $R_X$ and $R_Y$ do not have any
least upper bound in the $\equiv_\I$-degrees.
\end{thm}

\begin{proof}
Firstly, note that for any c.e. set $X$, $R_X\oplus \Id_1$ is equivalent to
$R_Z$ for $Z=\{x+1\mid x\in X\}$. By iteration, for every $n$, $R_X\oplus
\Id_n$ is equivalent to $R_Z$, for some c.e. set $Z$. Suppose $E$ is a least
upper bound in the $\equiv_\I$-degrees of $R_{X}$ and $R_{Y}$. Note that
$R_{X\oplus Y}$ is an upper bound for both $R_X$ and $R_Y$, so $E \leq
R_{X\oplus Y}\oplus \Id_n$ for some $n$, so $E\leq R_Z$ for some $Z$. Since
the degrees of the ceers of the form $R_Z$ form an interval in $(\Ceers,
\leq)$ (Fact~\ref{fact:fund-ce-sets}), we know that up to equivalence any
least upper bound must be a ceer of this form form as well.  Now, suppose
$R_Z$ were such a least upper bound in the $\equiv_\I$-degrees. Then
$R_X\oplus R_Y$ is also an upper bound. So, $R_Z\leq_\I R_X\oplus R_Y$. So
$R_Z\leq R_X\oplus R_Y \oplus \Id_n$ for some $n$ and let $f$ be a computable
function witnessing the reduction. But then (as $Z$ is not decidable) either
the infinite class in $R_Z$ is sent to the infinite class representing $X$ or
the infinite class representing $Y$. We either get  by
Lemma~\ref{coproduct3}(1) (taking $W=\{3x\mid x \in \omega\}$ and taking $U$
to be the intersection of the complement of $W$ with the image of the
reduction) that $R_Z\leq R_X\oplus \Id$ or $R_Z\leq R_Y\oplus \Id$. In the
first case we get that $R_Y\leq R_Z\oplus \Id_k\leq R_X\oplus \Id$ for some
$k$, and in the second case, we get that $R_X\leq R_Z\oplus \Id_k\leq R_Y
\oplus \Id$ for some $k$. Either way contradicts the assumptions as
$R_X\oplus \Id= R_{X\oplus \emptyset}$ and $R_Y\oplus \Id= R_{Y \oplus
\emptyset}$.
\end{proof}

We know already that if $R_X$ and $R_Y$ are dark and $\leq$-incomparable then
they do not have any least upper bound in $\Ceers$. The next corollary shows
that this is the case even in the quotient structure modulo $\equiv_\I$.
(Notice that Corollaries \ref{cor2bis}, \ref{cor3bis}, \ref{cor4bis} imply
respectively via Lemma~\ref{JoinsPreservedByI} the analogous Corollaries
\ref{darkRNoJoin}, \ref{cor:there-are-pairs}, \ref{JoinIrreducible1Degrees}
proved for $\leq$.)

\begin{cory}\label{cor2bis}
If $R_X$ and $R_Y$ are dark and $\leq_{\I}$-incomparable, then they have no
least upper bound in the $\equiv_\I$-degrees.
\end{cory}

\begin{proof}
As above, if they have an upper bound in the $\equiv_\I$-degrees, it has up
to equivalence the form $R_Z$. Suppose $R_Z$ is such a least upper bound.
Then $R_Z\leq_\I R_X\oplus R_Y$, so $R_Z\leq R_X\oplus R_Y\oplus \Id_n$.
Without loss of generality, we assume the infinite class in $R_Z$ is sent to
the infinite class in $R_X$. Then the image of the reduction in $R_Y$ must be
finite since $R_Y$ is dark. Thus by Lemma~\ref{coproduct3}(1) (taking
$W=\{3x\mid x \in \omega\}$ and $U$ the intersection of the complement of $W$
with the image of the reduction) we have that $R_Z\leq R_X\oplus \Id_m$ for
some $m$, so $R_Z\equiv_\I R_X$. That is, $R_Y\leq_\I R_X$, which contradicts
$\leq_{\I}$-incomparability of $R_X$ and $R_Y$.
\end{proof}

\begin{cory}\label{cor3bis}
There are pairs $X,Y$ so that both are dark, so that both are light, and so
that exactly one is dark, which do not have least upper bounds in the
$\equiv_\I$-degrees.
\end{cory}

\begin{proof}
By Theorem~\ref{thm:nojoinI}, using the argument in the proof of
Corollary~\ref{cor:there-are-pairs}. For the existence of a pair (dark,dark)
the claim follows also from Corollary~\ref{cor2bis}.
\end{proof}

\subsection{Join-irreducibility in $\Ceers_\I$}
In this section we single out some classes of ceers that are join-irreducible
in $\Ceers_{\I}$.

\begin{cory}\label{cor4bis}
Every dark $R_{X}$ is join-irreducible in the $\I$ degrees.
\end{cory}

\begin{proof}
Let $R_X$ be dark. Every ceer $\leq_\I R$ is also dark and up to equivalence
of the form $R_Z$. Since no two incomparable such degrees have a least upper
bound in the $\I$-degrees by Corollary~\ref{cor2bis}, $R$ cannot be a join of
two smaller $\equiv_\I$ degrees.
\end{proof}

We note a second example (in addition to that in Corollaries
\ref{JoinIrreducible1Degrees} and \ref{cor4bis}) of
degrees which are join-irreducible:

\begin{thm}\label{thm:join-irreducible-in-I}
Let $E$ be any ceer. Then $E'$ is light and join-irreducible in both the
ceers and in the $\equiv_\I$ degrees.
\end{thm}

\begin{proof}
By Lemma \ref{JoinsPreservedByI}, we only need to show that $E'$ is
join-irreducible in the $\equiv_\I$-degrees. If $E'$ were the $\equiv_I$-join
of $\leq_{\I}$-incomparable $X$ and $Y$, then $E'\leq X\oplus Y\oplus \Id$.
But then $E'\leq X$, $E'\leq Y$ or $E'\leq \Id$ by the fact that jumps are
uniform join-irreducible (see Fact~\ref{JumpsUniformJoinIrreducible}). But no
jump is $\leq \Id$ because $E'$ has non-computable classes, so again we
contradict $X$ and $Y$ being incomparable.
\end{proof}

\begin{cory}\label{cor:joinR-irred}
There exist light ceers of the form $R_{X}$ which are join-irreducible both in
ceers and in the $\equiv_\I$ degrees.
\end{cory}

\begin{proof}
Consider for instance $\Id=R_{\emptyset}$, or $R_{K}$: this latter example
follows from the fact that $R_{K}=(\Id_{1})'$.
\end{proof}

\subsection{When least upper bounds exist in $\Ceers$ and in $\Ceers_\I$}
Contrary to the structure of ceers where two incomparable dark degrees do not
have a least upper bound (Theorem~\ref{NoJoinOfDark}), in
Corollary~\ref{cor:dark-with-supI} below we give an example that shows that
this is not so when working in the $\I$-degrees. We first prove the following
theorem.

\begin{thm}\label{DarkCeersJoinInI}
There are dark ceers $E_1$, $E_2$ so that $E_{1} \nleq E_{2} \oplus \Id$,
$E_{2} \nleq E_{1} \oplus \Id$, and all classes of each $E_i$ are finite and
so that for any infinite c.e. set $W\subseteq \omega^2$, there are
$(x,y),(x',y')\in W$ so that
\[
x \rel{E_1} x' \cancel{\Leftrightarrow} y \rel{E_2} y'.
\]
\end{thm}

\begin{proof}
We build ceers $E_1, E_2$ satisfying the following requirements, where $i,m,k
\in \omega$ and $j\in \{1,2\}$:

\begin{itemize}
\item[$P_i$:] If $W_i\subseteq \omega^2$ is infinite, then there are
    $(x,y),(x',y')\in W$ so that
    \[x \rel{E_1} x' \cancel{\Leftrightarrow}
    y \rel{E_2} y'.
    \]

\item[$D^j_m$:] If $W_m$ is an infinite c.e. set, then there are $x,y\in
    W_m$ so that $x \rel{E_j} y$.

\item[$I^{1,2}_{e}$:] $\phi_{e}$ is not a reduction from $E_{1}$ to
    $E_{2}\oplus \Id$.

\item[$I^{2,1}_{e}$:] $\phi_{e}$ is not a reduction from $E_{2}$ to
    $E_{1}\oplus \Id$.

\item[$F^j_k$:] The $E_j$-class of $k$ is finite.
\end{itemize}

\medskip
We describe the strategies to meet the requirements, on which we have fixed
some computable priority ordering of order type $\omega$.

\medskip

\emph{Strategy for $P_i$:} With finitely many restrained numbers (giving only
finitely many restrained equivalence classes), $P_i$ waits for $W_i$ to
enumerate two pairs $(x,y),(x',y')$ so that currently $x\cancel{\rel{E_1}}x'$
and the pair $x,x'$ is not $E_{1}$-equivalent to a pair of restrained
numbers. At this point $P_i$ is \emph{ready to act} and then
(\emph{$P_i$-action}), if $y \rel{E_2} y'$, then we restrain $x$ and $x'$;
otherwise we $E_1$-collapse $x$ and $x'$ and $E_2$-restrain $y$ and $y'$.

\medskip

\emph{Strategy for $D^j_m$:} If there is no pair of distinct numbers of
$W_{m}$ which are already $E_{j}$-collapsed, then wait until $W_m$ enumerates
an unrestrained pair of elements $x,y$: when this happens $D^j_m$ is
\emph{ready to act} and (\emph{$D^j_m$-action}) it $E_j$-collapses them.

\medskip

\emph{Strategies for $I^{1,2}_{e}$ and $I^{2,1}_{e}$:} We only look at
$I^{1,2}_{e}$, as the strategy for $I^{2,1}_{e}$ is completely similar. This
is the usual diagonalization strategy for incomparability: $I^{1,2}_{e}$
appoints two new witness $x,y$, and restrains them from being
$E_{1}$-collapsed until both $\phi_{e}(x), \phi_{e}(y)$ converge. When this
happens if already $\phi_{e}(x) \rel{E_{2} \oplus \Id} \phi_{e}(y)$ then we
keep the restraint in $E_{1}$ for $x,y$; otherwise if still $\phi_{e}(x)
\nrel{E_{2} \oplus \Id} \phi_{e}(y)$ then \emph{$I^{1,2}_{e}$ is ready to
act}, and (\emph{$I^{1,2}_{e}$-action}) it $E_{1}$-collapses $x,y$ and we
restrain the pair $\phi_{e}(x), \phi_{e}(y)$ from being $E_{2}  \oplus
\Id$-collapsed.

\medskip

\emph{Strategy for $F^j_k$:} Do not allow any element to collapse with $k$ in
$E_j$. Higher priority requirements will ignore this, but this strategy will
succeed in guaranteeing that the $E_{j}$-class of $k$ is finite.

\medskip

\medskip
\emph{The construction (sketch):} A formal construction follows along the
usual lines of a finite injury argument, as for instance in the proof of
Theorem~\ref{MinimalDark}. In particular when a requirements acts, if allowed
by restraints imposed by higher priority strategies, it \emph{initializes}
all lower-priority $P$- and $I$-requirements (the only requirements that may
be injured by the actions of higher priority requirements). At any given
stage we take the least requirement $R$ that \emph{requires attention} (which
is defined in the usual way, i.e. the requirement is initialized, or it is
ready to take action): if $R$ is an $I$-requirement and is initialized then
it chooses new parameters $x^R, y^R$ and ceases to be initialized; if $R$ is
a $P$-requirement and it is initialized, then it simply ceases to be
initialized; otherwise we let $R$ act. After this, we go on to next stage. At
each stage $s$ we define in the usual way the restraints sets $\rho^{R}_{s}$,
and $E_{l,s}$ starting with $E_{l,0}=\Id$, and stipulating at stage $0$ the
once and for all restraint $(k,j)^F$ so that always $(k,j)^F \in \rho^{R'}$
for all $R'$ having lower priority than $F^{j}_{k}$, and thus $R'$ cannot
$E_j$-collapse any number to $k$: since $F$-requirements are never
re-initialized, their restraints will never change.

\medskip

\emph{The verification:} An easy inductive argument on the priority ordering
of requirements shows that each requirement is initialized at most finitely
often. After any initialization, each $I$-requirement after choosing its
parameters acts at most once, each $P$-requirement acts at most once, and in
each case they set only a finite restraint. The other requirements act at
most once; the $D$-requirements set no restraint; the $F$-requirements set a
once and for all finite restraint. Satisfaction of the $F$-requirements gives
that each $E_j$ has finite classes, and thus infinitely many equivalence
classes (this is also ensured by satisfaction of all incomparability
requirements). Together with their having infinitely many classes,
satisfaction of all $D^j_m$ guarantees that the ceers are dark: if $W_m$ is
infinite then eventually a pair of unrestrained $x,y$ will appear as $D^j_m$
has to cope with only a finite set of equivalence classes that are restrained
by higher priority requirements. A similar observation applies to the
$P$-requirements, although we need here a different argument to show that the
requirement eventually stops waiting if $W_i$ is infinite: indeed, we use
here that the equivalence classes are finite, and thus eventually a pair
$(x,y), (x',y')$ with still $x \nrel{E_1} x'$ will appear. The $P$- and the
$D$-requirements may be met also by the waiting outcomes of their strategies
(no suitable numbers in the corresponding c.e. sets ever appear, so these
sets are finite), in which case they may never act but are satisfied. It is
clear that after their last initialization the $I$-requirements are
eventually met: for instance $I^{1,2}_{e}$ is met either by waiting forever
for both computations $\phi_{e}(x), \phi_{e}(y)$ to converge on the final
values $x,y$ of the witnesses appointed after last initialization, or
otherwise by keeping the restraint $x \nrel{E}_{1} y$ when both computations
show up and already $\phi_{e}(x) \rel{E_{2}\oplus \Id} \phi_{e}(y)$, or by
$E_1$-collapsing $x,y$ and keeping the restraint $\phi_{e}(x)
\nrel{E_{2}\oplus \Id} \phi_{e}(y)$. As to the finiteness requirement $F^j_k$
we further observe that higher priority requirements will ignore $F^j_k$'s
restraint, but since they are finitary it follows that $[k]_{E_j}$ is finite.
\end{proof}

\begin{cory}\label{cor:dark-with-supI}
There are two
$\I$-incomparable  dark ceers $E_{1}, E_{2}$ so that $E_1\oplus E_2$ is a
least upper bound of $E_1$ and $E_2$ in the $\equiv_\I$-degrees.
\end{cory}

\begin{proof}
Let $E_1$ and $E_2$ be as constructed in Theorem \ref{DarkCeersJoinInI}. Let
$X$ be any ceer so $E_1\leq_\I X$ and $E_2\leq_\I X$. Then for some $n$,
$E_1,E_2\leq X\oplus \Id_n$. Let $Y=X\oplus \Id_n$, so $Y\equiv_\I X$. It now
suffices to show that $E_1\oplus E_2\leq_\I Y$.

We fix $f_1,f_2$ witnessing that $E_1,E_2\leq Y$. Let $W$ be the c.e. set
\[
W=\{(x,y)\mid f_1(x) \rel{Y} f_2(y)\}.
\]
By Lemma~\ref{obs:coproduct} we have that $(E_1\oplus {E_2})_{/W} \leq Y$. By
the property of $E_1$ and $E_2$ from Theorem \ref{DarkCeersJoinInI} (namely
by the fact that $E_1$ and $E_2$ satisfy the requirement $P_i$ with $W_i=W$),
we see that $W$ is finite. Further, using that the equivalence classes of
$E_{1}\oplus E_{2}$ are finite and thus computable, we have that $E_1\oplus
{E_2}_{/W}\equiv_I E_1\oplus E_2$ by Lemma \ref{ComputableCollapses}. Finally
we get $E_1\oplus E_2 \leq_I Y$.

$\I$-incomparability of $E_{1}, E_{2}$ is guaranteed by satisfaction of the
incomparability $I$-requirements in the proof of Theorem
\ref{DarkCeersJoinInI}.
\end{proof}

\begin{cory}\label{DarkLightJoin}\label{LightDarkJoin}
There are pairs of incomparable ceers $E,R$ so that $E$ is dark, $R$ is light
and not $\equiv \Id$, and $E\oplus R$ is a least upper bound of $E$ and $R$.
\end{cory}

\begin{proof}
Let $E_1$ and $E_2$ be as constructed in Theorem \ref{DarkCeersJoinInI}. Let
$E=E_1$ and $R=E_2\oplus \Id$. Clearly $R \nequiv \Id$ as otherwise it would
be $E_2 \leq \Id$ contradicting that $E_2$ is dark. To show that $E$ and $R$
have the desired properties, we first show the following sublemma.

\medskip

\noindent \textbf{Sublemma} \emph{If $W \subseteq \omega^{2}$ is an infinite
c.e.\ set, then there are $(x,y),(x',y')\in W$ so that $x \rel{E} x'
\cancel{\Leftrightarrow} y \rel{R} y'$.}

\begin{proof}[Proof of Sublemma]
If $W$ has infinitely many pairs $(x,y)$ where $y$ is odd, then there are
$(x,y),(x',y')\in W$ so that $x \rel{E_1} x' \cancel{\Leftrightarrow} y
\rel{R} y'$ by the darkness of $E_1$, as otherwise the c.e. set $\{x\mid
(\exists y)[(x,y) \in W \,\&\, \textrm{ $y$ odd}]\}$ would be an infinite
c.e. set consisting of non-$E_1$-equivalent numbers. Now, suppose $W$ has
infinitely many pairs $(x,y)$ where $y$ is even. Then let $W'=\{(x,y)\mid
(x,2y)\in W\}$. Then since $W'$ is infinite, there is $(x,y),(x',y')\in W'$
so that $x \rel{E_1} x' \cancel{\Leftrightarrow} y \rel {E_2} y'$. But then
$(x,2y),(x',2y')$ witness the result for $E,R$.
\end{proof}

Now, suppose $E,R\leq X$. We want to show that $E\oplus R\leq X$. Again, we
fix two reductions $f_E$ and $f_R$ of $E,R$ to $X$. Let $W=\{(x,y)\mid f_E(x)
\rel{X} f_R(y)\}$.  By Lemma~\ref{obs:coproduct} we can reduce $(E\oplus
R)_{/W}$ to $X$. By the above Sublemma, $W$ is necessarily finite, and thus
as the equivalence classes of $E \oplus R$ are computable, by Lemma
\ref{ComputableCollapses} we have that $E\oplus R \equiv (E\oplus R)_{/W}
\oplus \Id_n$ for some $n$. We claim that $(E \oplus R)_{/W} \oplus \Id \le
(E \oplus R)_{/W}$: if we show this then we are done. But $R=E_2\oplus \Id$,
so by separating out the set of odd numbers which are not in the second
coordinate of an element of $W$, we can find an infinite strong effective
transversal $U$ of $(E \oplus R)_{/W}$, and thus Lemma~\ref{obs:coproduct2}
can be applied, giving $(E \oplus R)_{/W} \oplus \Id \le (E \oplus R)_{/W}$
as desired.

Incomparability (in fact $\I$-incomparability) of $E, R$ is guaranteed by
satisfaction of the incomparability $I$-requirements for $E_1$ and $E_2$ in
the proof of Theorem \ref{DarkCeersJoinInI}.
\end{proof}

We note that the above construction yields a more flexible method, producing
more examples, though the result of the existence of pairs consisting of one
dark and one light ceer having a join could be taken from Observation
\ref{joinIdDark}.

\begin{cory}\label{SometimesJoinOfLight}
There are two
incomparable light ceers $E,R$ for which $E\oplus R$ is a least upper bound
for $E,R$.
\end{cory}

\begin{proof}
Let $E_1$ and $E_2$ be as constructed in Theorem \ref{DarkCeersJoinInI}. Let
$E=E_1\oplus \Id$ and $R=E_2\oplus \Id$. Note that $E\oplus R\equiv E_1\oplus
E_2\oplus \Id$. Since any upper bound $X$ of $E,R$ is also an upper bound of
$E_1, R$, by the same argument as in the proof of
Corollary~\ref{LightDarkJoin} we have that $(E_1 \oplus R)_{/W} \leq X$. (We
view here $E_1\oplus E_2\oplus \Id$ as $E_1\oplus (E_2\oplus \Id)$ and $W$ is
as in the proof of the previous corollary.) Since $E_1$ is dark, there can
only be finitely many odd $z$ so that $f_{E_1}(x)=f_{E_2 \oplus \Id}(z)$. By
discarding these finitely many elements, we see that there is an infinite
strong effective transversal $U$ of $(E_1\oplus E_2 \oplus \Id)_{/W}$ coming
from the non-collapsed elements in the $\Id$-part of the ceer, so that
Lemma~\ref{obs:coproduct2} can be applied, giving $(E_1\oplus E_2 \oplus
\Id)_{/W} \oplus \Id\leq (E_1\oplus E_2 \oplus \Id)_{/W}$. On the other hand,
by Lemma~\ref{ComputableCollapses}, $E_1 \oplus E_2 \oplus \Id \equiv
(E_1\oplus E_2 \oplus \Id)_{/W} \oplus \Id_n$ for some $n$, thus
\[
E_1 \oplus E_2 \oplus \Id \equiv (E_1\oplus E_2 \oplus \Id)_{/W}
       \oplus \Id_n \leq
(E_1\oplus E_2 \oplus \Id)_{/W} \oplus \Id \leq (E_1\oplus
E_2 \oplus \Id)_{/W} \leq X,
\]
as desired.

Once again, incomparability (in fact, $\I$-incomparability) of $E, R$ is
guaranteed by satisfaction of the incomparability $I$-requirements for $E_1$
and $E_2$ in the proof of Theorem \ref{DarkCeersJoinInI}.
\end{proof}

\begin{cory}\label{cor:have-join-I-oplus}
There are pairs of incomparable $\I$-degrees, both dark, exactly one light
and one dark, and both light which have a join given by $\oplus$ in the
$\equiv_\I$-degrees.
\end{cory}

\begin{proof}
The claims follow from Corollary~\ref{cor:dark-with-supI},
Corollary~\ref{LightDarkJoin} and Corollary~\ref{SometimesJoinOfLight},
respectively, together with Lemma~\ref{JoinsPreservedByI}.
$\I$-incomparability of the pair is guaranteed by satisfaction of the
incomparability $I$-requirements for $E_1$ and $E_2$ in the proof of Theorem
\ref{DarkCeersJoinInI}.
\end{proof}

Lastly, we give examples where $E_1$ and $E_2$ have a join, but not
equivalent to $E_1\oplus E_2$.

\begin{thm}\label{thm:sup-no-oplus}
There are two incomparable ceers $E_1, E_2$ (so that $E_1$ is dark and $E_2$
is light) with a least upper bound $R$ so that $R<E_1\oplus E_2$.
\end{thm}

\begin{proof}
We build ceers $X,Y,Z$ and will set $E_1=X\oplus Z$ and $E_2= Y\oplus Z$. Our
$R$ will be $X\oplus Y\oplus Z$. We will ensure that $R<X\oplus Z\oplus
Y\oplus Z$. Further, we construct $Y$ as $Y=Y_0\oplus \Id$, which gives $E_2$
light; to make $E_1$ dark, by Observation~\ref{obs:dark-closure} it is enough
to ensure that $X$ and $Z$ are both dark. Once again we build the two ceers
so that $E_1 \nleq E_2 \oplus \Id$ and $E_2 \nleq E_1 \oplus \Id$: this
stronger incomparability property will be used to have incomparability in
Corollaries~\ref{cor:light-sup-less-than-o1}~and~\ref{cor:light-sup-less-than-o2}
below.

We construct the ceers $X,Y,Z$ with the following requirements:
\begin{itemize}
\item[$F^X_k$:] The $X$-class of $k$ is finite.
\item[$F^Z_k$:] The $Z$-class of $k$ is finite.
\item[$P_i$:] If $W_i\subseteq\omega^2$ is a c.e. set so that
    $W_i\smallsetminus \{(x,y)\mid x,y\text{ odd}\}$ is infinite, then
    there are $(x,y),(x',y')\in W_i$ so that
    \[
    x \rel{X\oplus Z} x'\cancel{\Leftrightarrow} y \rel{Y\oplus Z} y'.
    \]

\item[$Q_j$:] $\phi_j$ is not a reduction of $X\oplus Y\oplus Z\oplus Z$ to
    $X\oplus Y\oplus Z$.
\item[$DX^j_m$:] If $W_m$ is an infinite c.e. set, then there are $x,y\in
    W_m$ so that $x \rel{X} y$.
 \item[$DZ^j_m$:] If $W_m$ is an infinite c.e. set, then there are $x,y\in
     W_m$ so that $x \rel{Z} y$.
 \item[$I^{1,2}_{e}$:] $\phi_{e}$ is not a reduction from $X\oplus Y$ to $X
     \oplus Z \oplus \Id$.
 \item[$I^{2,1}_{e}$:] $\phi_{e}$ is not a reduction from $Y\oplus Z$ to $X
     \oplus Z \oplus \Id$.
\end{itemize}

The stronger form taken by the incomparability $I$-requirements (yielding
that $X\oplus Y \nleq X \oplus Z \oplus \Id$ and $Y\oplus Z \nleq X \oplus Z
\oplus \Id$, and not just $X\oplus Y \nleq X \oplus Z$ and $Y\oplus Z \nleq X
\oplus Z$) is in view of later applications in
Corollaries~\ref{cor:light-sup-less-than-o0},
\ref{cor:light-sup-less-than-o1}, \ref{cor:light-sup-less-than-o2}.

\medskip We describe the strategies:

\medskip

\emph{Strategy for $F^X_k$ or $F^Z_k$:} The strategy for $F^X_k$ simply
restrains against any $X$-collapse to $k$. Higher priority requirements will
ignore this, but this strategy will succeed in guaranteeing that the
equivalence class of $k$ is finite. The strategy for $F^Z_k$ is similar.

\medskip

\emph{Strategy for $P_i$:} Given finite restraint (and $Y$'s restraint on the
odd numbers: no strategy can $Y$-collapse odd numbers, as we want $Y$ of the
form $Y=Y_0 \oplus \Id$), we wait until we see $(x,y),(x',y')$ enter $W_i$
with $x=x' \;\textrm{mod}_2$ (i.e. of the same parity), and $y= y'
\;\textrm{mod}_2$ and at least one pair among $(x,x'), (y,y')$ consists of
even numbers, and if $(x,x')$ consists of even numbers then still
$\frac{x}{2} \cancel{X} \frac{x'}{2}$, whereas if $y,y'$ both even then still
$\frac{y}{2} \cancel{X} \frac{y'}{2}$. When this happens, $P_i$ is
\emph{ready to act}. Without loss of generality, suppose $x,x'$ are even.
Then (\emph{$P_i$ action}) we $X$-collapse $\frac{x}{2}$ and $\frac{x'}{2}$
and $Y$-restrain $\frac{y}{2}, \frac{y'}{2}$ if $y,y'$ are even and
$Z$-restrain $\frac{y-1}{2}, \frac{y'-1}{2}$ if $y,y'$ are odd: all this
unless we have already collapsed on the other side, in which case we just
restrain in $X$. The case when $y,y'$ are both even is similar. Notice that
if $W_i\smallsetminus \{(x,y)\mid x,y\text{ odd}\}$ is infinite then the wait
is eventually over, since either $W_i$ contains infinitely many pairs with
first coordinate an even number, but then eventually our wait hits a pair as
desired with still $\frac{x}{2} \cancel{X} \frac{x'}{2}$ as $X$ has no
infinite equivalence class; or $W_i$ contains infinitely many pairs with
second coordinate an even number, but then eventually our wait hits a pair as
desired with still $\frac{y}{2} \cancel{X} \frac{y'}{2}$ as $X$ has no
infinite equivalence class.

\medskip

\emph{Strategy for $Q_j$:} Wait for $a,b$ from the different copies of $Z$ in
$Z\oplus Z$ (from $X\oplus Y \oplus Z \oplus Z$) so that:
\begin{enumerate}
  \item $\phi_j(a),\phi_j(b)$ lie both either in the $X$-part, or in the
      $Y$-part or in the $Z$-part.
  \item $\phi_j(a),\phi_j(b)$ are unrestrained in the $X$-part or in the
      $Z$-part, or in the $Y_0$-part of the $Y$-part.
\end{enumerate}
When this happens, $Q_j$ is \emph{ready to act}. Then (\emph{$Q_j$-action})
collapse $\phi_j(a),\phi_j(b)$ in the $X$- or $Y$- or $Z$-part accordingly,
making $\phi_j(a) \rel{X\oplus Y \oplus Z} \phi_j(b)$: this action satisfies
the requirement since $a, b$ remain $X\oplus Y \oplus Z \oplus
Z$-inequivalent.

\medskip

\emph{Strategy for $DX^j_m$ and $DZ^j_m$:} The strategies to meet these
requirements are exactly as the strategies for the darkness requirements in
the proof of Theorem~\ref{DarkCeersJoinInI}. For $DX^j_m$: If there is no
pair of distinct numbers of $W_{m}$ which are already $X$-collapsed, then
wait until $W_m$ enumerates an unrestrained pair of elements $x,y$. When this
happens, the requirement is \emph{ready to act}, and (\emph{$D^j_m$-action})
we $E_j$-collapse them. Same for $DZ^j_m$. If all these requirements are met,
then clearly $X, Z$ are dark as by the $F$-requirements they have infinitely
many equivalence classes.

\medskip

\emph{Strategies for $I^{1,2}_{e}$ and $I^{2,1}_{e}$:}  Consider first
$I^{1,2}_{e}$. We want to make sure in this case that $\phi_{e}$ is not a
reduction from $X\oplus Z$ to $Y \oplus Z \oplus \Id$. We appoint two new
witness $2x,2y$ and restrain the pair $x,y$ from being $X$-collapsed until
both $\phi_{e}(2x), \phi_{e}(2y)$ converge. When this happens if already
$\phi_{e}(2x) \rel{Y \oplus Z \oplus \Id} \phi_{e}(2y)$ then we keep the
restraint in $X$ for $x,y$; otherwise, if still $\phi_{e}(2x) \nrel{Y \oplus
Z \oplus \Id} \phi_{e}(2y)$ then $I^{1,2}_{e}$ is \emph{ready to act} and
(\emph{$I^{1,2}_{e}$-action}) it $X$-collapses $x,y$ and restrains in $Y
\oplus Z \oplus \Id$ the pair $\phi_{e}(2x), \phi_{e}(2y)$. The case of
$I^{2,1}_{e}$ is similar with $Y$ in place of $X$ except that we must now
choose the witnesses $2x,2y$ so that $x,y$ are in the $Y_{0}$-part of $Y$ and
thus we can $Y$-collapse them if needed.

\medskip

We now argue that if $\phi_j$ is a reduction of $X\oplus Y\oplus Z\oplus Z$
to $X\oplus Y\oplus Z$ and $Z \nleq \Id$ then the $Q_j$-strategy eventually
stops waiting and $Q_j$ acts. Notice that in this case since the image of
$\phi_{j}$ has to avoid only a finite number of equivalence classes
restrained by higher-priority requirements, and since $Z$ has finite
equivalence classes (and thus $Z \oplus Z$ has infinitely many distinct
equivalence classes), then for cofinitely many $a$ in the first $Z$-part of
the $Z\oplus Z$-part of $X\oplus Y \oplus Z \oplus Z$, and cofinitely many
$b$ in the second $Z$-part of the $Z\oplus Z$-part of $X\oplus Y \oplus Z
\oplus Z$, we have that $\phi_j(a),\phi_j(b)$ are unrestrained as in (2). So
by the pigeon-hole principle there are pairs $a,b$, where
$\phi_j(a),\phi_j(b)$ either both lie in the $X$-part, or in the $Y$-part or
in the $Z$-part. Thus $Q_j$-action is prevented only if for cofinitely many
such pairs we have that both $\phi_j(a),\phi_j(b)$ land in the $\Id$-part of
the $Y$-part of $X\oplus Y\oplus Z$. Using this we can thus easily see that
there is a reduction from $Z\oplus Z$ to $X\oplus Y\oplus Z$ which all lands
into the $\Id$-part of $Y$ except for finitely many classes, so by
Lemma~\ref{coproduct3}(1) we would have $Z \oplus Z \leq \Id\oplus \Id_n$ for
some $n$, hence  $Z\oplus Z \leq \Id$, contrary to the fact that $Z \nleq
\Id$.

\medskip

Once we satisfy these requirements, we claim that $R$ is a least upper
bound of $E_1$ and $E_2$. Certainly $X\oplus Z,Y\oplus Z\leq R$. Suppose
$X\oplus Z,Y\oplus Z\leq A$ via reductions $f,g$. Then let $C=\{(x,y)\mid
f(x) \rel{A} g(y)\}$. It follows from the $P$-requirements that
$E=C\smallsetminus \{(x,y)\mid x,y\text{ odd}\}$ is finite. The following is
very similar to what is done in Lemma~\ref{obs:coproduct}:
define a computable function $h$ by
\[
h(b)=
\left\{
  \begin{array}{ll}
    f(2a),  & \hbox{if $b=3a$;} \\
    g(2a), & \hbox{if $b=3a+1$;} \\
    g(2a+1),& \hbox{if $b=3a+2$.}
  \end{array}
\right.
\]
Then $h$ is a reduction from $(X\oplus Y\oplus Z)_{/D}$ to $A$, where $D$ is
a finite set of pairs which identifies the $X\oplus Y\oplus Z$-collapses
corresponding to the collapses between $X\oplus Z$- and $Y\oplus Z$-classes
described by $E$. Since the equivalence classes of $X$ and $Z$ are finite,
hence computable, by Lemma~\ref{ComputableCollapses} $(X\oplus Y\oplus
Z)_{/D} \oplus \Id_n \equiv X\oplus Y\oplus Z$ for some $n$. But since
$Y=Y_0\oplus \Id$ and only finitely many numbers in the $\Id$-part of $Y$ are
coordinates in $D$, we can find an infinite strong effective transversal $U$
for $(X\oplus Y\oplus Z)_{/D}$ coming from the still non-collapsed $\Id$-part
of the $Y$-part in the relation, so that Lemma~\ref{obs:coproduct2} can be
applied obtaining $(X\oplus Y\oplus Z)_{/D}\oplus \Id \leq (X\oplus Y\oplus
Z)_{/D}$ and thus
\[
X\oplus Y\oplus Z \equiv (X\oplus Y \oplus Z)_{/D} \oplus \Id_n  \leq
(X\oplus Y \oplus Z)_{/D} \oplus \Id \leq (X\oplus Y\oplus Z)_{/D} \leq A.
\]

\medskip
\emph{Outcomes of the strategies:} The $F$-strategies clearly produce winning
outcomes. $P$-strategies may have a waiting outcome if the corresponding c.e.
set $W$ is such that $W \smallsetminus \{(x,y)\mid x,y \textrm{ odd}\}$ is
finite, in which case the requirement is met without acting; or $W
\smallsetminus \{(x,y)\mid x,y \textrm{ odd}\}$ is infinite: then as
explained in the description of the strategy, the fact that the equivalence
classes of $X$ are finite implies that the wait eventually ends, and the
action that takes place gives a winning outcome. A $Q$-strategy has a waiting
outcome, and an outcome coming from acting: both outcomes fulfill the
requirement. The darkness $D$-requirements have a waiting outcome, and an
outcome coming from acting: both outcomes fulfill each requirement. The
$D$-requirements guarantee that $X$ and $Z$ are dark, as they both have
infinitely many equivalence classes, since each class is finite. The outcomes
of the incomparability $I$-requirements are as in the proof of
Theorem~\ref{DarkCeersJoinInI}, but having now $E_{1}=X\oplus Z$ and $E_{2}=Y
\oplus Z$.

\medskip

\emph{The construction:} We fix a computable priority ordering on the
requirements having order type $\omega$. The construction follows the usual
pattern of a finite injury argument. When a requirement acts, it initializes
all lower-priority  $P$- and $I$-requirements, which are the only
requirements whose action may be injured by higher priority requirements. A
requirement $R$ \emph{requires attention} at stage $s$ if $R$ is either
initialized, or it is now ready to act; or $R$ is a $Q$- or $D$-requirement
and it has never acted but it is now ready to act as described in the
corresponding strategy.

\medskip
At stage $0$, all $P$-
and $I$-requirements are initialized, and for every $F_k^X$, and $F_k^Z$ we
establish the once and for all restraint $(k,X)^F$, or $(k,Z)^F$ so that
lower priority requirements are not allowed to $X$-collapse or to
$Z$-collapse to $k$ respectively. We also define $E_{0}=\Id$ for $E\in
\{X,Y,Z\}$

\medskip

At stage $s+1$ we consider the least requirement $R$ which requires
attention, and we act correspondingly. If $E\in \{X,Y,Z\}$ then $E_{s+1}$ is
defined as the ceer generated by $E_s$ plus the pairs which have been
$E$-collapsed at $s+1$.

\medskip

\emph{The verification:} By induction on the priority ordering of the
requirements, one sees that every requirement $R$ is re-initialized only
finitely many times and after its last re-initialization acts at most once,
sets a finite restraint, and is eventually satisfied. Indeed suppose that
this is true of every $R'$ having higher priority than $R$. So starting from
some stage $t_0$, no higher priority $R'$ acts any more so if $R$ is a $P$-
or an $I$-requirement then $R$ after $t_0$ is not re-initialized any more,
and after its last re-initialization acts at most once, sets a finite
restraint, and is satisfied, as is clear by the above comments on the
outcomes of its strategy. On the other hand the $Q$- and $D$-requirements act
at most once, the $F$-requirements never act; as to satisfaction of these
requirements, it is clear by the above comments on the outcomes of their
strategies that the $F$-requirements are satisfied, and the $D$-requirements
are satisfied; but then the argument attached to the description of the
strategies for $Q$-requirements shows that each $Q$-requirement $Q_j$ is
satisfied as well, in that either $\phi_j$ is not a reduction, or our wait
for $a,b$ eventually ends as otherwise there would be a reduction from $Z$ to
$\Id$ which is excluded by the fact that $Z$ is dark.
\end{proof}

\begin{cory}\label{cory}\label{cor:light-sup-less-than-o0}
If $E_1, E_2$ are the ceers built in Theorem~\ref{thm:sup-no-oplus} then
they are $\I$-incomparable and any least upper bound (which by
Lemma~\ref{JoinsPreservedByI} is also a least upper bound in the $\I$-degrees)
is $<_\I E_1 \oplus E_2$.
\end{cory}

\begin{proof}
$\I$-incomparability follows from satisfaction of the incomparability
requirements in the proof of Theorem~\ref{thm:sup-no-oplus}. Suppose that $X
\oplus Y \oplus Z \oplus Z\le X \oplus Y \oplus Z \oplus \Id_n$ for some $n$.
Since there is an infinite strong effective transversal
$U$ for $X \oplus Y \oplus Z$ (given
by the $\Id$-part of $Y$), by
Lemma~\ref{obs:coproduct2} we have that $X \oplus Y \oplus Z \oplus \Id \le X
\oplus Y \oplus Z$ giving that $X \oplus Y \oplus Z \oplus Z\le X \oplus Y
\oplus Z$, contradicting the conclusion of Theorem~\ref{thm:sup-no-oplus}.
\end{proof}

Of the two ceers built in the proof of Theorem~\ref{thm:sup-no-oplus} one of
them is light as $E_2=Y \oplus Z$, and $Y=Y_0 \oplus \Id$, and it must be so
by Theorem~\ref{NoJoinOfDark}. An interesting question is whether we can make
both ceers light. The following observation shows that this is so, and
completes the picture of the possible cases when a join can be smaller than
the uniform join.

\begin{cory}\label{cor:light-sup-less-than-o1}
There are two incomparable light ceers $R_1, R_2$ with a least upper bound
$R$ so that $R<R_1\oplus R_2$.
\end{cory}

\begin{proof}
Suppose we have $E_1=X\oplus Z$ and $E_2=Y_0\oplus \Id\oplus Z$ as in
Theorem~\ref{thm:sup-no-oplus} where $E_1$ is dark. Consider $R_1= E_1\oplus
\Id$ and $R_2=E_2$. It is easy to see that $E_1\oplus \Id \leq X\oplus
Y\oplus Z$ as $Y$ is of the form $Y=Y_0\oplus \Id$, thus there is an infinite
strong effective transversal for the reduction $E_1 \leq X\oplus Y\oplus Z$
(coming from the $\Id$-part), and thus Lemma~\ref{obs:coproduct2} can be
applied (or, alternatively, $E_1\oplus \Id$ is a least upper bound of $E_1$
and $\Id$ by Observation~\ref{joinIdDark}, so since both reduce to $X\oplus Y
\oplus Z$, we have that $X\oplus Y \oplus Z$ must bound any least upper
bound). Similarly, $E_2 \leq X\oplus Y \oplus Z$. On the other hand, $X\oplus
Y \oplus Z$ must also be a least upper bound because for any ceer $R$, if $R$
is above $E_1\oplus \Id$ and above $E_2$, then $R$ is above $E_1$ and $E_2$,
thus above any least upper bound of those, hence above $X\oplus Y \oplus Z$
which is such a least upper bound by the proof of
Theorem~\ref{thm:sup-no-oplus}. Thus we have two ceers, $E_1\oplus \Id$ and
$E_2$, with a least upper bound which is strictly below $E_1\oplus \Id \oplus
E_2$, because it is strictly below $E_1\oplus E_2$, which is itself $\leq
E_1\oplus \Id \oplus E_2$.

Incomparability of $R_{1}$ and $R_{2}$ follows by the fact that $E_{1},
E_{2}$ satisfy the incomparability requirements in the proof of
Theorem~\ref{thm:sup-no-oplus}.
\end{proof}

\begin{cory}\label{cor:light-sup-less-than-o2}
There are two $\I$-incomparable light ceers $R_1, R_2$ with a least upper
bound $R$ so that $R$ is strictly less than $R_1\oplus R_2$ in the
$\I$-degrees.
\end{cory}

\begin{proof}
Consider the ceers $R_1$ and $R_2$ given by the previous corollary. It
suffices to show that there is no $n \in \omega$ such that $X \oplus Z \oplus
Y_0 \oplus \Id \oplus Z \leq X \oplus Y \oplus Z \oplus \Id_n$. Since there
is an infinite strong effective transversal of $X \oplus Y \oplus Z$ (coming
from the $\Id$-part of the $Y$-part in the equivalence relation), by
Lemma~\ref{obs:coproduct2}  we have that $X \oplus Y \oplus Z \oplus \Id \leq
X \oplus Y \oplus Z$. But then, assuming such an $n$ as before, this implies
$X \oplus Z \oplus Y_0 \oplus \Id \oplus Z \leq X \oplus Y \oplus Z$,
contradicting the conclusion of the previous corollary.

Once more incomparability of $R_{1}$ and $R_{2}$ follows by the fact that
$E_{1}, E_{2}$ satisfy the incomparability requirements in the proof of
Theorem~\ref{thm:sup-no-oplus}.
\end{proof}

\begin{thm}\label{thm:dark-I-less-than-oplus}
There are two $\I$-incomparable dark ceers $E_{1},
E_{2}$ which have a  least upper bound $R$ in the $\I$-degrees so that
$R<_{\I} E_{1} \oplus E_{2}$.
\end{thm}

\begin{proof}
(Sketch.)  We sketch the proof which is a combination of ideas in the proofs
of Theorems~\ref{thm:sup-no-oplus}~and~\ref{DarkCeersJoinInI}. We build
$X,Y,Z$ to be dark ceers so that $E_{1}=X\oplus Z$ and $E_{2}=Y \oplus Z$
have the desired properties. Notice that darkness of $E_{1}, E_{2}$ follows
from darkness of $X,Y,Z$. Suitable requirements are:

\begin{itemize}
\item[$F^X_k$:] The $X$-class of $k$ is finite.
\item[$F^Y_k$:] The $Y$-class of $k$ is finite.
\item[$F^Z_k$:] The $Z$-class of $k$ is finite.
\item[$P_i$:] If $W_i\subseteq\omega^2$ is a c.e. set so that
    $W_i\smallsetminus \{(x,y)\mid x,y\text{ odd}\}$ is infinite, then
    there are $(x,y),(x',y')\in W_i$ so that
    \[
    x \rel{X\oplus Z} x'\cancel{\Leftrightarrow} y \rel{Y\oplus Z} y'.
    \]
\item[$Q_j$:] $\phi_j$ is not a reduction of $X\oplus Y\oplus Z\oplus Z$ to
    $X\oplus Y\oplus Z \oplus \Id$.
\item[$DX^j_m$:] If $W_m$ is an infinite c.e. set, then there are $x,y\in
    W_m$ so that $x \rel{X} y$.
 \item[$DY^j_m$:] If $W_m$ is an infinite c.e. set, then there are $x,y\in
    W_m$ so that $x \rel{Y} y$.
    \item[$DZ^j_m$:] If $W_m$ is an infinite c.e. set, then there are
        $x,y\in W_m$ so that $x \rel{Z} y$.
 \item[$I^{1,2}_{e}$:] $\phi_{e}$ is not a reduction from $X\oplus Y$ to $X
     \oplus Z \oplus \Id$.
 \item[$I^{2,1}_{e}$:] $\phi_{e}$ is not a reduction from $Y\oplus Z$ to $X
     \oplus Z \oplus \Id$.
\end{itemize}

\smallskip

As will be clear from the proof there would be in fact no need to make $Z$
with finite equivalence classes: what we need for darkness in addition to
satisfaction of the $D$-requirements, is that $Z$ has infinitely many
equivalence classes, and this would be  achieved  by ad-hoc requirements
guaranteeing that there are infinitely many equivalence classes, or by
requirements guaranteeing that there is no reduction of $Z$ to $\Id$. We have
chosen the $F^{Z}$-requirements just to follow an already known and familiar
path.

\smallskip

One of the major differences with Theorem~\ref{thm:sup-no-oplus} is that we
are requesting $Y$ to be dark, and this has the advantage that there will be
no $\Id$-part of $Y$ (and thus no $Y_{0}$-part either) to worry about, whose
only purpose was to make $Y$ light. This simplifies the \emph{strategy for
$I^{2,1}$} as now we do not have the restriction of having to choose new
witness in the $Y_{0}$-part of $Y$.

\medskip

Moreover a $Q_j$ requirement is also now slightly different, as we want that
$\phi_{j}$ not to be a reduction of $X\oplus Y \oplus Z \oplus Z$ to $X\oplus
Y \oplus Z \oplus \Id$, instead of just to  $X\oplus Y \oplus Z$. The new
corresponding strategy is now:

\emph{Strategy for $Q_j$:} Wait for $a,b$ from the different copies of $Z$ in
$Z\oplus Z$ (from $X\oplus Y \oplus Z \oplus Z$) so that:
\begin{enumerate}
  \item $\phi_j(a),\phi_j(b)$ lie both either in the $X$-part, or in the
      $Y$-part or in the $Z$-part.
  \item $\phi_j(a),\phi_j(b)$ are unrestrained in the $X$-part or in the
      $Y$-part, or in the $Z$-part.
\end{enumerate}
When this happens, $Q_j$ is \emph{ready to act}. Then (\emph{$Q_j$-action})
it collapses $\phi_j(a)$, $\phi_j(b)$ in the $X$- or $Y$- or $Z$-part
accordingly, making $\phi_j(a) \rel{X\oplus Y \oplus Z\oplus \Id} \phi_j(b)$.
As in Theorem~\ref{thm:sup-no-oplus}  this action satisfies the requirement
since $a, b$ remain $X\oplus Y \oplus Z \oplus Z$-inequivalent.

\medskip

\emph{The construction:} The construction is as in the proof of
Theorem~\ref{thm:sup-no-oplus}, modulo the obvious modifications deriving
from the above described changes in the requirements and in the strategies.

\medskip

\emph{The verification:} The verification is also as in the proof of
Theorem~\ref{thm:sup-no-oplus}, modulo the obvious modifications, and using
the fact that all strategies are finitary and all together may be combined as
a straightforward finite injury construction. From this, we can see that all
requirements which are not $Q_j$-requirements are satisfied. We now argue
that each $Q_j$-requirement is satisfied as well. If $\phi_j$ is a reduction
of $X\oplus Y\oplus Z\oplus Z$ to $X\oplus Y\oplus Z \oplus \Id$ and $Z \nleq
\Id$ then the $Q_j$-strategy eventually stops waiting and $Q_j$ acts. Indeed,
arguing as in the proof of Theorem~\ref{thm:sup-no-oplus} $Q_j$-action is
prevented to act only if for cofinitely many pairs $a,b$ we have that both
$\phi_j(a),\phi_j(b)$ land in the $\Id$-part of $X\oplus Y\oplus Z \oplus
\Id$. But this would give a reduction of $Z\oplus Z$ to $\Id$, contrary to
the fact that $Z \nleq \Id$ because $Z$ is dark (from $F_k^Z$ and
$DZ^j_m$-requirements, which we already know are satisfied). Satisfaction of
the $Q$-requirements shows that
\[
(X\oplus Z) \oplus (Y\oplus Z) \equiv X\oplus Y \oplus Z\oplus
Z\nleq_{\I} X\oplus Y \oplus Z.
\]

$\I$-incomparability of $E_{1}$ and $E_{2}$ is provided by satisfaction of
the $I$-requirements.

It remains to show that $X\oplus Y \oplus Z$ is a least upper bound of
$X\oplus Z$ and $Y\oplus Z$. To see this, we argue as in the proof of
Theorem~\ref{DarkCeersJoinInI}. So assume that $U$ is a ceer such that
$X\oplus Z \leq U\oplus \Id_{n}$ and $Y\oplus Z\leq U\oplus \Id_n$ for some
$n$, via reductions $f,g$ respectively, and let $V=U\oplus \Id_{n}$. Then let
$C=\{(x,y)\mid f(x) \rel{A} g(y)\}$ and $E=C\smallsetminus \{(x,y)\mid
x,y\text{ odd}\}$. It follows that $E$ is finite. Thus, as in the proof of
Theorem~\ref{thm:sup-no-oplus} there is a finite set $D$ such that $(X\oplus
Y\oplus Z)_{/D} \leq V$. But  since $X,Y$ have finite and hence computable
equivalence classes, by Lemma~\ref{ComputableCollapses} we have that
$(X\oplus Y\oplus Z)_{/D} \oplus \Id_k \equiv X\oplus Y\oplus Z$ for some
$k$, and thus $X\oplus Y\oplus Z \leq_{\I} V$. It follows that $X\oplus
Y\oplus Z$ is $\leq_{\I}$ all $\I$-upper bounds of $X\oplus Z$ and $Y\oplus
Z$, and thus is a least $\I$-upper bound.
\end{proof}

\subsection{Summary tables}
Tables~\ref{table:1},~\ref{table:2},and~\ref{table:3} summarize the various
cases when $X \lor Y$ exists, and $X$ is join-irreducible, in the structures
$\Ceers$ and $\Ceers_\I$, as $X, Y$ vary in $\Dark$ and $\Light$. The
differences between $\Ceers$ and $\Ceers_\I$ are highlighted in boldface in
the columns relative to $\Ceers_\I$.

\begin{table}[H]
\begin{minipage}{6cm}
\begin{center}
$\Ceers$
\end{center}
\begin{tabular}{c|c|c}
$R_{X}$ & $R_{Y}$  & $R_{X }\lor R_{Y}$? \\\hline
light & light & NO \\ \hline
dark&dark & NO \\ \hline
&&--Sometimes NO\\
light & dark & --Sometimes Yes\\
\hline
\end{tabular}
\end{minipage}
\begin{minipage}{6cm}
\begin{center}
$\Ceers_\I$
\end{center}
\begin{tabular}{c|c|c}
$R_{X}$ & $R_{Y}$  & $R_{X }\lor R_{Y}$? \\\hline
light & light & NO\\ \hline
dark&dark & NO \\ \hline
&&--Sometimes NO\\
light & dark & --Sometimes Yes\\
\hline
\end{tabular}
\end{minipage}\caption{The problem of the existence of joins in $\Ceers$
and $\Ceers_\I$ for incomparable ceers of the form $R_X$.
}\label{table:1}

\smallskip

\begin{minipage}{6cm}
\begin{center}
$\Ceers$
\end{center}
\begin{tabular}{c|c|c}
$X$ & $Y$  & $X \lor Y$? \\\hline
light &light &--Sometimes NO\\
& &--Sometimes YES\\
&&with $\sup$ $X\oplus Y$\\
& &--Sometimes YES\\
&&with $\sup$ $<X\oplus Y$\\
\\\hline
dark & dark & NO\\
\hline
light &dark &--Sometimes NO\\
& &--Sometimes YES\\
&&with $\sup$ $X\oplus Y$\\
& &--Sometimes YES\\
&&with $\sup$ $<X\oplus Y$\\
\hline
\end{tabular}
\end{minipage}
\begin{minipage}{6cm}
\begin{center}
$\Ceers_\I$
\end{center}
\begin{tabular}{c|c|c}
$X$ & $Y$  & $X \lor Y$? \\\hline
light &light &--Sometimes NO\\
& &--Sometimes YES\\
&&with $\sup$ $X\oplus Y$\\
& &--Sometimes YES\\
& &with $\sup$ $<_{\I}X\oplus Y$\\
\hline
dark & dark & --Sometimes NO\\
& &--Sometimes \textbf{YES}\\
&&with $\sup$ $X\oplus Y$\\
& &--Sometimes \textbf{YES}\\
& &with $\sup$ $<_{\I}X\oplus Y$\\
\hline
light &dark &--Sometimes NO\\
& &--Sometimes YES\\
&&with $\sup$ $X\oplus Y$\\
& &--Sometimes YES\\
&&with $\sup$ $<_{\I} X\oplus Y$\\
\hline
\end{tabular}
\end{minipage}\caption{The problem of the existence of joins in $\Ceers$
and $\Ceers_\I$ for incomparable general ceers.}\label{table:2}

\smallskip

\begin{minipage}{6cm}
\begin{center}
Join-irreducible in $\Ceers$
\end{center}
\begin{tabular}{c|c|c}
 &   & join-irreducible? \\\hline
$R_X$ & dark &YES\\
\hline
$R_X$ & light & - Sometimes YES\\
&& - Sometimes NO\\
\hline
$X$ &dark &YES\\
\hline
$X$ &light &--Sometimes YES\\
& & --Sometimes NO\\
\end{tabular}
\end{minipage}
\begin{minipage}{6cm}
\begin{center}
Join-irreducible in $\Ceers_I$
\end{center}
\begin{tabular}{c|c|c}
 &   & join-irreducible? \\\hline
$R_X$ & dark &YES\\
\hline
$R_X$ & light & - Sometimes YES\\
&& - Sometimes NO\\
\hline
$X$ &dark  &--Sometimes YES\\
&& --Sometimes \textbf{NO}\\
\hline
$X$ &light &--Sometimes YES\\
& & --Sometimes NO\\
\end{tabular}
\end{minipage}\caption{Join-irreducible elements in $\Ceers$
and $\Ceers_\I$.}\label{table:3}
\end{table}

\section{Greatest Lower Bounds}\label{sct_glb}
We now turn our attention to greatest lower bounds. It is known (see e.g.
\cite{ceers}) that the poset $\Ceers$ is not a lower semilattice. An easy way
to witness this fact is by looking at dark ceers, as follows already from
results in earlier sections. For instance:

\begin{obs}\label{obs:nodarkmeet-easy}
There are dark ceers $E_1,E_2$ so that $R\leq E_1,E_2$ implies that $R=\Id_n$
for some $n$.
\end{obs}

\begin{proof}
Let $E_1$ and $E_2$ be two distinct minimal dark ceers: see
Theorem~\ref{MinimalDark}.
\end{proof}

The rest of this section is devoted to studying meets, meet-reducible
elements, and strong minimal covers in the structures of ceers, and ceers
modulo $\equiv_{I}$.

\subsection{The Exact Pair Theorem}
The following theorem provides a useful tool to deal with meets in $\Ceers$
and $\Ceers_{\I}$.

\begin{thm}[Exact Pair Theorem]\label{thm:EPT}
Let $(A_i)_{i \in \omega}$ be a uniformly c.e. sequence of ceers. Then there
exist two ceers $X$ and $Y$ above every $\bigoplus_{i\leq n} A_i$ so that any
ceer $Z$ below both $X$ and $Y$ is below $\bigoplus_{i\leq n} A_i$ for some
$n$. Further, if each of the $A_i$ is dark or in $\I$, then $X$ and $Y$ can
be chosen to be dark as well.
\end{thm}

\begin{proof}
Given a  uniformly c.e. sequence $(A_i)_{i \in \omega}$  of ceers, with
uniformly computable approximations $\{A_{i,s}\mid i, s \in \omega\}$ as in
Section~\ref{ssct:background}, we construct two ceers $X$ and $Y$, with the
desired properties.

\medskip

\emph{The requirements:}  We have the following requirements (the
requirements $D_{m}$ are omitted if not every $A_i$ is dark or in $\I$):
\begin{itemize}
\item[$Q_n$:] There is some column of $X$ and $Y$ which contains $A_n$
    (where, given equivalence relations $S,T$ we say that \emph{the $k$-th
    column of $S$ contains $T$}, if  $\langle k, u \rangle \rel{S} \langle
    k, v \rangle$ if and only if  $u \rel{T} v$, for all $u,v$).
\item[$P_{j,k}$:] If $Z$ is a ceer such that $\phi_j$ witnesses $Z\leq X$ and
    $\phi_k$ witnesses $Z\leq Y$, then for some $n$, $Z\leq
    \bigoplus_{i \leq n}A_i$.
\item[$D_{m}$:] If $W_m$ is an infinite c.e. set, then $x \rel{X} y$ for
    some $x,y\in W_m$ and $x' \rel{Y} y'$ for some $x',y'\in W_m$.
\end{itemize}
Notice that we write $P_{j,k}$, and not $P_{j,k,Z}$, since $Z$ is determined
by $j,k$, as $P_{j,k}$ in fact states that if $\phi_{j}$ and $\phi_{k}$ are
total, and for every $x,y$, $\phi_{j}(x) \rel{X} \phi_{j}(y)$ if and only if
$\phi_{k}(x) \rel{Y} \phi_{k}(y)$, then the ceer $Z$ given by $x \rel{Z} y$
if and only if $\phi_{j}(x) \rel{X} \phi_{j}(y)$, is reducible to some finite
uniform join of the $A_{n}$'s.

We fix the following priority ordering on requirements:
\[
Q_0 < P_0<D_0<Q_1 < P_1< D_1 <\cdots < Q_n < P_n <D_n< \cdots
\]
where we write $P_{\langle j, k \rangle} = P_{j,k}$. The $m$-th requirement
in this ordering will be denoted by $R_{m}$.

Given a requirement $R_{m}$, $\rho_X(m)$ and $\rho_Y(m)$ will denote the
restraint sets for the requirement, and will be comprised of a finite set of
columns of $\omega$, so that $R_{m}$ can neither $X$-collapse nor
$Y$-collapse elements lying in different columns of $\rho_X(m)$ or
$\rho_Y(m)$, or  (unless $R_{m}$ is a $Q$-requirement)
elements in a same
column of $\rho_X(m)$ or $\rho_Y(m)$.  Eventually, each column will contain
$A_n$ for some $n$, or a ceer equivalent to $\Id_n$ for some $n$.

\medskip
\emph{Strategy for $Q_n$:} Actions for requirement $Q_n$  are as follows:
Pick a column $\omega^{[k]}$ in $X$ and $Y$ so that the elements of
$\omega^{[k]}$ are larger than any number yet mentioned. At all future
stages, $Q_n$ restrains (either $X$-restrain or $Y$-restrain according to
whether we code $A_n$ in $X$ or $Y$) the elements of this column, so that no
lower priority requirement can either $X$- or $Y$-collapse distinct elements
in this column. Then $Q_n$ causes collapse of the $i$th and $j$th elements on
this columns if and only if $i\mathrel{A}_n j$.

\medskip

\emph{Strategy for $P_{j,k}$:} Actions for requirement $P_{j,k}$ are as
follows. Denote $\rho_{X}=\rho_{X}(m)$ and $\rho_{Y}=\rho_{X}(m)$, where
$R_{m}=P_{j,k}$. Then  $P_{j,k}$ waits for a pair of numbers $x,y$ so that
the following currently hold:

\begin{enumerate}
\item $\phi_j(x),\phi_j(y),\phi_k(x), \phi_k(y)$ have all converged.
\item Not both of $\phi_j(x)$ and $\phi_j(y)$ are $X$-restrained (i.e.
    $X$-equivalent to elements of $\rho_X$).
\item Not both of $\phi_k(x)$ and $\phi_k(y)$ are $Y$-restrained (i.e.
    $Y$-equivalent to elements of $\rho_Y$).
\item Not both $\phi_j(x)\mathrel{X} \phi_j(y)$ and $\phi_k(x)\mathrel{Y}
    \phi_k(y)$.
\end{enumerate}

Having found such a pair, if still $\phi_k(x) \cancel{\mathrel{Y}} \phi_k(y)$
(if already $\phi_k(x) \mathrel{Y} \phi_k(y)$, but still $\phi_j(x)
\cancel{\mathrel{X}} \phi_j(y)$, then we act symmetrically) then
(\emph{$P_{j,k}$-action}) $P_{j,k}$ $X$-collapses $\phi_j(x)$ and
$\phi_j(y)$, restrains $\phi_k(x)$ and $\phi_k(y)$ in $Y$ (by adding suitable
columns to the restraint set that the requirement passes on to lower-priority
requirements), and initializes all lower priority requirements. If $P_{j,k}$
acts and its restraint is not injured by higher priority requirements then
the action corresponds to a winning outcome. The verification will show that
the other outcome, in which we wait forever for suitable $x,y$, is
comprehensive of the case when not both $\phi_j$ and $\phi_k$ are total (we
wait forever for some computation to end), or otherwise $\phi_j$ and $\phi_k$
induce morphisms into some finite uniform join of the $A_n$. Thus, if they
are both reductions from a same ceer $Z$ we have that $Z$ is reducible to
some finite uniform join of the $A_n$, as desired, so this waiting outcome
fulfills the requirement as well.

\medskip

\emph{Strategy for $D_m$:} Actions for the $D_m$ requirements are as follows.
If $W_{m}$ does not contain already distinct elements $x,x'$ that are already
$X$-equivalent, then $D_m$ waits for elements $x,x'$ to be enumerated into
$W_m$ which are not both not $X$-restrained (i.e. $X$-equivalent to elements
of $\rho_X$) and not as yet $X$-equivalent. If these are found, then
(\emph{$D_m$-action for $X$}) $D_m$ $X$-collapses $x$ and $x'$. Similarly, if
$W_{m}$ does not contain already distinct elements $y,y'$ that are already
$Y$-equivalent, then $D_m$ waits for elements $y,y'$ to be enumerated into
$W_m$ which are not both $Y$-restrained (i.e. $Y$-equivalent to elements of
$\rho_Y$), and not as yet $Y$-equivalent. If these are found, then
(\emph{$D_m$-action for $Y$}) $D_m$ $Y$-collapses $y$ and $y'$. In the
verification, we will use the fact these $\rho_{X}$ and $\rho_{Y}$ are
comprised of finitely many columns that are either finite or dark, and thus
correspond to a finite uniform join of ceers each of which is dark or in
$\I$, so if $W_{m}$ is infinite then it can not be the case that $W_{m}$
contains only pairs which are $X$-non-equivalent and all of its elements lie
in $\rho_{X}$, or $W_{m}$ contains only pairs which are $Y$-non-equivalent
and all of its elements lie in $\rho_{Y}$.

\medskip

\emph{The construction:} At stage $s$, we build, through a computable
procedure, approximations $X_s$ and $Y_s$ to the desired $X, Y$ respectively.
We use the following parameters: $c(n,s)$ approximates the column of $X$ and
$Y$ where $Q_{n}$ tries to code $A_{n}$ at stage $s$, so as to get in the
limit that the $c(n)$-th column of $X$ and $Y$ contains $A_n$. The parameters
$\rho_X(m,s)$ and $\rho_Y(m,s)$ denote the columns in $X$ and $Y$ that are
restrained at $s$ by requirements of higher priority than $R_m$, that is
$R_m$ can not collapse equivalence classes of numbers in these columns.

At stage $s+1$, we say that $P_{j,k}$ \emph{requires attention} if $P_{j,k}$
has not as yet acted after its last initialization, but now $P_{j,k}$ is
ready to act i.e. there exist $x,y$ so that (1) through (4) as in the
description of the $P_{j,k}$-strategy hold, for the current approximations to
the restraints $\rho_X$ and $\rho_Y$ relative to $P_{j,k}$. We say that a
requirement $D_m$ \emph{requires attention} if: $W_{m}$ has not as yet
enumerated elements $x,x',y,y'$ such that already $x \rel{X} x'$ or $y
\rel{Y} y'$ and it is now ready (through suitable $x,x'$) to act for $X$, or
(through suitable $x,x'$) to act for $Y$. After acting at a stage $s$, a
requirement \emph{initializes} all lower priority $P$- and $Q$-requirements
$R$: if we initialize a $Q$-requirement $Q_n$ then we set $c(n,s)$ undefined.
At the end of stage $s+1$, every parameter retains the same value as at the
previous stage, unless explicitly redefined during stage $s+1$.

\medskip

\emph{Stage $0$} Let $X_{0}=Y_{0}=\Id$. Initialize all $P$- and
$Q$-requirements. Go to Stage $1$.

\medskip

\emph{Stage $s+1$} We act in two steps:
\begin{enumerate}
  \item If there is no $P_{j,k}$ with $j, k \le s$ such that $P_{j,k}$
      requires attention, or no $D_m$ with $m \le s$ that requires
      attention, then go directly to step (2). Otherwise, pick the highest
      priority requirement $R_{m}$ that requires attention.

      If $R_{m}=P_{j,k}$ then pick the least pair (under code) $x,y$ as in
      the description of the $P_{j,k}$-strategy; if currently $\phi_k(x)
      \cancel{\mathrel{Y}} \phi_k(y)$ then $X$-collapse $\phi_j(x) \rel{X}
      \phi_j(y)$;  if $\phi_k(x) \mathrel{Y} \phi_k(y)$ but still
      $\phi_j(x) \cancel{\mathrel{X}} \phi_j(y)$ then do nothing; in either
      case, $P_{j,k}$ sets restraints $\rho_X(m+1)$, $\rho_Y(m+1)$
      extending the restraints of higher priority requirements plus the
      columns up to the least column $r$ so that $\phi_{j}(x), \phi_{j}(y),
      \phi_{k}(x), \phi_{k}(y) \in \bigcup_{i \le r} \omega^{[i]}$.

      If $R=D_m$ then pick a suitable pair $x,x'$ or $y, y'$ as in the
      description of the strategy and $X$-collapse $x,x'$ or $Y$-collapse
      $y,y'$ according to which case applies.

      Initialize all $P$- and $Q$-requirements of priority lower than that
      of $R$. Notice that if $R$ is a $P$-requirement, then initialization
      implies also that the unrestrained column (or columns) of which $R$
      has collapsed some class to another class will no longer be used for
      coding by $Q$-requirements.

      Go to step (2).

  \item  Let $n_0$ be the least number such that $c(n_0)$ is currently
      undefined; define $c(n_0)$ to be a fresh number (i.e. all numbers so
      far mentioned in the construction lie in some column $r$ with
      $r<c(n_0)$. Furthermore for each column smaller than this one which
      is not currently chosen for a coding $Q$-requirement, $X$-collapse
      and $Y$-collapse all non-restrained elements within this column, thus
      making the column have only finitely many classes in both $X$ and
      $Y$: we refer to this move as \emph{auxiliary $D$-collapsing}. Next,
      for all $n<n_0$, \emph{code} $A_n$ into the $c(n)$-column of $X$ and
      $Y$,  by collapsing $\langle c(n), u \rangle \rel{X} \langle
      c(n),v\rangle $ and $\langle c(n), u \rangle \rel{Y} \langle c(n),
      v\rangle$ for all $u,v \le s$ such that $u \rel{A_{n,s}} v$. (If one
      sees these collapses as actions of $Q_{n}$ then, strictly speaking, a
      $Q$-requirement acts infinitely often, but it never directly
      collapses numbers from different
       columns.)
\end{enumerate}

After completing (2), we close the stage by performing the following actions:
define $X_{s+1}$ and $Y_{s+1}$ to be the equivalence relations generated by
$X_{s}$, and $Y_{s}$, plus the pairs obtained by the above described
$X$-collapsing and $Y$-collapsing actions, respectively. Go to stage $s+2$.

Notice that contrary to other proofs in the paper, $X_{s+1}$ and $Y_{s+1}$
are not finite extensions of $\Id$, because of the infinite collapses
deriving from the auxiliary $D$-collapsing, which however produces computable
equivalence classes, so that $X_{s+1}$ and $Y_{s+1}$ are computable
equivalence relations, and at the next stage it will be decidable to see
whether two numbers are $X_{s+1}$-equivalent or $Y_{s+1}$-equivalent, if and
when this may be required by the construction.

\medskip

\emph{Verification}: The verification rests on the following lemmata. $X$- or
$Y$- collapses performed by the construction between pairs of numbers from
different columns of $\omega$ will be called \emph{horizontal} collapses
(made by $P$- and $D$-requirements); collapses within a same column will be
called \emph{vertical} collapses (made by $Q$-requirements
or by auxiliary $D$-collapsing).

\begin{lemma}
Each requirement can be initialized only finitely often and (if not a
$Q$-requirement) acts only finitely many times.
\end{lemma}

\begin{proof}
This is a straightforward inductive argument, since a requirement can be
initialized only if some higher priority $P$-requirement or $D$-requirement
acts, and, on the other hand, a $P$-requirement can act only one more time or
a $D$-requirement can act at most two more times, after their last
initializations.
\end{proof}

\begin{lemma}
For each requirement $R_{m}$ there is a least stage $t$ such that at no $s
\ge t$ does either $\rho_{X}(m+1)$ or $\rho_{Y}(m+1)$ change.
\end{lemma}

\begin{proof}
Obvious by induction on $m$.
\end{proof}

\begin{lem}\label{DarkColumns}
If each $A_{i}$ is dark or in $\I$ then there is no infinite c.e. set $W_n$
contained in the $X$-closure or the $Y$-closure of a single column, so that
if $W_n$ is infinite then $W_n$ gives pairwise $X$-inequivalent or pairwise
$Y$-inequivalent elements which avoid the restraints imposed by higher
priority requirements.
\end{lem}

\begin{proof}
If the $j$-th column of $X$ is encoding a ceer $A_n$, then this follows from
the assumption that each $A_n$ is dark. Otherwise, the first time that a
coding column $>j$ is chosen, $X$ is collapsed to have only finitely many
classes in the $j$-th column. Thus $W_n$ cannot be an infinite set of
pairwise $X$-inequivalent elements on the $j$-th column of $X$. The same
argument holds for $Y$.
\end{proof}

\begin{lemma}
Each requirement is satisfied.
\end{lemma}

\begin{proof}
We first consider requirement $Q_n$. Let $s$ be the stage at which we choose
the final value $c(n,s)$: then for every $t\ge s$, $c(n,t)=c(n,s) =c(n)$.
Since $c(n)$ is chosen larger than any number mentioned so far, this column
has no higher priority restraint either by $X$ or $Y$. Since no lower
priority requirement can cause $X$-collapse or $Y$-collapse \emph{on} this
column (although it could collapse numbers from bigger columns to this
column), and $Q_n$ causes collapse on this column exactly to correspond to
collapse in $A_n$, we see that $A_n\leq X$ and $A_n\leq Y$ via the computable
mapping $m\mapsto\langle c(n) ,m\rangle$.

\medskip

We now consider requirement $P_{j,k}$. Let $s$ be the least stage after which
$P_{j,k}$ is never re-initialized. By a previous lemma, all restraints
imposed by higher priority requirements have  stabilized, giving sets
$\rho_X$ for $X$, and $\rho_Y$ for $Y$, say $\rho_X= \bigcup_{i \le n}
\omega^{[i]}$, for some $n$ and $\rho_Y= \bigcup_{i \le m} \omega^{[i]}$, for
some $m$.

We first consider the case where $P_{j,k}$ acts at some stage $t\geq s$,
i.e.~successfully finds a pair $x,y$ satisfying conditions (1)-(4): suppose
at stage $t$ we have that $\phi_{k}(x) \cancel{\rel{Y}} \phi_{k}(y)$, the other case being
similar. Then by placing the $Y$-restraint on $\phi_k(x)$ and $\phi_k(y)$,
and by initialization, we will have $\phi_k(x)\cancel{\mathrel{Y}}\phi_k(y)$,
but $\phi_j(x)\mathrel{X} \phi_k(y)$. This satisfies the requirement.

Now, suppose at no stage $t\geq s$ do we find a pair $x,y$ satisfying
conditions (1)-(4). We may also assume that $\phi_j$ and $\phi_k$ are total,
and $\phi_{j}, \phi_{k}$ are reductions of the same ceer $Z$ (i.e. for all
$x,y$, $x \rel{Z} y$ if and only if $\phi_{j}(x) \rel{X} \phi_{j}(y)$ if and
only if $\phi_{k}(x) \rel{Y} \phi_{k}(y)$): otherwise, $P_{j,k}$ is
satisfied. If, for every $x$, $\phi_j(x)$ is in $[\rho_X]_{X}$, then we claim
that $P_{j,k}$ is satisfied. To see this,  let $I=\{i\le n\mid \textrm{the
$i$-th column of $X$ codes some $A_{r}$}\}$: in other words $I$ is comprised
of the numbers $c(r) \le n$. Recall that if $i, j\in I$ are distinct then (by
initialization: see remark at the end of part (1) of Stage $s+1$) the
$X$-equivalence classes of elements in $\omega^{[i]}$ are disjoint from the
$X$-equivalence classes of elements in $\omega^{[j]}$.
By our coding, if $r \in I$,
we have that $\langle r, x \rangle \rel{X} \langle r, y\rangle$
if and only if $x \mathrel{A_{n_r}} y$ where
$c(n_{r})=r$.

After $s$ no requirement (which is not a $Q$-requirement) of higher priority
than $P_{j,k}$ acts, and thus no pair of $X$-classes of elements in
$\rho_{X}$ coming from distinct columns are $X$-collapsed by actions of
requirements of higher priority than $P_{j,k}$ (no more horizontal collapses,
as the higher-priority requirements no longer act, and the lower-priority
requirements do not break the restraint $\rho_X$), as the $Q$-requirements do
not collapse classes of numbers from different columns (only vertical
collapses in this case).  Let $\beta_{1}, \ldots, \beta_{h}$ be the distinct equivalence classes
at stage $s$ of elements lying in $\bigcup_{i \le n}\omega^{[i]}$, not
containing elements in any column $\omega^{[r]}$ with $r \in I$: restricted
to $\bigcup_{i \le n}\omega^{[i]}$ these equivalence classes, and their
number, will no longer change.

Consider the following algorithm to compute a function $f$. On input $x$
search for the first $\langle r, u\rangle$ with $r\le n$ such that
$\phi_{j}(x) \rel{X} \langle r, u \rangle$: if $\langle r, u \rangle \in
\beta_i$, for some $i$, then map $x$ to $\langle n+i, 0\rangle$;  otherwise
search for $\langle r', u'\rangle$ with $r' \in I$ such that $\langle r,u
\rangle \rel{X} \langle r', u' \rangle$ and map $x$ to $\langle n_{r'},
 u' \rangle$.

Clearly, for every $x$, $f(x)$ is defined and codes a pair whose first
projection is a number $i\le n+h$. It is not difficult to see (using that
each $n_{r} \le c_{n_{r}} \le n$ and that there is no more horizontal
collapse within $\rho_{X}$ after $s$) that for all $x,y$, $\phi_{j}(x)
\rel{X} \phi_{j}(x)$ if and only if there are $i\le n+h$  and $u,v$ such that
$f(x)= \langle i, u\rangle$, $f(y)= \langle i, v\rangle$ and $u \rel{A_{i}}
v$. From this it easily follows that $Z\leq \bigoplus_{i \le n+h} A_i$.

Now, suppose $x$ is some number so that $\phi_j(x)$ is not $X$-equivalent to
any element of $\rho_X$. Then, for every $y$ and every stage $t \geq s$ where
$\phi_k(y)\downarrow $, we must have $\phi_k(y)\in [\rho_Y]_{Y}$ or, at $t$,
$\phi_k(x)\mathrel{Y} \phi_k(y)$. But then for every $y$, $\phi_k(y) \in
[\bigcup_{i\le p} \omega^{[i]}]_{Y}$, for the least $p$ such that $\rho_Y
\subseteq \bigcup_{i\le p} \omega^{[i]}$ and $\phi_k(x) \in [\bigcup_{i\le p}
\omega^{[i]}]_{Y}$. An argument similar to the previous one shows now that
$Z$ is reducible to a finite join of the $A_i$.

\medskip

We finally consider requirement $D_m$. It is subject to restraint of only
finitely many columns: if $W_m$ enumerates an infinite set so that $a,b\in
W_m$ implies $a\cancel{\rel{X}}b$, then it follows that $W_m$ is not
contained in any finite number of columns (as each column does not have such
a set by Lemma \ref{DarkColumns}). Thus, eventually $D_m$ will find $x,x'$
and $y,y'$ as needed, and an $X$-collapse of $x,x'$ and a $Y$-collapse of
$y,y'$ will permanently satisfy $D_m$.

\end{proof}
This completes the proof of the Exact Pair Theorem.
\end{proof}

\begin{cory}\label{cor:to-Exact}
If the uniformly c.e. sequence $(A_i)_{i \in \omega}$ consists of finite
ceers and dark ceers, then we can get $X,Y$ dark.
\end{cory}

\begin{proof}
This follows from the proof of the Exact Pair Theorem.
\end{proof}

In the rest of the paper when we apply the Exact Pair Theorem to get dark
ceers $X,Y$ then we will say that we appeal to the ``dark Exact Pair
Theorem''.

\subsection{Meet-irreducibility and self-fullness}
Our first application of the Exact Pair Theorem is to show that the
meet-irreducible ceers coincide with the self-full ceers.

\begin{thm}\label{LightMeets}
Let $E$ be any ceer. Then $E$ is non-self-full if and only if there exists
incomparable $X,Y$ so that $E$ is a greatest lower bound of $X$ and $Y$.
\end{thm}

\begin{proof}
If $E$ is a self-full degree, then $E$ cannot be a greatest lower bound of
any pair of ceers: Given any $E_1,E_2$ both $>E$, then by
Lemma~\ref{StrongMinimalCoversOfSF} $E\oplus \Id_1$ is a ceer strictly above
$E$ and still below $E_1$ and $E_2$.

For the converse, suppose $E$ is non-self-full. Let $(A_i)_{i\in \omega}$ be
the uniform sequence with $A_0=E$ and $A_j=\Id_1$ for every $j>0$. Let $X, Y$
be as guaranteed by the Exact Pair Theorem. Then $E\leq X,Y$, and $Z\leq X,Y$
implies $Z\leq E\oplus \Id_n$ for some $n$. But since $E$ is non-self-full,
$E\oplus \Id_n\equiv E$.
\end{proof}

\subsection{Strong minimal covers}
In this section we study minimal covers and strong minimal covers. Recall
that in $\Ceers$ every self-full ceer $E$ has exactly one strong minimal cover,
which is the least degree of the set of degrees strictly above $E$. The next
result follows immediately from Theorem \ref{SelfFullMinimalCoversInI}. We
repeat it here to record the consequence in the context of our discussion of
the $\equiv_\I$-degrees.

\begin{thm}\label{StrongMinimalCoversNSF}
Every non-universal $\equiv_\I$-degree $R$ has infinitely many incomparable
self-full strong minimal covers. Moreover if $R$ is dark, then it has
infinitely many incomparable dark strong minimal covers.
\end{thm}

\begin{proof}
Let $R$ be a non-universal ceer: notice that by
Observation~\ref{obs:Icontiguous} being non-universal is the same as being
$\I$-non-universal. Then by Theorem~\ref{SelfFullMinimalCoversInI} there are
infinitely many incomparable self-full ceers $E_l$ above $R$ satisfying the
properties stated in Theorem~\ref{SelfFullMinimalCoversInI}, and they are
built so that no $E_l$-equivalence class is computable. Thus the $E_l$ are
$\leq_{\mathcal{I}}$-incomparable: indeed if $E_l \leq_{\mathcal{I}} E_{l'}$
then $E_l \leq E_{l'} \oplus \Id_n$ for some $n$ but, being undecidable, no
$E_l$ class can be mapped to a class in the $\Id_n$-part; so, in fact $E_l
\leq E_{l'}$ which is a contradiction. As to show that each of the $E_l$ is a
strong minimal cover of $R$ in the $\mathcal{I}$-degrees, suppose that $X
\leq_{\mathcal{I}} E_l$, hence $X\le E_l \oplus \Id_n$, for some $n$, via a
reduction $f$ such that all classes in the $\Id_n$-part are in the range of
$f$. By Lemma~\ref{coproduct3}(2) let $X_0$ be such that $X \equiv X_{0}
\oplus \Id_{n}$ and $X_0 \le E_l$. If $X_0 \equiv E_l$ then $X\equiv E_{l}
\oplus \Id_{n}$, giving $X\equiv_{\I} E_{l}$. Otherwise, $X_0 < E_l$, but
then (by the properties of the ceers $E_{l'}$ provided by
Theorem~\ref{SelfFullMinimalCoversInI}) $X_0 \le R \oplus \Id_k$ for some
$k$, giving $X\equiv X_0 \oplus \Id_n \le R \oplus \Id_k \oplus \Id_n$,
giving $X \leq_{\mathcal{I}} R$.

The latter part of the statement about dark strong minimal covers comes by
the same argument using Corollary~\ref{cory:selfdarkminimalcovers}.
\end{proof}

\begin{cory}
Suppose $Y>X$ has the property that $Z>X$ implies $Z\geq Y$. Then $Y$ is a
strong minimal cover of $X$.
\end{cory}

\begin{proof}
If $X$ and $Y$ have this property then $X$ is not branching, so is self-full
by Theorem~\ref{LightMeets}. Then up to $\equiv$ there is a unique such $Y$,
namely $X\oplus \Id_1$. We know that $X\oplus \Id_1$ is a strong minimal
cover by Lemma~\ref{StrongMinimalCoversOfSF}.
\end{proof}

Non-self-full non-universal ceers are not only meet-reducible, but in fact we
can prove the following stronger fact.

\begin{cory}\label{cor:non-sf-inf-min-covers}
If $X$ is non-self-full and non-universal, then there are infinitely many
incomparable self-full strong minimal covers of $X$.
\end{cory}

\begin{proof}
If $X$ is non-self-full and non-universal, then by Theorem
\ref{SelfFullMinimalCoversInI} $X$ has infinitely many incomparable self-full
ceers $E_l$ above it. Now, if $Y<E_l$ then by properties of $E_l$ established
in that theorem we have that $Y\leq X \oplus \Id_n$ for some $n$. But by
self-fullness we have (Observation~\ref{obs:nonselffull}) $X \equiv X\oplus
\Id_n$, thus $Y\leq X$.
\end{proof}

Notice that, for $X$ non-self-full every strong minimal cover $Y$ of $X$ does
not have the property that $Z>X$ implies $Z\geq Y$.

\begin{q}
Does every non-self-full ceer have a non-self-full strong minimal cover? Does
every non-self-full ceer have infinitely many incomparable non-self-full
strong minimal covers?
\end{q}

Note that a positive answer to the second form of this question would give an
embedding $F$ of $\omega^{<\omega}$ into the ceers where $F(\sigma i)$ is a
strong minimal cover over $F(\sigma)$.

\subsection{Branching and non-branching}
Theorem~\ref{LightMeets} shows that we have meet-irreducible
(also called non-branching) elements in $\Ceers$, by showing that they
coincide with the self-full-ceers. Since every dark ceer is self-full we have
that there exist meet-irreducible dark ceers. But also:

\begin{cory}\label{cor:light-meet-irr}
There is a light degree $E$ which is not a greatest lower bound of any
incomparable degrees.
\end{cory}

\begin{proof}
Theorem \ref{LightSelfFulls} shows that there are light self-full degrees
$E$. It follows from Theorem~\ref{LightMeets} that $E$ cannot be a
meet of any incomparable pair of degrees.
\end{proof}

Contrary to this, Theorem~\ref{StrongMinimalCoversNSF} shows that in
$\Ceers_\I$ every non-universal element is meet-reducible (also called
branching).

\begin{cory}\label{DarkBranching}
Every non-universal $\equiv_\I$-degree $E$ is branching. Moreover if $E$ is
dark then in the $\I$-degrees $E$ is a meet of two incomparable dark
$\I$-degrees.
\end{cory}

\begin{proof}
Directly from Theorem~\ref{StrongMinimalCoversNSF}.
\end{proof}

The following observation gives an alternative proof of branching of
self-full ceers in $\Ceers_\I$, using the Exact Pair Theorem.

\begin{obs}\label{DarkBranching-bis}
For any self-full ceer $E$, there are ceers $X,Y\geq E$ and  $>_\I E$ so that
$R\leq_\I X$ and $R\leq_\I Y$ if and only if $R\leq_\I E$. Further, if $E$ is
dark, we can choose $X$ and $Y$ to be dark as well.
\end{obs}

\begin{proof}
Apply the Exact Pair Theorem to the sequence $(A_i)_{i \in \omega}$ where
$A_0=E$ and $A_i=\Id_1$ for $i>0$. Let $X,Y$ be as produced by the Exact Pair
Theorem. Clearly $X,Y>_\I E$. Then $R\leq_{\I} X,Y$ implies that for some $n$
there are reductions $f$ and $g$ witnessing that $R\leq X\oplus \Id_n$ and
$R\leq Y\oplus \Id_n$. By Lemma~\ref{coproduct3}(3) there exists a ceer $R_0$
so that $R_0\leq X,Y$ and $R\leq R_0\oplus \Id_{2n}$. Then $R_0\leq
\bigoplus_{i\leq m} A_i$ for some $m$, from which $R_0 \leq E \oplus
\Id_{m-1}$, so $R\leq E\oplus \Id_{m-1+2n}$, showing that $R\leq_\I E$. The
claim about dark ceers follows by Corollary~\ref{cor:to-Exact}.
\end{proof}

Corollary~\ref{cor:dark-light-for-darks} below shows that for dark ceers in
$\Ceers_\I$ we can have also dark-light branching. We first need the
following theorem.

\begin{thm}\label{LightDarkMeets}
If $X,Y$ are dark ceers with a greatest lower bound $Z$ in the
$\equiv_\I$-degrees, then $X,Y\oplus \Id$ also have a greatest lower bound
$Z$ in the $\equiv_\I$ degrees.

Conversely, if $X,Y$ are dark and $Z$ is a greatest lower bound of $X,Y\oplus
\Id$ in the $\I$-degrees, then $Z$ is a greatest lower bound of $X,Y$ in the
$\I$-degrees.
\end{thm}

\begin{proof}
Let $X,Y$ be dark ceers. Assume first that $Z$ is a greatest lower bound of
$X,Y$ in the $\equiv_\I$ degrees, and let $E\leq_I X,Y\oplus \Id$. Then
$E\leq X\oplus \Id_n, Y\oplus \Id \oplus \Id_n$ for some $n$. But $E$ is dark
since $E\leq X\oplus \Id_n$, so at most finitely many elements of $\Id$ can
be in the range of $E$ in the reduction to $Y\oplus \Id \oplus \Id_n$. Thus
$E\leq X\oplus \Id_k, Y\oplus \Id_k$ for some $k$. Thus $E\leq_\I X,Y$. Thus
$E\leq_\I Z$.

Now suppose $Z$ is a greatest lower bound of $X,Y\oplus \Id$ in the
$\I$-degrees. Since $Z$ is dark, as above $Z\leq_\I Y$, so $Z\leq_\I X,Y$.
Given any $R\leq_\I X,Y$, we have that $R\leq_\I X,Y\oplus \Id$, so $R\leq_\I
Z$.
\end{proof}

\begin{cory}\label{cor:dark-light-for-darks}
We have dark-light branching in the $\equiv_\I$-degrees of dark ceers: For
any dark ceer $X$, there is a dark ceer $A$ and a light ceer $B$, both
$\geq_{\mathcal{I}} X$, so that $R\leq_\I A$ and $R\leq_\I B$ if and only if
$R\leq_\I X$. In particular, many pairs of light and dark ceers have meets.
\end{cory}

\begin{proof}
This follows directly from Theorem \ref{LightDarkMeets} and Corollary
\ref{DarkBranching}.
\end{proof}

\subsection{Meets in $\Ceers$ and $\Ceers_\I$}
We complete our study of meets (i.e. pairs with a meet, pairs without a meet,
etc.) comparing the two structures $\Ceers$ and $\Ceers_\I$. First of all
notice that meets in $\Ceers$ become meets in $\Ceers_{\I}$ as well. Indeed,
the following lemma parallels the analogous result for joins, see
Lemma~\ref{JoinsPreservedByI}.

\begin{lem}\label{IPreservesMeets}
If $B,C$ are ceers with a greatest lower bound $E$, then $E$ is also a
greatest lower bound for $B,C$ in the $\equiv_\I$-degrees.
\end{lem}

\begin{proof}
Let $Z\leq_\I B,C$. Then $Z\leq B\oplus \Id_n, C\oplus \Id_n$ for some $n$.
By Lemma~\ref{coproduct3}(3) there exists some $Z_0$ such that
$Z_0\equiv_\I Z$ and $Z_0\leq B,C$. Thus $Z_0\leq E$, showing that $Z\leq_\I
E$.
\end{proof}

\subsubsection{Dark ceers and infima in $\Ceers$ and $\Ceers_{\I}$}
As an easy consequence of Theorem~\ref{LightMeets}, we can dualize
Theorem~\ref{NoJoinOfDark}: in fact, a stronger result holds.

\begin{thm}\label{thm:nodarkmeet}
If $E_1$ and $E_2$ are incomparable and at least one is dark, there is no
greatest lower bound of $E_1$ and $E_2$.
\end{thm}

\begin{proof}
Any potential greatest lower bound would have to be $<$ a dark degree, thus
dark, so self-full, contradicting Theorem~\ref{LightMeets}.
\end{proof}

However, contrary to what happens in $\Ceers$ as highlighted by
Theorem~\ref{thm:nodarkmeet}, we may have infima of incomparable dark ceers
in $\Ceers_{\I}$:

\begin{cory}\label{cor:some-dark-with-inf-I}
There are pairs of dark degrees which have a greatest lower bounds in the
$\equiv_\I$-degrees.
\end{cory}

\begin{proof}
Theorem~\ref{DarkBranching} allows us to take $X,Y$ dark so that $E$ is a
greatest lower bound of $X,Y$ in the $\I$-degrees, if we start with $E$ dark.
\end{proof}

On the other hand there are cases of incomparable dark ceers with no infimum
in $\Ceers_{\I}$. To see this, we first need the following lemma.

\begin{lemma}\label{ProperlyIncomparable}
There is an infinite uniformly c.e. family of dark ceers $(A_i)_{i\in
\omega}$ so that each $A_j$ is not $\leq_\I\bigoplus_{i<j} A_i$.
\end{lemma}

\begin{proof}
Let $A_0$ be any dark ceer. Let $A_{n+1}$ be one ceer constructed by Theorem
\ref{MinimalDark} from $R=\bigoplus_{i\leq n} A_i$. Hence  $A_{n+1} \nleq R$:
on the other hand, it can not be $A_{n+1} \leq_{\mathcal{I}} R \oplus \Id_n$
for any $n$ since the equivalence classes in $A_{n+1}$ are not computable.

As Theorem \ref{MinimalDark} is uniform, this is a uniformly c.e. family of
dark ceers so that each $A_j$ is not $\leq_\I\bigoplus_{i<j} A_i$.
\end{proof}

\begin{thm}\label{DarkNonMeets}
Not all pairs of dark degrees have a greatest lower bound in the
$\equiv_\I$-degrees.
\end{thm}

\begin{proof}
By Lemma~\ref{ProperlyIncomparable} take a uniformly c.e. family of dark
ceers $(A_i)_{i\in \omega}$ so that each $A_j$ is not $\leq_\I\bigoplus_{i<j}
A_i$. By the dark Exact Pair Theorem let $B,C$ (above all $A_i$) be dark
ceers so that $X\leq B,C$ if and only if $X\leq \bigoplus_{i\leq n} A_i$ for
some $n$. Given any ceer $E\leq_\I B,C$, by Lemma~\ref{coproduct3}(3) we
have that $E\equiv E_0\oplus \Id_k$ for some $k$ and $E_0$ such that $E_0\leq
B,C$. Then $E_0\leq \bigoplus_{i\leq n} A_i$ for some $n$. Therefore,
$A_{n+1}$ is a ceer which is below $B,C$ but is not $\leq_\I E_0$, so
$A_{n+1}\not\leq_\I E$. This shows that $E$ cannot be a greatest lower bound
of $B$ and $C$ in the $\equiv_{\mathcal{I}}$-degrees.
\end{proof}

\subsubsection{Light/dark pairs and infima in $\Ceers$ and $\Ceers_{\I}$}
In the structure of ceers we know already:

 \begin{cory}\label{cor:light-dark-meet}
There are pairs consisting of one light and one dark ceer which do not have
any meet in the degrees.
\end{cory}

\begin{proof}
By Theorem~\ref{thm:nodarkmeet}.
\end{proof}

This extends also to $\Ceers_{\I}$:

\begin{cory}\label{cor:light-dark-meetI}
There are pairs consisting of one light and one dark ceer which do not have
any meet in the $\equiv_\I$ degrees.
\end{cory}

\begin{proof}
This follows directly from Theorem \ref{LightDarkMeets} and Theorem
\ref{DarkNonMeets}
\end{proof}

\subsubsection{Light ceers and infima}
We now turn to considering meets of pairs of light degrees.

\begin{cory}\label{cor:light-light-meet}
Some pairs of light ceers have a greatest lower bound.
\end{cory}

\begin{proof}
This is a consequence of Theorem~\ref{LightMeets} as many light degrees are
non-self-full, such as any degree of the form $X\oplus \Id$.
\end{proof}

\begin{cory}\label{cor:some-light-have-inf-I}
Some pairs of light ceers have greatest lower bounds in the
$\equiv_\I$-degrees.
\end{cory}

\begin{proof}
This follows directly from Corollary~\ref{cor:light-light-meet} and Lemma
\ref{IPreservesMeets}.
\end{proof}

\begin{thm}\label{NoLightMeet}
There are pairs of light ceers with no greatest lower bound in $\equiv_\I$.
\end{thm}

\begin{proof}
Let $(B_i)_{i\in \omega}$ be a uniformly c.e. family of dark ceers so that
$B_j\not\leq_\I \bigoplus_{i<j} B_i$ as constructed in Theorem
\ref{ProperlyIncomparable}.

Then we apply the dark Exact Pair Theorem to the sequence $(A_i)_{i\in
\omega}$ defined by $A_0=\Id$ and $A_i=B_{i-1}$ for $i>0$. This gives a pair
of ceers $X,Y$ so that $Z\leq X,Y$ if and only if $Z\leq \Id\oplus
\bigoplus_{i<n} B_i$ for some $n$. Then we need to show that $B_n\not\leq_\I
\Id\oplus \bigoplus_{i<n} B_i$. Otherwise, we would have $B_n\leq
\Id\oplus\bigoplus_{i<n} B_i$ (as the $\Id$ absorbs the extra $\Id_k$ where
$k$ is such that $B_n \leq \Id\oplus \bigoplus_{i<n} B_i \oplus \Id_k$ if we
assume that $B_n \leq_\I \Id\oplus \bigoplus_{i<n} B_i$: this is a trivial
consequence of Lemma~\ref{obs:coproduct2} by taking an infinite strong
effective transversal for a reduction $\Id_k \leq \Id$, obtaining $\Id_k
\oplus \Id \leq \Id$), and thus, since $B_n$ is dark, we have $B_n\leq
\Id_k\oplus \bigoplus_{i<n} B_i$ for some $k$, showing $B_n\leq_\I
\bigoplus_{i<n} B_i$, but this is known to be false by choice of the $B_i$.
\end{proof}

\begin{cory}\label{cor:two-light-no-inf}
There are two light ceers with no greatest lower bound.
\end{cory}

\begin{proof}
This follows directly from Theorem \ref{NoLightMeet} and Lemma
\ref{IPreservesMeets}.
\end{proof}

\subsection{Summary tables}
Tables~\ref{table:4}~and~\ref{table:6} summarize the various cases of when $X
\wedge Y$ exists, and $X$ is meet-irreducible, as $X,Y$ vary in the classes
$\Dark$ and $\Light$. Table~\ref{table:5} summarizes some of the results
about strong minimal covers. The differences between $\Ceers$ and $\Ceers_\I$
are highlighted in boldface in the columns relative to $\Ceers_\I$

\begin{table}[H]
\begin{minipage}{6cm}
\begin{center}
$\Ceers$
\end{center}
\begin{tabular}{c|c|c}
$X$ & $Y$  & $X \wedge Y$? \\\hline
light &light &--Sometimes NO\\
& &--Sometimes YES\\
\hline
dark & dark & NO\\
\hline
light &dark &NO\\
\end{tabular}
\end{minipage}
\begin{minipage}{6cm}
\begin{center}
$\Ceers_\I$
\end{center}
\begin{tabular}{c|c|c}
$X$ & $Y$  & $X \wedge Y$? \\\hline
light &light &--Sometimes NO\\
& &--Sometimes YES\\
\hline
dark & dark &--Sometimes NO\\
& &--Sometimes \textbf{YES}\\
\hline
light &dark &--Sometimes NO\\
& &--Sometimes \textbf{YES}\\
\end{tabular}
\end{minipage}\caption{The problem of the existence of $\wedge$ in $\Ceers$
and $\Ceers_\I$ for incomparable general ceers.}\label{table:4}

\bigskip

\begin{minipage}{6cm}
\begin{center}
Strong minimal covers\\ in $\Ceers$\\
for non-universal ceers
\end{center}
\begin{tabular}{c|c|c}
 &   & infinitely\\
 && many? \\\hline
 $X$ &self-full & NO (only one)\\
\hline
$X$ &non-self-full &YES\\
\end{tabular}
\end{minipage}
\begin{minipage}{6cm}
\begin{center}
Strong minimal covers\\ in $\Ceers_\I$\\
for non-universal ceers
\end{center}
\begin{tabular}{c|c}
 &    infinitely\\
 & many?\\
\hline
$X$ & \textbf{YES}\\
\end{tabular}
\end{minipage}\caption{Strong minimal covers in $\Ceers$
and $\Ceers_{\I}$. In $\Ceers$ every self-full has exactly a strong minimal
cover, which is the least of all degrees strictly above it.}\label{table:5}

\bigskip

\begin{minipage}{6cm}
\begin{center}
Non-universal meet-irreducible ceers in $\Ceers$
\end{center}
\begin{tabular}{c|c|c}
 &   & meet-irreducible? \\\hline
$X$ &dark &YES\\
\hline
$X$ &light &--Sometimes NO\\
& &--Sometimes YES\\
\end{tabular}
\end{minipage}
\begin{minipage}{6cm}
\begin{center}
Non-universal meet-irreducible ceers in $\Ceers_I$
\end{center}
\begin{tabular}{c|c|c}
 &   & meet-irreducible? \\\hline
$X$ &dark &\textbf{NO}: meet of \\
&&(dark, light)\\
&&(dark, dark)\\
$X$ &light &\textbf{NO}\\
\end{tabular}
\end{minipage}\caption{Meet-irreducible elements in $\Ceers$
and $\Ceers_{\I}$. Notice that there are self-full light ceers (hence
meet-irreducible) and non-self-full light ceers (hence meet-reducible), as by
Theorem~\ref{LightMeets} the meet-irreducible ceers coincide with the
self-full ones.}\label{table:6}
\end{table}

\subsection{Minimal tuples}
For $n \ge 1$, a \emph{minimal dark $n$-tuple} is an $n$-tuple of dark ceers
so that every ceer $\leq$ every member of the tuple is in $\I$. Minimal dark
$n$-tuples trivially exist for every $n\ge 1$,  since by
Theorem~\ref{MinimalDark} there exist infinitely many minimal dark ceers, so
any $n$-tuple chosen from among these minimal dark ceers is a minimal
$n$-tuple, but in this case any sub-$k$-tuple, $1\le k <n$, of this $n$-tuple
is also minimal. The following theorem shows that this has not always to be
the case.

\begin{thm}
For every $n\geq 2$, there is a minimal dark $n$-tuple which does not contain
a minimal dark $(n-1)$-tuple.
\end{thm}

\begin{proof}
(Sketch.) We construct dark ceers $R_1,\ldots R_n$ to be a minimal $n$-tuple.
In order to ensure that no $n-1$-sub-tuple is a minimal tuple, we construct
dark ceers $E_X$ for each $X\subseteq \{1,\ldots , n\}$ of size $n-1$ and we
let $R_i=\bigoplus_{i\in X} E_X$: as the $E_X$ are dark it follows that each
$R_i$ is dark as well by Observation~\ref{obs:dark-closure}. This ensures
that each $E_X$ is below $R_i$ if $i\in X$. Thus $\{R_i\mid i\in X\}$ does
not form a minimal $n-1$-tuple. Each $R_i$ is the $\oplus$-sum of $n-1$
addenda of the form $E_X$: we identify each $X$ such that $i \in X$ with its
canonical index, so that the $k$-th addendum in the $\oplus$-sum giving $R_i$
corresponds to the $k$-th canonical index in order of magnitude. Thus $u
\rel{R_i} v$ if and only if $u=v=k \textrm{ mod}_{n-1}$ for some $k<n-1$, and
$\hat{u} \rel{E_{X_k}} \hat v$, where for any $x$, $\hat{x}$ is the quotient
of $x$ by $n-1$ (i.e.\ $x=\hat{x}(n-1)+j$ for some $0\le j<n-1$): see
Section~\ref{ssct:uniform-joins} for the definition of a uniform join with
finitely many addenda. We will say that two numbers $u,v$ are in $(k,E_X)$ if
$u=v=j \textrm{ mod}_{n-1}$, for some $j<n-1$, and the $j$-th addendum of
$R_k$ is $E_X$ (hence $k \in X$).

Our requirements to build the desired ceers $E_X$ are as follows, where the
sequence $\vec{i}= (i_1, \ldots, i_n)$ varies on all sequences of $\omega$
having $n$ elements:

\begin{itemize}
\item[$P_{\vec{i}}$:] (where $\vec{i}= (i_1, \ldots, i_n)$) If the
    reduction $\phi_{i_k}$ to $R_k$ gives the same ceer $Z$ reduced to
    $R_k$ for each $k$, then $Z$ has only finitely many classes.
\item[$Q^k_X$:] $E_X$ has at least $k$ classes.
\item[$D^k_X$:] If $W_k$ is infinite, then there are $x,y\in W_k$ such that
    $x \rel{E_X} y$.
\end{itemize}

As usual the requirements are given a computable priority ordering of order
type $\omega$. We describe the strategies to meet the requirements. Each
requirement may restrain pairs from being $E_X$-collapsed for some $X$. We
will see that each requirement will place only a finite restraint.

\medskip

\emph{Strategy for $P_{\vec{i}}$:} Since for $P_{\vec{i}}$ only finitely many
classes in each $E_X$ are restrained, we wait for $\phi_{i_1}$  to converge,
to be in the same $(1, E_{X})$, and have $\phi_{i_1}(x),\phi_{i_1}(y)$ not as
yet $R_1$-equivalent (i.e. the pair $\widehat{\phi_{i_1}(x)}$,
$\widehat{\phi_{i_1}(y)}$ not yet $E_{X}$-equivalent), and
$\widehat{\phi_{i_1}(x)}$, $\widehat{\phi_{i_1}(y)}$ not restrained in $E_X$
by higher priority requirements. (Note that, since each $E_X$ has infinitely
many equivalence classes, if the wait never ends then either $\phi_{i_1}$ is
not total or the range of $\phi_{i_1}$ hits finitely many $E_X$-classes for
every $X$ with $1 \in X$ by the pigeonhole principle, so that if $\phi_{i_1}$
is a reduction then the image of this reduction in $R_1$ would be finite.)
Now we restrain $\widehat{\phi_{i_1}(x)}$, $\widehat{\phi_{i_1}(y)}$ in
$E_X$. Pick $j\notin X$. We now wait for
$\phi_{i_j}(x),\phi_{i_j}(y)\downarrow$. Note that $E_X$ does not appear in
the $\oplus$-sum which forms $R_j$. Once these computations converge (or if
they have already converged), we diagonalize: If $\phi_{i_j}(x) \rel{R_j}
\phi_{i_j}(y)$, then we simply maintain our restraint in $E_X$. If
$\phi_{i_j}(x) \cancel{\rel{R_j}} \phi_{i_j}(y)$, then we $R_j$-restrain
$\phi_{i_j}(x)$, $\phi_{i_j}(y)$ (no problem if $\phi_{i_j}(x)$,
$\phi_{i_j}(y)$ hits different addenda of $R_{j}$, as in this case they can
never become $R_{j}$-equivalent; otherwise, if $\phi_{i_j}(x), \phi_{i_j}(y)$
are in the same $(j, E_{Y})$ then we $E_Y$-restrain
$\widehat{\phi_{i_j}(x)}$, $\widehat{\phi_{i_j}(y)}$), and we $E_X$-collapse
$\widehat{\phi_{i_1}(x)}$, $\widehat{\phi_{i_1}(y)}$, so that $\phi_{i_1}(x)
\rel{R_1} \phi_{i_1}(y)$ and $\phi_{i_j}(x) \cancel{\rel{R_j}}
\phi_{i_j}(y)$. This ensures that there is no $Z$ such that $\phi_{i_1}$ and
$\phi_{i_j}$ reduce $Z$ to $R_1$ and $R_j$, respectively.

\medskip

\emph{Strategy for $Q^k_X$:} We take a new $k$-tuple and $E_X$-restrain this
tuple.

\medskip

\emph{Strategy for $D^k_X$:} Wait for two unrestrained elements to be
enumerated into $W_k$. Then $E_X$-collapse these two elements.

\medskip
We organize these requirements in a finite priority argument. It is clear
that after each initialization each requirement acts at most once, so each
requirement is eventually not re-initialized, thus its final action satisfies
the requirement, and sets only a finite restraint. In particular no $E_X$ is
finite so that our wait for $P_{\vec{i}}$ eventually stops if all
$\phi_{i_r}$ are total and the range of $\phi_{i_1}$ contains infinitely many
$R_1$-equivalence classes. Since $E_X$ has infinitely many equivalence
classes, each $D^k_X$ is satisfied since if $W_k$ is infinite then it
eventually will enumerate a pair of numbers not restrained by higher priority
requirements: thus each $E_X$ is dark.

It follows that the $R_i$ are dark, they form a minimal $n$-tuple, and no
$n-1$-sub-tuple is a minimal tuple.
\end{proof}

\begin{obs}
There is a pair of incomparable ceers $A,B$ which do not form a minimal pair
so that if $(A,B,C)$ form a minimal triple, then $(A,C)$ and $(B,C)$ form a
minimal pair.
\end{obs}

\begin{proof}
Let $A,B$ be two strongly minimal covers of the same degree.
Then if
$(A,B,C)$ form a minimal triple, then every $X\leq A$ is also $\leq B$, so
cannot be $\leq C$, unless it is finite. This shows that $(A,C)$ form a
minimal pair. The same argument gives that $(B,C)$ form a minimal pair.
\end{proof}

\section{Definable classes of degrees of ceers}
A class $\mathcal{A}$ of degrees of ceers is \emph{definable in $(\Ceers,
\leq)$} if there is a first order formula $\phi(v)$ in the language of posets
such that
\[
\mathcal{A}=\{\mathbf{a}\mid (\Ceers, \leq)\models \phi(\mathbf{a})\}.
\]

For instance, by Theorem~\ref{LightMeets} the self-full degrees are definable
as exactly those degrees that are meet-irreducible.

\begin{cory}\label{cor:definability}
The classes of degrees provided by $\I$, $\{\Id\}$, $\Dark$, and $\Light$ are
all definable in the poset $(\Ceers, \leq)$.
\end{cory}

\begin{proof}
$X\in \I$ if and only if every ceer is $\leq$-comparable with $X$ and $X$ is
not universal.

$\Id$ is definable as the unique minimal degree above $\I$ which has a least
upper bound with all the other minimal degrees above $\I$. Indeed by Theorem
\ref{MinimalDark}, for every dark degree which is minimal above $\I$ there is
a minimal dark degree above $\I$ which is incomparable with it: since no two
dark ceers have a least upper bound by Theorem \ref{NoJoinOfDark}, the only
minimal ceers which have a join with all of the minimal ceers over $\I$ must
be light.  By Observation~\ref{joinIdDark} $\Id$ has a least upper bound with
every dark ceer, and is the only minimal light ceer.

$\Dark$ is definable as the ceers not in $\I$ which are not above $\Id$, and
$\Light$ is defined as the set of ceers above $\Id$.
\end{proof}

Next we notice:

\begin{cory}
The map $S:X\mapsto X\oplus \Id_1$ is definable in the structure $(\Ceers,
\leq)$.
\end{cory}

\begin{proof}
If $X$ is branching, then $X$ is not self-full by Theorem~\ref{LightMeets}
and thus $S(X)=X$ by Observation~\ref{obs:nonselffull}. Otherwise $X$ is
self-full (again by Theorem~\ref{LightMeets}) and thus by
Lemma~\ref{StrongMinimalCoversOfSF} $S(X)$ is the unique degree $Y$ so that
$Z>X$ implies $Z\geq Y$.
\end{proof}

As seen in Observation~\ref{joinIdDark}, for $X$ dark the ceer $X\oplus \Id$
is definable uniformly from $X$ as the smallest light ceer which bounds $X$.
No definability of $X\oplus \Id$  is known when $X$ is light, so we ask the
following question.

\begin{q}
Is the operation $X\mapsto X\oplus \Id$ definable?
\end{q}

\begin{q}
Is $R_K$ definable? Is $\Id'$ definable? Are there any definable degrees
other than those in $\I$ (notice that each $\Id_n$ is clearly definable, as
is the unique ceer with exactly $n-1$ ceers below it, if $n>0$, or is the
least ceer if $n=0$), that of $\Id$ or the universal degree?
\end{q}

\section{The poset of ceers modulo the dark ceers}
We define the relation $\leq_D$ on ceers, where we let $E\leq_D R $ if $E\leq
R\oplus X$, for some dark ceer $X$: this is clearly a pre-ordering relation,
hence a reducibility on ceers which originates the structure of the
\emph{$D$-degrees}. Note

\begin{obs}\label{obs:min-in-quotient-dark}
The $D$-degree of $\Id$ satisfies that it is nonzero in the structure
$\Ceers_{/\Dark}$ and every nonzero $D$-degree is $\ge_D$ it.
\end{obs}

\begin{proof}
The least element in $\Ceers_{/\Dark}$ is the $D$-equivalence class comprised
of all dark ceers, plus the finite ceers. On the other hand it can not be
$\Id \le_{D} E$ for any dark $E$, so $\Id_{/D}$ is nonzero and $\Id\le X$ for
any light $X$.
\end{proof}

\begin{obs}
We always have $E<_D E'$ for any non-universal ceer $E$.
\end{obs}

\begin{proof}
Let $E$ be non-universal. If $E' \le E \oplus X$ then by  uniform-join
irreducibility of $E'$ established in Fact~\ref{JumpsUniformJoinIrreducible}
we have that either $E' \leq E$, which can not be since $E$ is non-universal,
or $E' \leq X$, which implies that $X$ is not dark as $E'$ is light.
\end{proof}

\begin{obs}
The universal $\equiv_D$ class is comprised of exactly the universal ceers.
\end{obs}

\begin{proof}
If $E$ is $\leq_{D}$-universal, then for every universal $U$ we have $U \leq
E \oplus X$ for some dark $X$: but universal ceers are uniform-join
irreducible (each one being $\equiv$ to the jump of itself) and light, so
this gives $U \leq E$.
\end{proof}

\begin{thm}\label{thm:dark-two-elements}
Let $E$ be any ceer, and suppose $E\leq X< E\oplus \Id$. Then for some $n\in
\omega$, $X\leq E\oplus \Id_n$
\end{thm}

\begin{proof}
Let $f$ give a reduction of $E$ to $X$ and $g$ give a reduction of $X$ to
$E\oplus \Id$. Suppose towards a contradiction that the range of $g$ contains
infinitely many odd elements. We can assume without loss of generality that
the range of $g$ is onto the $\Id$-portion of $E\oplus \Id$: indeed, it is an
infinite c.e. set so we may assume it is everything by composing with a
computable $1-1$ map between this set and $\omega$. We analyze the situation
by cases:

Case 1: The image of $g\circ f$ in the $\Id$ portion of $E\oplus \Id$ is
finite. Let $Y$ be the portion of $\Id$ not in the range of $g\circ f$. Since
$Y$ is co-finite in the odd numbers, it is computable, and so we may consider
a computable bijection $k$ from the odd numbers onto $Y$. Now we describe a
reduction of $E\oplus \Id$ to $X$, leading to a contradiction: define
$h(z)=f(\frac{z}{2})$ if $z$ is even and $h(z)$ to be $g^{-1}(k(z))$ for $z$
odd.

Case 2: The image of $g\circ f$ in the $\Id$ portion of $E\oplus \Id$ is
infinite: by Lemma~\ref{coproduct3}(4), there is a ceer $E_0$ such that
$E \equiv E_0\oplus \Id$. But then $E\oplus \Id \equiv E_0\oplus \Id \oplus
\Id\equiv E_0 \oplus \Id\equiv E$, a contradiction.
\end{proof}

Since under $\equiv_D$ all dark and finite ceers collapse to the zero
$D$-degree, how does the structure $\Ceers_{/\Dark}$ compare to $\Light$?
Differences between the two structures are shown in the following theorem.

\begin{thm}
$\Ceers_{/\Dark} \not \equiv_{\leq, \oplus} \Light \cup \{\mathbf{0}\}$
(i.e. the two structures
are not elementarily equivalent in the language with $\leq$ and $\oplus$).
\end{thm}

\begin{proof}
By Theorem~\ref{thm:dark-two-elements} $\Ceers_{/\Dark}$ satisfies the
statement $(\forall X)(\vert [X,X\oplus \Id]\vert\leq 2$, as $E\leq_{\I} R$
implies $E \leq_{D} R$. On the other hand this statement is false in the
light degrees: if $X$ is light and self-full then by
Lemma~\ref{obs:nonselffull} and Lemma~\ref{StrongMinimalCoversOfSF} the
interval $[X, X\oplus \Id]$  contains the infinite chain
\[
X< X\oplus \Id_{1}<  X \oplus \Id_{2} <\cdots.
\]
Moreover, $\Id$ is definable in $\Light$ as the least element, and $\Id$ is
definable in $\Ceers_{/\Dark}$ by Observation~\ref{obs:min-in-quotient-dark}.
\end{proof}

\begin{thm}
The first order theory of the structure $(\Light, \leq)$ is undecidable.
\end{thm}

\begin{proof}
The argument in \cite{ceers} showing that the theory of $\Ceers$ is
undecidable is based on the fact that one can embed into $\Ceers$ the
interval of c.e. $1$-degrees $[\mathbf{0}_1, \mathbf{0}_1']$ where
$\mathbf{0}_1$ is the $1$-degree of any infinite and coinfinite decidable set
and $\mathbf{0}_1'$ is the $1$-degree of the halting set $K$ (this embedding
is granted by Fact~\ref{fact:fund-ce-sets}).  We just note that this
embedding is all happening into the light ceers.
\end{proof}

\begin{obs}\label{obs:darkIindark}
For every dark ceer $X$, the $\equiv_\I$-class of $X$ is uniformly definable
in $(\Ceers, \leq)$, in the parameter $X$.
\end{obs}

\begin{proof}
Let $X$ be dark. By Theorem~\ref{thm:characterization-of-equiI} a ceer $Y$ is
$\equiv_\I X$ if and only if for some $k$, either $X\oplus \Id_k\equiv Y$ or
$Y\oplus \Id_k \equiv X$. Since the second condition is symmetric to the
first with the variables $X$ and $Y$ swapped (and $Y$ is dark), we need only
show that the first condition ``$X\oplus \Id_k\equiv Y$ for some $k$'' is
definable. We claim this holds if and only if
\[
X\leq Y \;\&\; [X,Y] \textrm{ is lineraly ordered }\;\&\;
(\forall Z) [X\leq Z \rightarrow [Z> Y \vee Z \in [X,Y]]].
\]
Indeed, if $X\oplus \Id_k\equiv Y$ for some $k$, then  $[X,Y]$ is comprised
of degrees containing ceers of the form $X\oplus \Id_l$; and by Lemma
\ref{StrongMinimalCoversOfSF}, every degree $\geq X$ is either equivalent to
$X\oplus \Id_l$ for some $l\leq k$ or is $\geq X\oplus \Id_{k+1}> X\oplus
\Id_{k}$.

The right-to-left direction comes from
Corollary~\ref{cory:selfdarkminimalcovers}. Let $Y \cancel{\equiv}_\I X$: we
claim that $Y$ does not satisfy the condition. So, assume that $X \leq Y$,
hence $X<Y$. Let $(A_i)_{i\in \omega}$ be the collection of incomparable dark
ceers above $X$ granted by Corollary~\ref{cory:selfdarkminimalcovers}: then
for every $i,n$, $X \oplus \Id_n \leq  A_i$ and if $Y< A_i$ then there exists
$m$ such that $Y \le X \oplus \Id_m$. This shows that for no $i$ can it be
$Y< A_i$ otherwise $Y \le X \oplus \Id_m$ for some $m$ contrary to
assumption. So, if one of these $A_i$ is not bounded by $Y$, then $Z=A_i$
shows that the condition does not hold. On the other hand, if $Y$ does bound
all of these $A_i$, then any two of them, by incomparability, show that the
condition does not hold since $[X,Y]$ is not linearly ordered.
\end{proof}

Despite the fact that the dark ceers form a discrete partial order under the
$\oplus \Id_1$-relation, this cannot help in showing that their first order
theory is decidable. This is because the structure of $\Dark_{\I}$ is
definable in the dark degrees as Observation~\ref{obs:darkIindark} shows.

\begin{q}
Is the theory of dark ceers decidable?
\end{q}

\section{$\mathbb{Z}$-chains}
$\mathbb{Z}$-chains will be useful when discussing automorphisms of ceers in
Section~\ref{sct:automorphisms}.

\begin{defn}\label{def:chains}
A \emph{$\mathbb{Z}$-chain under the map $X\mapsto X\oplus \Id_1$} (or simply
\emph{under $\oplus \Id_1$}) is a collection $\zeta$ of degrees such that
there is an order theoretic isomorphism with the structure $\langle
\mathbb{Z}, S \rangle$, with $\mathbb{Z}$ the integers, $S$ the successor in
$\mathbb{Z}$, and under the isomorphism $S$ corresponds to the map $X\mapsto
X\oplus \Id_1$.
\end{defn}

\begin{obs}
Let $R$ be a self-full ceer. Then the degree of $R$ is in the range of the
$\oplus \Id_1$ map if and only if some class in $R$ is computable.
\end{obs}

\begin{proof}
If $R$ has a computable class $[x]_R$, then by Lemma
\ref{ComputableCollapses}, $R\equiv R/(x,y)\oplus \Id_1$, for $y$ any element
inequivalent to $x$. On the other hand, every ceer of the form $E\oplus
\Id_1$ has a computable class. Now suppose $R$ is equivalent to a ceer
$E\oplus \Id_1$. Then since $R$ is self-full and $R\leq E\oplus \Id_1\leq R$,
we see that every class of $E\oplus \Id_1$ must be in the range of this
reduction (otherwise $R$ properly reduces to itself contradicting
self-fullness). Then the class sent to the class $\Id_1$ is a computable
class in $R$.
\end{proof}

Notice that every non-self-full ceer $E$ is in the range of the $\oplus
\Id_1$ map, since $E\oplus \Id_1\equiv E$.

\begin{cory}
There are dark degrees which are not part of $\mathbb{Z}$-chains under
$\oplus \Id_1$. There are also dark degrees which are part of
$\mathbb{Z}$-chains under $\oplus \Id_1$.
\end{cory}

\begin{proof}
There are dark ceers which have no computable classes. For example, Badaev
and Sorbi~\cite{Badaev-Sorbi} have constructed dark weakly precomplete ceers.

On the other hand, there are dark ceers with finite classes, such as those
constructed in Theorem~\ref{DarkCeersJoinInI}. Such degrees are part of
$\mathbb{Z}$-chains under $\oplus \Id_1$. Indeed, if $X$ is dark then the
successor  $X\oplus \Id_1$ (under the map $\oplus \Id_1$) by
Lemma~\ref{StrongMinimalCoversOfSF} is a strong minimal cover of $X$; on the
other hand, if $X$ is a dark ceer with finite or computable classes, then by
collapsing two of them as in Lemma~\ref{ComputableCollapses}, we still get a
dark ceer which is a predecessor of it under the map $X\mapsto X\oplus \Id_1$
and by Lemma~\ref{StrongMinimalCoversOfSF} $X$ is a strong minimal cover of
it.  The result then follows by repeatedly applying Lemma
\ref{ComputableCollapses}. This procedure is illustrated in
Figure~\ref{fig:chains3}.
\end{proof}

The proof of the previous corollary shows in fact how to build a
$\mathbb{Z}$-chain around  a self-full $X$ with infinitely many computable
classes. Let us start with $E_0=X$; having defined $E_n$ for $n\ge 0$ then
let $E_{n+1}=E_n \oplus \Id_1$; having defined $E_{-n}$ with $n \ge 0$, so
that $E_{-n}$ has computable classes (in fact $E_{-n}$ of the form
$E_{-n}=X_{/\{(x_i,y_i)\mid 1\le i\le n\}}$ so it has infinitely many
computable classes) then take an $X$-inequivalent pair $(x_{n+1}, y_{n+1})$
so that $[x_{n+1}]_X$ is computable, $x_{n+1}, y_{n+1}$ pairwise
$X$-inequivalent to all $x_i$ and $y_i$, and define
$E_{-(n+1)}=X_{/\{(x_i,y_i)\mid 1\le i\le
n+1\}}=(E_{-n})_{/(x_{n+1},y_{n+1})}$: by Lemma~\ref{ComputableCollapses}, we
have that $E_{-n}\equiv E_{-(n+1)} \oplus \Id_1$; moreover $E_{-(n+1)}$ has
infinitely many computable classes.

\begin{figure}[h!]
  \centering
  \includegraphics[scale=.8]{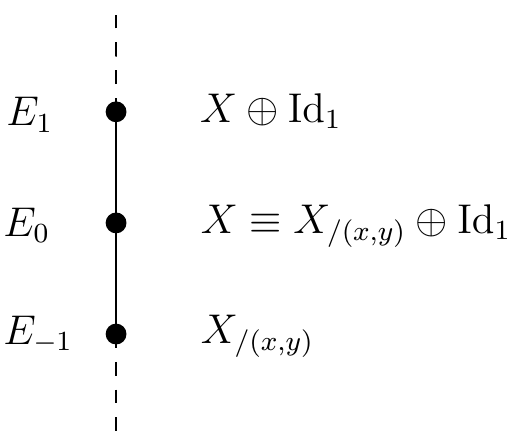}
  \caption{Building a $\mathbb{Z}$ chain around a self-full $X$ with
  infinitely many computable classes.}\label{fig:chains3}
\end{figure}

\section{Automorphisms of ceers}\label{sct:automorphisms}
The following observation is the key tool for constructing automorphisms of
$\Ceers$. $\mathbb{Z}$-chains under $\oplus \Id_1$ were defined in
Definition~\ref{def:chains}. In the following, when writing $F(X)$ for
an automorphism $F$ on $\Ceers$, we mean of course $F$ applied to
the degree of $X$.

\begin{thm}\label{AutsFromChains}
From countably many $\mathbb{Z}$-chains under $\oplus \Id_1$, one can build
continuum many automorphisms of the structure of ceers.
\end{thm}

\begin{proof}
Let $\zeta$ be any $\mathbb{Z}$-chain of ceers under the operation $\oplus
\Id_1$. For any $n\in \omega$, consider the function $F$ on degrees of ceers
which sends $X\notin \zeta$ to $X$ and sends $X\in \zeta$ to $X\oplus \Id_n$.
Then $F$ is an automorphism of $(\Ceers, \leq)$. This is because $X\leq Y$
implies $X\oplus \Id_1 \leq Y$ or $X\equiv Y$. The inverse of this map sends
$X$ to a ceer $Y$ so that $Y\oplus \Id_n\equiv X$. Now, for any collection of
countably many $\mathbb{Z}$-chains, this gives uncountably many
automorphisms, as we can move each of these chains up or down by various
$n$'s independently.
\end{proof}

\begin{cory}\label{AutOverLight}
There are continuum many automorphisms of the structure of ceers fixing the
light degrees.
\end{cory}

\begin{proof}
Simply take countably many incomparable dark ceers with infinitely many
computable classes: for instance by Lemma~\ref{lem:Turing-simple} take
infinitely many $\le_1$-incomparable simple sets $(X_i)_{i \in \omega}$ and
take the corresponding $R_{X_i}$. This (by the procedure illustrated in
Figure~\ref{fig:chains3}) gives countably many $\mathbb{Z}$-chains to yield
uncountably many automorphisms where every light degree is fixed.
\end{proof}

\begin{cory}
There are continuum many automorphisms of the structure of ceers fixing
the dark degrees.
\end{cory}

\begin{proof}
Take infinitely many incomparable light self-full degrees as in
Corollary~\ref{LightSelfFulls}, which have finite classes. This (by the
procedure illustrated in Figure~\ref{fig:chains3}) gives countably many
$\mathbb{Z}$-chains to yield uncountably many automorphisms where every dark
degree is fixed.
\end{proof}

\begin{obs}
If $\tau$ is an automorphism of ceers fixing the light ceers, then $\tau(X)
\equiv_\I X$ for all $X$.
\end{obs}

\begin{proof}
Suppose $\tau$ is an automorphism that fixes the light ceers, and suppose
that $\tau(X)\not\equiv_\I X$. We may assume $X\not\geq_\I \tau(X)$
(otherwise, we can consider $\tau^{-1}$). It follows that $X$ is dark, as
every automorphism clearly fixes the finite ceers. Then $X\oplus \Id$ is a
light ceer which bounds $X$ and therefore it bounds $\tau(X)$ as well. But
$\tau(X)$ is dark, and if $Y \le E \oplus \Id$ for a dark $Y$ then there
exists $k$ such that $Y\le E \oplus \Id_{k}$, otherwise the inverse image of
the odd numbers under a reduction would give an infinite c.e. set of
non-$Y$-equivalent numbers. Then $\tau(X) \leq_\I X$, contrary to the
assumptions.
\end{proof}

Thus any automorphism fixing the light ceers must send each $\equiv_\I$-class
to itself. The $\equiv_\I$-classes of ceers which form a $\mathbb{N}$-chain
under $<$ must be fixed, so if you consider all $\mathbb{Z}$-chains of dark
ceers, then the automorphisms built in Theorem \ref{AutsFromChains} (which
move each $\mathbb{Z}$-chain into itself) is the entire group
$\text{Aut}_{\Light}(\Ceers)$ of the automorphisms that fix $\Light$.

\begin{q}
If $\tau$ is an automorphism of ceers which fixes every dark ceer, then must
$\tau(X)\equiv_\I X$ for every $X$?
\end{q}

\begin{obs}
	The self-full ceers form an automorphism base for the structure of ceers.
\end{obs}
\begin{proof}
	By Corollary 7.11, every non-self-full ceer $X$ has a self-full strong minimal cover $Y$. Thus, if $\tau$ is an automorphism of the structure of ceers, $\tau(Y)$ is a strong minimal cover of $\tau(X)$, and thus $\tau(X)$ is determined by $\tau(Y)$.
\end{proof}

\subsection{A remark on automorphisms of the c.e. $1$-degrees}
The same argument yields that there are continuum many automorphisms of the
structure of c.e. sets under $1$-reduction which fix all non-simple sets. For
any simple set, we have that $X+1$ is a strong minimal cover of $X$ and
$Y>_1X$ implies $Y\geq_1 X+1$. The predecessor can be found here by
re-ordering the c.e. set so that $0\notin X$ then using $X-1$ as the
predecessor. So each simple set is part of a $\mathbb{Z}$-chain, and we can
shift these independently.

\subsection{The non-definability of the jump}
Fix any ceer $X$ which is part of a $\mathbb{Z}$-chain, and let $\sigma$ be a
non-trivial automorphism of $\Ceers$ (constructed as in the proof of
Theorem~\ref{AutsFromChains}) which fixes everything outside this
$\mathbb{Z}$-chain. Then since $X<_\I X'$ for every non-universal $X$ by
Observation \ref{jumpUpInI}, $X'$ is fixed by $\sigma$, though $X$ is moved.
But $Y<Z$ if and only if $Y'<Z'$ (\cite{Gao-Gerdes}), so $\sigma(X)'\neq
X'=\sigma(X')$. Thus the jump cannot be definable in the structure of
$\Ceers$.

\section{Index sets}
We conclude by computing the complexity of the index sets of the classes of
ceers studied in the previous sections. For more about index sets of classes
of ceers see \cite{Andrews-Sorbi:index-sets}.

\begin{obs}
The following hold:
\begin{enumerate}
  \item The index set $\{x: R_x \textrm{ light}\}$ is
      $\Sigma^0_3$-complete;
  \item the index set $\{x: R_x \textrm{ dark}\}$ is $\Pi^0_3$-complete;
  \item the index set of the ceers in $\mathcal{I}$ is
      $\Sigma^0_3$-complete.
\end{enumerate}
\end{obs}

\begin{proof}
It is straightforward that each of these sets are in the corresponding
complexity classes, so we will verify completeness. It is proved in
\cite{Andrews-Sorbi:index-sets} that if $R$ is any ceer with infinitely many
classes then $(\Sigma^0_3, \Pi^0_3)\leq_1 (\{i\mid R_i \equiv R\}, \{i\mid
R_i\not\leq R\, \& \, R_i\not \geq R\})$ (where $(\Sigma^0_3,\Pi^0_3)
\leq_{1} (A,B)$ means that for every $\Sigma^{0}_{3}$ set $C$, there is a
computable function which $1$-reduces $C$ to $A$, and the complement of $C$
to $B$.

Thus if we take $R=\Id$ we immediately get that $\{x: R_x \textrm{ light}\}$
is $\Sigma^0_3$-complete. On the other hand, every $\Pi^0_3$ set is
$1$-reducible to $\{i\mid R_i\not\leq \Id\, \& \, R_i\not \geq \Id\})$, which
is exactly the index set of the ceers in $\Dark$.

Finally, to show the claim about $\mathcal{I}$, for every $x$ let $E_x$ be
the ceer so that $u \rel{E_x} v$ if
\[
u=v \lor [u<v \,\&\,[u,v] \subseteq W_x]
\]
where $[u,v]=\{z\mid u\le z \le v\}$. By the $s$-$m$-$n$ Theorem let $f$ be a
$1$-$1$ computable function such that $E_x=R_{f(x)}$. Let
$\textrm{Cof}=\{x\mid W_x \textrm{ cofinite}\}$. Clearly $x \in \textrm{Cof}$
if and only if $R_{f(x)} \in \mathcal{I}$. The result then follows from the
fact that $\textrm{Cof}$ is $\Sigma^0_3$-complete.
\end{proof}

A slightly more complicated argument is needed to compute the complexity of
the index set of the self-full ceers.

\begin{thm}
The index set $\{i\mid R_i \text{ is self-full}\}$ is $\Pi^0_3$-complete.
\end{thm}

\begin{proof}
(Sketch) By Observation~\ref{obs:nonselffull}, self-fullness for a ceer $E$
is equivalent to the following $\Pi^0_3$-condition:
\[
(\forall i)[\phi_i \text{ non-total} \vee (\exists x,x') [x \rel{E\oplus \Id_1}
x' \not\Leftrightarrow \phi_i(x)\rel{E} \phi_i(x')]].
\]

To show $\Pi^0_3$-completeness, given a c.e. set $W$, we uniformly construct
a ceer $E$ so that $E$ is self-full if and only if every column of $W$ is
finite: we rely on the fact that every $\Pi^0_3$  set $A$ can be expressed by
a relation $(\forall k)R(x,k)$ where $R\in \Sigma^0_2$ and thus reducible to
$\{e\mid W_e \textrm{ finite}\}$, so that there is a computable function $g$
with $A=\{x\mid (\forall k)[W_{g(x,k)} \textrm{ finite}]\}= \{x\mid (\forall
k)[V_{x}^{[k]} \textrm{ finite}]\}$, where $V_{x}=\{\langle k, y\rangle \mid
y \in W_{g(x,k)}\}$ and an index for $V_{x}$ can be found uniformly from $x$.
If $E_{x}=R_{h(x)}$ is the self-full ceer constructed from $V_{x}$ then for
every $x$, we have that  $x \in A$ if and only if $R_{h(x)}$ is self-full.

\smallskip

To build $E$ given $W$ we have requirements:
\begin{itemize}
\item[$SF_{ij}$:] If $W_i$ intersects infinitely many classes, then it
    intersects $[j]_E$.
\item[$NSF_k$:] If $W^{[k]}$ is infinite, then $f$ is a reduction of $E$ to
    itself missing some class ($f$ is a computable function produced by the
    requirement).
\end{itemize}
These requirements are arranged in a computable priority ordering of order
type $\omega$. Each requirement may restrain numbers in $E$ to keep them
$E$-inequivalent.

\smallskip

\emph{Strategy for $SF_{ij}$:} If $W_i$ and $[j]_E$ are still disjoint, then
wait for $W_i$ to enumerate a number $x$ such that $x,j$ are not restrained
by higher priority requirements, and (\emph{action}) $E$-collapse $x$ to $j$.

\smallskip

\emph{Strategy for $NSF_k$:} To define $f$, code $E$ on a column
$\omega^{[n]}$ of $E$ having chosen a new $n$ so that when we choose it for
the first time every number in the column is new. Restrain all elements in
$\im(f)$ to protect the coding so that no lower priority requirement is
allowed to $E$-collapse equivalence classes of elements of $\im(f)$, also
guaranteeing that $f$ is not onto the classes of $E$ by appointing yet a new
number $a\notin \omega^{[n]}$ so that its equivalence class does not contain
any element equivalent to a number in the range of $f$: this can be achieved
by restraining $a$ and $\im(f)$, i.e. by requesting that no lower priority
requirement can $E$-collapse $a$ to any number in $\im(f)$. The numbers $n$
and $a$, and the coding function $f$ are the \emph{parameters of $NSF_k$},
which are started anew any time we re-initialize the requirement.

\smallskip Satisfaction of all these requirements give the desired $E$.
Self-fullness follows from satisfying $SF_{ij}$ for all $j$, where
$W_i=\range(f)$ if $f$ is any reduction of $E$ to $E$.

\smallskip

Again, we employ the priority machinery (initialization, requiring attention,
etc.) to build $E$ according to these strategies. $NSF_k$ \emph{requires
attention at stage $s$} if $W^{[k]}$ grows at that stage; if so, and is the
least such that this happens, $NSF_k$ builds $f$ as follows: firstly, it
considers all $x,y$ for which $f(x), f(y)$ have been already defined, and
$E$-collapses $f(x)$ and $f(y)$ if already $x \rel{E} y$; secondly, it takes
the first $z$ such that $f(z)$ is still undefined, and sets $f(z)=\langle n,
w\rangle$, where $[w]_E$ is still $\{w\}$.

$SF_{ij}$ \emph{requires attention} if it is ready to act as indicated (i.e.
$W_i$ and $[j]_E$ are still  disjointand $W_i$ has enumerated an $x$ so that
the pair $x,j$ is not restrained by higher priority requirements): if it acts
then it is forever satisfied and does not need to receive attention any more.

Every time a requirement acts, it must re-initialize all lower-priority
requirements. If $W$ has an infinite column and $W^{[k]}$ is the first column
of $W$ which is infinite, then $NSF_k$ can be injured and re-initialized due
to collapses deriving from finitely many higher priority
$SF_{ij}$-requirements, but after that, it will succeed in picking the final
columns $\omega^{[n]}$ and $a$, and build the reduction $f$, as it requires
attention infinitely often and is not re-initialized any more. Note that all
requirements having lower priority than $NSF_k$ are in this case
re-initialized infinitely often, but after all none of them needs to be
satisfied.

On the other hand, if every column $W^{[k]}$ is finite, then every $NSF_k$
places finite restraint, and after last re-initialization every $SF_{ij}$
will succeed by either waiting forever for a number $x$ so that $x,j$ are not
restrained (which can only be the case when $W_i$ hits at most finitely many
equivalence classes), or by acting once for all.

So, $E$ is self-full if and only if every column of $W$ is finite. Uniformity
clearly holds, as an index for $E$ can be effectively found starting from any
c.e. index of $W$.
\end{proof}


\end{document}